\documentclass{asl}
\usepackage{graphicx,amsfonts,bbm,proof,amssymb,stmaryrd,arydshln,caption}
\newtheorem{thm}{Theorem}[section]
\newtheorem{cor}[thm]{Corollary}
\newtheorem{lem}[thm]{Lemma}
\theoremstyle{definition}
\newtheorem{defn}[thm]{Definition}
\newtheorem{ex}[thm]{Example}

\usepackage{FLe_Und_Macros} 

\title{Most simple extensions of $\mathsf{FL_e}$ are undecidable}
\author{Nikolaos Galatos}
\revauthor{Galatos, Nikolaos}
\address{Department of Mathematics\\ University of Denver\\
Denver, CO, 80210, USA}
\email{Nikolaos.Galatos@du.edu}
\author{Gavin St.\,John}
\revauthor{St.\,John, Gavin} 
\address{Department of Mathematics\\ University of Denver\\
Denver, CO, 80210, USA}
\email{Gavin.StJohn@gmail.com}

\begin{document}

\begin{abstract}
All known structural extensions of the substructural logic $\m {FL_e}$, Full Lambek calculus with exchange/commutativity, (corresponding to subvarieties of commutative residuated lattices axiomatized by $\{\jn, \cdot, 1\}$-equations) have decidable theoremhood; in particular all the ones defined by knotted axioms enjoy strong decidability properties (such as the finite embeddability property). We provide infinitely many such extensions that have undecidable theoremhood, by encoding machines with undecidable halting problem. An even bigger class of extensions is shown to have undecidable deducibility problem (the corresponding varieties of residuated lattices have undecidable word problem); actually with very few exceptions, such as the knotted axioms and the other prespinal axioms, we prove that undecidability is ubiquitous. Known undecidability results for non-commutative extensions use an encoding that fails in the presence of commutativity, so and-branching counter machines are employed. Even these machines provide encodings that fail to capture proper extensions of commutativity, therefore we introduce a new variant that works on an exponential scale. The correctness of the encoding is established by employing the theory of residuated frames.
\end{abstract}

\maketitle

%
%SECTION: Introduction
%
\section{Introduction} 
Substructural logics are defined as extensions of the Full Lambek calculus $\m {FL}$ and include among others classical, intuitionistic, linear, relevance, bunched-implication and many-valued logics. They find applications to areas as diverse as mathematical linguistics, philosophy, management of pointers in computer architecture, engineering, theoretical physics and functional programming. Their algebraic semantics, in the sense of Blok and Pigozzi \cite{BP}, are (pointed) residuated lattices (or FL-algebras) and they have an independent history with connections to classical and to ordered algebra. In particular they include  the lattice of ideals of rings, lattice-order groups, algebras of relations, and of course Boolean and Heyting algebras. Pointed residuated lattices form a variety $\sf{FL}$ and its subvarieties  correspond to extensions of $\m {FL}$ via a dual lattice isomorphism; algebraic and logical properties are tightly linked. See \cite{GJKO} for an introduction to the area.

Decidability questions are at the core of the study of logical systems and here we explore logics/varieties with structure rich enough to allow for encoding the computation of machines with undecidable halting problem. This yields undecidability results for the word problem, and hence also for the quasiequational theory, of these varieties (namely the deducibility relation for the logics) and sometimes even the undecidability of the equational theory of the varieties (i.e., the theoremhood in the corresponding logics). 
 
The equational theory of $\sf{FL}$ is decidable and the same holds for many of its standard extensions such as $\sf{FL_e}$ (with exchange/commutativity: $xy=yx$), $\sf{FL_w}$ (with weakening/integrality: $x \leq 1$),  $\sf{FL_{ei}}$ (with exchange and weakening), $\sf{FL_{ec}}$ (with exchange and contraction: $x \leq x^2$), $\sf{FL_{em}}$ (with exchange and mingle: $x^2 \leq x$).  The equational theory of $\sf{FL_{c}}$ is one of the few known to be undecidable \cite{ChvHor}; a precursor to this result is the fact that the equational theory of $\mathsf{FL_{ec}}$, even though decidable, is not primitive recursive \cite{Urq}. The only other known subvarieties of $\sf{FL}$ with undecidable equational theory are the ones axiomatized by $x^m \leq x^n$ (where $0<m<n$), the one axiomatized by the modular law, and the one axiomatized by commutativity, involutivity and distributivity (corresponding to the relevance logic R). In particular,  the last one is the only subvariety of $\sf{FL_e}$ with undecidable equational theory; actually distributivity does not correspond to a sequent structural rule, unless the syntax is expanded, so it is not even a structural extension of $\m {FL}$.  
On the contrary, prominent subvarieties of $\sf{FL_e}$, such as (proper, non-trivial) subvarieties axiomatized by any knotted inequality $x^m \leq x^n$ (where $m\neq n$) not only have a decidable equational theory but also a decidable quasiequational theory, and even the finite embeddability property \cite{vanA}. (In \cite{Card} it is further shown that this remains true even under conditions weaker than commutativity.)

In contrast to the above, in this paper we construct infinitely many subvarieties of $\sf{FL_e}$ with undecidable equational theory. We also show that an even bigger collection of subvarieties of $\sf{FL_e}$ have an undecidable quasiequational theory, actually undecidable word problem. The encoding used for the undecidability of the word problem for $\sf{FL_e}$ does not work for its subvarieties, so we modify it in a novel way, by storing the values in the counter machines as powers of a sufficiently large constant, which depends on the subvariety. 

 A \textit{residuated lattice} $\mathbf{R}=( R, \vee,\land,\cdot,\backslash,/,1 )$ is an algebraic structure such that $(R,\vee,\land) $ is a lattice, $(R,\cdot, 1)$ is a monoid, and the \textit{law of residuation} holds: for all $x,y,z\in R$,
$$x\cdot y \leq z \mbox{ iff } x\leq z/y \mbox{ iff } y\leq x\backslash z,$$
where $\leq$ is the induced lattice order. The residuated lattice $\mathbf{R}$ is called \textit{commutative} if $(R,\cdot,1)$ is a commutative monoid; in such a case $x\backslash y = y/ x$ for all $x,y\in R$ and will use the notation $x\to y:=x \backslash y$. It is well known that (commutative) residuated lattices form a variety denoted $(\Cc )\RL $, see \cite{GJKO}. FL-algebras are defined as expansions of  residuated lattices by an arbitrary constant $0$ (which is used to define the negation operation), but we will not be making use of this constant in our encodings and our results remain true in the presence or absence of this constant.

For example, the extension of ($0$-free) $\m{FL_e}$ with the structural rule
\begin{equation*}\label{d}\infer[{(\varepsilon)}] {\Gamma, \Delta,\Sigma \vdash \Pi }{\Gamma,\Delta,\Delta,\Sigma\vdash\Pi &\Gamma,\Delta,\Delta,\Delta,\Sigma\vdash\Pi} \end{equation*}
corresponds to the subvariety of $\CRL$ defined by the inequality $\varepsilon: x\leq x^2\vee x^3$.  More generally, structural rules, such as the above, correspond to inequalities in the signature $\{\vee,\cdot,1\}$. We will show that theoremhood for $\m{FL_e}+(\varepsilon)$ and the equational theory of $\CRL +\varepsilon$ are undecidable.

To establish the main undecidability results, we encode in the theory of commutative residuated lattices the computation of machines with undecidable halting problem; these are \textit{and-branching counter machines}, a variant of counter machines, introduced in \cite{Lincoln}. In \S3 we outline their structure and in \S4 we use them to give a simple and direct proof of the undecidability of every variety between $\CRL$ and $\RL$, extending the result that was known for just these two varieties. This is done to set the stage for the much more complicated development that follows, while still introducing two of the main tools: abstract machines and residuated frames. The theory of residuated frames, developed in \cite{GJ}, is used to prove the completeness of the encoding, inspired by \cite{Hor,ChvHor}. The encodings used in the latter and the standard encodings, however, do not work in the presence of commutativity and this is why we make use of the one considered in \cite{Lincoln}; in the Conclusions section \S9 we compare some of the encodings and discuss further some related results. 

In the beginning of \S5 we review the algebraic counterparts of structural rules: equalities in the signature $\{\vee,\cdot,1\}$, as alluded to above, and show how they can be viewed as conjunctions of simple inequalities. We also explain why the encoding of and-branching counter machines fails to work for subvarieties of $\CRL$ axiomatized by such equations, thus necessitating the use of a novel encoding. Parts of \S5 and \S6 explore properties of the new encoding that are needed for capturing the additional equations: the application of the equation even though it may interfere with the computation of the new machine, should not add any more accepted configurations. We refer to this condition as \emph{admissibility} of the equation relative to the machine and also we work out some motivating examples. 

In \S7 we define the new machines as variants of and-branching counter machines working on an exponential level; the main idea is that the new equation does affect the computation but only in a linear/polynomial way, so if the encoding is done at the exponential level no new accepted configurations will be added. The desired properties of these exponential machines are studied in detail in \S7 and a technical condition $(\star \star)$ emerges as a prominent condition for the given equation so that the encoding will work. In the beginning of \S8 it becomes clear that if the added equation fails condition $(\star\star)$, not even the new exponential encoding will work to establish undecidability and thus such equations are outside the scope of this paper. However, in  \S7 it is shown that if the equation satisfies condition $(\star \star)$, then it indeed defines a variety of residuated lattices with undecidable word problem. Actually, Theorem~\ref{th:1-var} shows that almost all 1-variable equations actually satisfy condition $(\star \star)$, indicating that this condition is not really restrictive. In parallel to all this, starting already from \S5 we introduce a very big class of equations, called \emph{spineless}; the class is so big that we define it via its complement (prespinal equations), which forms a very small portion of all equations and admits a very natural and simple definition. 

In \S8 we show that, surprisingly, the very transparently defined spineless equations are precisely the ones that satisfy the technical condition $(\star \star)$. Therefore, the results of \S7 and \S8 together imply that every spineless equation defines a variety with undecidable word problem (Corollary~\ref{UndLink}).
Wanting to showcase this result as early as possible, we (re)state it in advance in \S5 and use it to derive the stronger result of the undecidability of the equational theory for certain subvarieties; this is done by using a special form of the deduction theorem, relying on the characterization of congruences in commutative and expansive residuated lattices. In that sense, someone who reads just the first third of the paper, up to \S5, has a full and clear grasp of all the notions and also knows the statements of the two main results of the paper, Theorem~\ref{main} and Theorem~\ref{main2}.

The proof, given in \S8, that these two differently-looking notions coincide is quite involved and relies on positive linear algebra, therefore \S8 uses very different tools than the ones of the rest of the paper. We even find it helpful to move from the multiplicative notation for monoids, used in most of the paper and which suits the link to the encoding of the machines, to additive notation for monoids, giving rise to positive linear transformations and matrices with natural numbers as entries. Therefore we chose to have this proof done as the last section of the paper, as it already takes up one-fourth of the length of the paper.  
%
%SECTION: Preliminaries
%
\section{Preliminaries}
We denote the sets of natural numbers, positive integers, and real numbers by $\N$, $\Z^+$, and $\R$, respectively.   
We denote the \textit{powerset} of a set $X$ by $\powerset(X)$. Given a set $X$ and a binary operation symbol $\cdot$, we denote by $(X^*,\cdot,1)$ the free commutative monoid generated by $X$ with unit $1$. A {\em substitution} on $X$ is a monoid homomorphism $\subt:X^*\to X^*$; substitutions are determined by their restriction to $X$. For $x\in X^*$, we write $x^n$ to denote $1$, if $n=0$, and the term $x\cdot\cdots \cdot x$ consisting of $n$ copies of $x$ for $n>0$. For subsets $A,B$ of $X$, we define $A\cdot B=\{a\cdot b: a\in A,b\in B\}$, and if $a\in X$ then $a\cdot B=\{a\}\cdot B$. 

Let $\mathcal{L}$ be an algebraic language, i.e., containing no relational symbols. Given a set of variables $X$, $T(X)$ denotes the set of terms over $X$ and $\mathcal{L}$ and $\m{T}(X)$ the absolutely free algebra of terms. A \textit{quasiequation} is (the universal closure of) a formula of the form
\begin{equation}\label{qe}s_1=t_1~ \& ~\dots ~\&~ s_n=t_n \implies s_0=t_0, \end{equation}
where $s_0,t_0,...,s_n,t_n\in T(X)$ are terms and $n \in \N$. If $n=0$ then the left-hand side is empty and we obtain an equation.

For a class of algebras $\mathcal{K}$ in the language $\mathcal{L}$, we say that (\ref{qe}) holds in $\mathcal{K}$ (i.e. $\mathcal{K}\models (\ref{qe})$) if for every algebra $\mathbf{A}\in \mathcal{K}$ and homomorphism $h:\mathbf{T}(X)\to \mathbf{A},$
$$(\forall i\in \nset{n}) (\mathbf{A},h\models s_i=t_i ) \implies \mathbf{A},h\models s_0=t_0.$$
Here $\mathbf{A}, h \models s=t$ means $h(s)=h(t)$.

A \textit{presentation} is a pair $\langle X,E\rangle$ where $X$ is a set of generators and $E$ is a set of equations over $T(X)$. A presentation $\langle X,E\rangle $ is said to be finite iff both $X$ and $E$ are finite. We denote the conjunction of equations in $E$ by $\&E$. For a variety of algebras $\mathcal{V}$ in the language $\mathcal{L}$, we say \textit{$\mathcal{V}$ has an undecidable word problem} if there exists a finite presentation $\langle X,E\rangle$ such that there is no algorithm, which on input $s,t\in T(X)$ decides whether the quasiequation
\begin{equation}\label{wordproblem}\& E \implies s=t \end{equation}
holds in $\mathcal{V}$. Note that if $\mathcal{V}$ has undecidable word problem then its quasiequational and universal theories are undecidable as well. The word problem is also referred to as the local word problem and the quasiequational theory as the global word problem, making explicit that in the former the antecedent of the quasiequations is fixed while in the latter it is unrestricted. 

In the following the proofs of the undecidability of the word problem will rely on the fact that residuated lattices have a $\{\jn, \cdot, 1\}$-reduct, and therefore the undecidability results apply to all reducts that contain it.
%
%SECTION: Counter Machines
%
\section{Counter Machines}
For proving undecidability we use a type of abstract machine known as an \textit{And-branching $k$-Counter Machine} ($k$-ACM), introduced in \cite{Lincoln}, as they have an undecidable halting problem. A $k$-ACM is a tuple ${\acm}=(\Reg_k, {\State},{\Inst}, q_f)$ representing a type of parallel-computing counter machine, where 
\begin{itemize}
	\item $\Reg_k:=\{\reg_1,...,\reg_k\}$ is a set of $k$ \textit{registers}, each able to store a nonnegative integer (representing the number of tokens in that register), 
	\item ${\State}$ is a finite set of \textit{states} with a designated \textit{final state} $q_f$, and
	\item ${\Inst}$ is a finite set of \textit{instructions} (to be formalized below) that indicate whether to, given a certain state of the machine, \textit{increment} a register or \textit{decrement} a nonzero register, as well as a ``branching'' instruction known as \textit{forking}, with no instruction applicable to the state $q_f$. 
\end{itemize}

A \textit{configuration} ${\cf}$ of a $k$-ACM is a tuple consisting of a single state and, for each register, a nonnegative integer indicating the contents of that register. We can imagine a configuration being a box labeled by a state and containing tokens each labelled by an element of the set $\Reg_k$. In essence, a configuration is specified by the state label and the multiset of register labels of the tokens.
Since the order of the symbols is irrelevant, we represent a configuration ${\cf}$ as a term in the free commutative monoid generated by ${\State}\cup \Reg_k$, and canonically arranged as
$$q\reg_1^{n_1} \reg_2^{n_2}\cdots \reg_k^{n_k}, $$
where $q\in {\State}$ is the \textit{state} of the configuration and $n_i$ is the number stored in the register $\reg_i$, for each $i=1,..,k$; if $n_i=0$, we say the register $\reg_i$ is \textit{empty}. Since ${\cf}$ contains precisely one state, we may define the set of configurations by $\conf {\acm} := {\State}\cdot \Reg_k^*$.

The instructions of a $k$-ACM replace a single configuration by a new configuration (via increment and decrement), or by two configurations (via forking). An \textit{increment} instruction can be understood as ``if a box is labeled by state $q$, add one register-$\reg_i$ token and relabel the box with state $q'$,'' \textit{decrement} as ``if a box is labeled by state $q$, and $\reg_i$ is not empty, remove one register-$\reg_i$ token and relabel the box with state $q'$,'' and \textit{forking} as ``if a box is labeled by state $q$, duplicate the box and its contents, resulting in two boxes relabeled by $q'$ and $q''$, respectively.'' As a consequence of the forking instruction, the machine can be operating on multiple configurations, i.e. branches, in parallel and is inherently nondeterministic. The status of a machine at a given moment in a computation, called an \textit{instantaneous description} (ID), is represented by the configurations that are present. Formally, an ID is an element
$${\cf}_1\vee\cdots\vee {\cf}_m, $$
of the free commutative semigroup $\indes {\acm}$ generated by $\conf {\acm}$; we denote the associated binary operation by $\vee$.

In this way, we view $\indes {\acm}$ as a subset of the commutative semiring $\mathbf{A_{\acm}}=(A_{\acm},\vee,\cdot, \bot, 1)$ generated by $\Reg_k\cup {\State}$, where
\begin{itemize}
	\item $(A_{\acm},\vee,\bot)$ is a commutative monoid with additive identity $\bot$, 
	\item $(A_{\acm},\cdot,1)$ is a commutative monoid with multiplicative identity 1, and
	\item multiplication distributes over join.
\end{itemize}
Even though $\jn$ in $\mathbf{A}_{\acm}$ is not defined to be idempotent, we will consider homomorphisms from $\mathbf{A}_{\acm}$ that will map $\jn$ to a semilattice operation and for our applications it would not hurt to define $\jn$ to be idempotent. However, the non-idempotent status of ID's matches better the intuition of computation and this is the reason for our choice.

Since multiplication fully distributes over $\jn$ in $\mathbf{A}_\acm$, each element of $A_{\acm}$ can be written as a finite join $\bigvee_{i\in I} m_i$, where $I$ is a finite (possibly empty) set, of monoid terms $m_i \in ({\State} \cup \Reg_k)^*$, for all $i \in I$; recall that the join of the empty set is the bottom element ($\bot=\bigvee \emptyset$).
As usual, each element of $\mathbf{A}_\acm$, which is the equivalence class $[t]$ of a term $t$ in the absolutely free algebra over $\{\jn, \cdot, 1\}$ and ${\State} \cup \Reg_k$, will be identified with the term $t$ itself, when no confusion arrises. 
 
Formally, an instruction $p$ of a $k$-ACM is an expression of the form 
$q\leq q'\reg_i$, $q\reg_i \leq q'$, or $q\leq q'\vee q''$, where $q, q', q'' \in {\State}$ and $\reg_i \in \Reg_k$, representing \textit{increment $\reg_i$}, \textit{decrement $\reg_i$}, and \textit{fork}, respectively. 
We will often write $p:{\cf}\leq {\uid}$ to indicate the instruction $p$ is given by ${\cf}\leq {\uid}$, where ${\cf} \in \conf {\acm}$ and ${\uid} \in  \indes {\acm}$. For a state $q\in \State$, we say $p$ is a \textit{$q$-instruction} if $p:qx\leq \uid$ for some $x\in\Reg_k^*$. Note that a machine $\acm$ with final state $q_f$ contains no $q_f$-instructions by definition.

The \textit{computation relation} $\leq$ for the machine ${\acm}=(\Reg_K,{\State},{\Inst},q_f)$ is defined to be the smallest ${\{\cdot,\vee\}}$-compatible preorder containing ${\Inst}$, and will be denoted by $\leq_{\acm}$. For a given instruction $p: {\cf}\leq {\uid}$, it will be useful to define the relation $\leq^p$ to be the closure of $p$ under the inference rules 
$$\infer[{[\cdot]}] {vx\leq^p wx }{v\leq^p w}\quad\mbox{and }\quad\infer[{[\vee]}]{ v\vee t \leq^p w\vee t}{v\leq^p w},$$
for all $v,w,x,t\in A_{\acm}$ (in that order, without loss of generality, due to the distributivity of $\cdot$ over $\vee$). Consequently, $v\leq^p w$ if and only if $v={\cf}x\vee t$ and $w=\uid x\vee t$, for some $x, t \in A_{\acm}$; these equalities are understood inside $ A_{\acm}$, so the terms $v$ and ${\cf}x\vee t$ need not be identical. We therefore conclude that if $v\leq^p w$, then $v\in\indes{\acm}$ if and only if $w\in\indes{\acm}$.

It is easily verified that $\leq_{\acm}$ is equivalent to the smallest preorder generated by $\bigcup\{\leq^p: p\in{\Inst}\}$. 
Therefore, if $s\leq_{\acm} t$, then there exist $n \in \mathbb{N}$, 
a sequence of $A_{\acm}$-terms $t_0, \ldots, t_n$ and a sequence of instructions $p_1, \ldots, p_n$ from ${\Inst}$, collectively called a \textit{computation in ${\acm}$} of \emph{length} $n$ witnessing $s\leq_{\acm} t$, such that
$$s_0=_{\mathbf{A}_\acm}t_0 \leq^{p_1} t_1 \leq^{p_2}\cdots \leq^{p_n} t_n=_{\mathbf{A}_\acm} t.$$
Clearly, if there is a computation witnessing $s\leq_{\acm} t$, then there is a computation of minimal length, the value of which we simple call the \textit{computation length}. The following result is an easy consequence of the definitions.
\begin{lem}\label{basicrel}
Let $s,t,t'\in A_{\acm}$.
\begin{enumerate}
\item  If $s\leq_{\acm} t$, then there exists a computation witnessing it and furthermore, $s\in\indes{\acm}$ iff $t\in \indes{\acm}$. 
\item $t\vee t' \leq_{\acm} s$ if and only if there exist $s',s''\in A_{\acm}$ such that $s=s'\vee s''$, $t\leq_{\acm} s'$ and $t'\leq_{\acm} s''$. Furthermore, the sum of the computation lengths of $t\leq_{\acm} s'$ and $t'\leq_{\acm} s''$ is less than or equal to the computation length of $t\vee t' \leq_{\acm} s$.
\end{enumerate}
\end{lem}

An ID consisting entirely of configurations labeled by the state $q_f$ with all registers empty is called a \textit{final ID} and we denote the set of final ID's by ${\Fid}({\acm}):= \{ \bigvee_{i=1}^n q_f: n\geq 1\}$. Our choice to not assume the idempotency of $\jn$ in  $\m A_{\acm}$ explains the necessity of treating $ \bigvee_{i=1}^n q_f$ as a final ID.  We say that a term $t\in A_{\acm}$ is \textit{accepted by ${\acm}$} if there exists ${\uid}_f\in {\Fid}({\acm})$ such that $t\leq_{\acm} {\uid}_f$, and we define the set \textit{accepted terms} to be ${\Acc}({\acm}):=\{t\in A_{\acm}: \exists {\uid}_f\in {\Fid}({\acm}),~t\leq_{\acm} {\uid}_f \}$. 
\begin{thm}\label{lincoln}
(\cite{Lambek, Minsky,Lincoln}) There exists a 2-ACM $\Uacm$ such that membership in ${\Acc}(\Uacm)$ is undecidable.
\end{thm}
\begin{ex}\label{ex: Meven}
Consider the $1$-ACM ${\acm}_{\mathrm{even}}= (\Reg_1,{\State}_{\mathrm{even}},{\Inst}_{\mathrm{even}},q_f)$, where ${\State}_{\mathrm{even}}=\{q_0,q_1,q_f\}$ and ${\Inst}_{\mathrm{even}}=\{ p_0,p_1,p_f\}$ is given by 
$$\begin{array}{ l c r c l}
p_0 &:& q_0 \reg_1& \leq & q_1 \\
p_1 &:& q_1 \reg_1& \leq & q_0 \\
p_f &:& q_0 & \leq & q_f\vee q_f .
\end{array} $$
For example, $q_0\reg_1^2\leq^{p_0} q_1\reg_1 \leq^{p_1} q_0 \leq^{p_f} q_f \vee q_f$ is a computation showing that $q_0\reg_1^2$ is accepted. On the other hand  $q_0\reg_1 \leq^{p_0} q_1$ and $q_0\reg_1 \leq^{p_f} (q_f \vee q_f)\reg_1=q_f\reg_1 \vee q_f\reg_1$ are the only maximal computations starting with $q_0\reg_1$ and none of them ends in a final configuration, so $q_0\reg_1$ is not accepted. In general, it is easy to see that $q_0\reg_1^n\in{\Acc}({\acm}_\mathrm{even})$ if and only if $n$ is even. 
\end{ex}
\begin{lem}\label{comprel}
Let ${\acm}$ be an ACM.
\begin{enumerate}
	\item ${\Acc}({\acm})\subseteq \indes {\acm}$ and the terms in any computation ending in a final ID are all ID's.
	\item For all ${u},{v}\in A_{\acm}$, ${u}\vee {v}\in {\Acc}({\acm})$ if and only if ${u}\in{\Acc}({\acm})$ and ${v}\in{\Acc}({\acm})$.
\end{enumerate}
\end{lem} 
\begin{proof} The first claim follows from Lemma~\ref{basicrel}(1) by induction on the computation length since $\Fid(\acm)\subseteq\indes{\acm}$ by definition. The second claim follows from Lemma~\ref{basicrel}(2) since $\Fid(\acm)$ is exactly the set of finite non-empty joins of the same configuration $q_f$.
\end{proof}

Lemma~\ref{comprel} shows that in a computation witnessing the acceptance of an ID all configurations are ID's and therefore for those cases the inference rule $[\cdot]$ could have been restricted to $x \in \Reg_k^*$.
%
%SECTION: Machines and Residuated Frames
%
\section{Machines and Residuated Frames}
As a demonstration of our general technique, we will use counter machines and residuated frames to show that any variety between $\CRL $ and $\RL $ has an undecidable word problem. 

Let $\acm=(\Reg_k,\State,\Inst,q_f)$ be an ACM. For each $u\in \indes{\acm}$, formally viewed as an $\{\vee,\cdot,1\}$-term in $\mathbf{T}(\Reg_k\cup \State)$, we define the quasiequation $\mathrm{acc}_{\acm}(u)$ to be 
$$\amper \Pcom \Rightarrow {u}\leq q_f ,$$
where $q_f$ is the final state of ${\acm}$, and $\Pcom:={\Inst}\cup\{xy\leq yx: x,y\in \Reg_k\cup\State\}$ is the (finite) set of instructions ${\Inst}$ together with a finite set encoding commutativity for letters (and hence also all words) over the set $\Reg_k\cup {\State}$; for our purposes we could actually restrict the $x$ in $\Pcom$ to only state variables. 

The following lemma shows that computations in machines can be performed also in the theory of residuated lattices.
\begin{lem}\label{crl1} 
Let ${\acm}$ be an ACM and ${u}$ an ID. If ${u}\in {\Acc}({\acm})$ then $\RL\models \mathrm{acc}_{\acm}({u})$. 
\end{lem}
\begin{proof}
Let $\acm=(\Reg_k,\State,\Inst,q_f)$ be an ACM and suppose $u\in \indes{\acm}$ is accepted in $\acm$. We proceed by induction on the length $n$ of the computation  witnessing the acceptance of $u$ in $\acm$. If the length is zero then $u\in \Fid(\acm)$. Since $\vee$ is idempotent in residuated lattices, $\RL\models \uid_f \leq q_f$ for any $\uid_f \in \Fid(\acm)$. Hence $\RL\models \mathrm{acc}_\acm (u)$ {\em a fortiori}. Now, suppose the claim holds for all accepted ID's with computation length $0\leq k<n$.  By Lemma~\ref{comprel}(1),  there is an instruction $p\in\Inst$ such that $u\leq^p u'\in \Acc(\acm)$ for some $u'\in \indes{\acm}$, where the acceptance computation of $u'$ has length less than $n$. Formally viewing $u'$ as an element in $\mathbf{T}(X)$ where $X=\Reg_k\cup \State$, $\RL\models \mathrm{acc}_\acm(u')$ by the induction hypothesis.

Now, suppose that for a residuated lattice $\mathbf{R}$ and for a homomorphism $f:\mathbf{T}(X)\to\mathbf{R}$ we have $\mathbf{R},f\models \Pcom$. Hence $f(u')\leq_\mathbf{R} f(q_f)$ since $\RL\models \mathrm{acc}_\acm(u')$. As $\leq_\mathbf{R}$ is transitive, we need only show $f(u)\leq_\mathbf{R} f(u')$ to establish $\mathbf{R},f\models \mathrm{acc}_\acm(u)$.

Let $\mathbf{S}(X)$ be the free algebra over $\{\vee,\cdot,1\}$ and $X$. As $\mathbf{R}$ has a semiring reduct and $f(a)f(b)=_\mathbf{R}f(b)f(a)$, for all $a,b\in X$, the restriction of $f$ on $\m{S}(X)$ factors through $\mathbf{A}^+_{\acm} := \mathbf{A}_{\acm} \setminus \{\bot\}$ as $f:S(X)\stackrel{\nu}{\to}A^+_{\acm} \stackrel{h}{\to}R$ as a semiring homomorphism, where $\nu$ is the natural surjective homomorphism and $h$ is a semiring homomorphism. So, 
 $h(a)h(b)=_{\mathbf{R}}h(b)h(a)$, for all $a,b\in X$, and $ h({\cf})\leq_{\mathbf{R}}  h ({\vid})$ where $p:\cf\leq \vid$.
By definition of $\leq^p$, $$u=_{\mathbf{A}_\acm} \cf x\vee w\leq^p \vid x \vee w =_{\mathbf{A}_\acm} u',$$ for some $x\in X^*$ and $w\in A_\acm$, where $\vid x \vee w=_{\mathbf{A}_\acm} \vid x$ if $w=\bot$. Using the properties above and the fact that $h$ is a semiring homomorphism we obtain
$$h(u)=_\mathbf{R}h ({\cf}x \vee w) \leq_{\mathbf{R}} h({\vid}x\vee w)=_\mathbf{R}h(u').$$ 
It follows that $ h ({u}) \leq_{\mathbf{R}}h(u')$ and therefore that $f(u) \leq_{\mathbf{R}} f(u').$
\end{proof}

To show the converse of Lemma~\ref{crl1}, we will need to show that given an ACM  ${\acm}$, if ${u}\nin {\Acc}({\acm})$ then there is a residuated lattice $\mathbf{W}_{\acm}^+$ (which will actually even be commutative) that falsifies $\mathrm{acc}_{\acm}({u})$; actually, in the proof we proceed by contraposition. We will further prove that every subvariety of $\RL $ that contains $\mathbf{W}_{\Uacm}^+$ has undecidable word problem (and thus undecidable quasiequational theory). The construction of $\mathbf{W}_{\acm}^+$ is based on residuated frames \cite{GJ}, structures that will also be used later in the paper, so we define them briefly here.

 For the purposes of this paper, a \textit{commutative residuated frame} is a structure $\mathbf{W}=(W,W', \Nuc, \cdot, 1)$, where $(W, \cdot, 1)$ is a commutative monoid, $W'$ is a set, and $\Nuc$ is a subset of $W \times W'$, such that 
there exists a function $\sslash: W' \times W \rightarrow W'$ with: $\forall x,y \in W$ $z \in W'$, $x \cdot y \Nuc z$ iff $x \Nuc z \sslash y$.
Given such a residuated frame, 
for $X\subseteq W$, $x \in W$, $Y \subseteq W'$ and $y \in W'$, we define
$X \Nuc y$ to mean $x \mathrel{N} y$ for all $x \in X$, and 
$x \Nuc Y$ to mean $x \mathrel{N} y$ for all $y \in Y$. 
For $X\subseteq W$ and $Y\subseteq W'$, we define $X^\triangleright:=\{y\in W':X \Nuc y\}$, $Y^\triangleleft:=\{x\in W: x \Nuc Y\}$. The pair $({}^\triangleright,{}^\triangleleft)$ forms what is known as a {\em Galois connection}, and we will make use of the fact that $X_1^{\triangleright\triangleleft} \subseteq X_2^{\triangleright\triangleleft}$ if and only if $X_2^{\triangleright}\subseteq X_1^{\triangleright}$ for any $X_1,X_2\subseteq W$. 

Define $\gamma(X)= X^{\triangleright\triangleleft}$. We write $\gamma(x)=\gamma(\{x\})$ for $x\in W$, and $\powerset(W)_\gamma=\gamma[\powerset(W)]$. It follows from \cite{GJ} that the algebra $\mathbf{W}^+:=(\powerset(W)_{\gamma},\cap,\cup_{\gamma} , \cdot_{\gamma},\to, \gamma(1))$ is a commutative residuated lattice, where $X\cup_\gamma Y:=\gamma(X\cup Y)$, $X\cdot_\gamma Y:=\gamma(X\cdot Y)$, and $X\to Y:=\{z\in W: X\cdot\{z\}\subseteq Y\}$.

Inspired by \cite{Hor}, given an ACM ${\acm}=(\Reg_k, {\State},{\Inst},q_f)$ we define the tuple $\mathbf{W}_{\acm}=(W_{\acm},W_{\acm}, {\Nuc}_{\acm}, \cdot, 1)$, where $W_{\acm}:=({\State}\cup \Reg_k)^*\subseteq A_{\acm}$ and $x\Nuc_{\acm}y$ if and only if $xy\in {\Acc}({\acm})$, for all $x,y\in W_{\acm}$.
\begin{lem}
$\mathbf{W}_{\acm}$ is a residuated frame and therefore $\mathbf{W}_{\acm}^+ \in \CRL $. 
\end{lem}
\begin{proof} We define $z \sslash y = yz$. 
Clearly, for $x,y,z\in W_{\acm}$,
$xy \Nuc_{\acm} z$ iff $xyz\in{\Acc}({\acm})$ iff $x \Nuc_{\acm} yz.$ 
\end{proof}

For an ACM ${\acm}=(\Reg_k,{\State},{\Inst},q_f)$, we define the assignment $e:{\State}\cup \Reg_k\to W_{\acm}^+$ via $e(a):=\{a\}^{\triangleright\triangleleft}$
and its homomorphic extension $\bar e:\m{T}({\State}\cup \Reg_k)\to \m{W}_{\acm}^+$.

We will need to make use of the following technical lemma.
\begin{lem}\label{valuation}
If  ${\acm}=(\Reg_k,{\State},{\Inst},q_f)$ is an ACM, then $\mathbf{W}_{\acm}^+,\bar e\models{\Pcom}$. Furthermore, $\bar{e} (x \vee y) = \{x, y\}^{\triangleright\triangleleft}$ for any $x, y \in W_{\acm}$.
\end{lem}
\begin{proof}
In \cite{GJ} it is shown that the map $\gamma$ satisfies the properties $\gamma(\gamma(X) \cdot \gamma(Y))=\gamma(X \cdot Y)$ and $\gamma(\gamma(X) \cup \gamma(Y))=\gamma(X \cup Y)$, for all $X , Y \sbs W$. Using the first one, for each $a,b\in {\State}\cup \Reg_k$ we have
\[
\bar e(ab)
 = \bar e(a)\cdot_{\gamma} \bar e(b)
 = e(a)\cdot_{\gamma} e( b)
 =\gamma( \gamma(a)\cdot \gamma( b))
 = \gamma(ab). 
\]
It follows by induction that $\bar e(x)=\gamma(x)$ for each $x\in W_{\acm}$. 

Now, let $x,y\in W_{\acm}.$ Then
\[
\bar e(x\vee y) 
= \bar e(x) \cup_{\gamma} \bar e(y) 
= \gamma(x) \cup_{\gamma} \gamma(y)
=\gamma(\gamma(x)\cup\gamma(y))
=\gamma(\{x,y\})
\]
where the last equality follows from the second property of $\gamma$ above.

Since $\mathbf{W}_\acm^+$ is commutative, we need only show $\mathbf{W}_\acm^+\models \Inst$. Let $p:{\cf}\leq {\cf}_1\vee {\cf}_2$ be in $\Inst$.\footnote{The argument that follows clearly works for $\vid=\cf_1$ as well.} By the calculation above, $\bar e({\cf})=\{{\cf}\}^{\triangleright\triangleleft}$ and 
$\bar e({\cf}_1 \vee {\cf}_2)= \{{\cf}_1, {\cf}_2\}^{\triangleright\triangleleft}$. So, to show $\mathbf{W}_{\acm}^+,\bar e\models {\cf}\leq {\cf}_1\vee {\cf}_2$, we need to show that $\{{\cf}\}^{\triangleright\triangleleft} \sbs \{{\cf}_1, {\cf}_2\}^{\triangleright\triangleleft}$, or equivalently that $\{{\cf}_1,{\cf}_2 \}^\triangleright \subseteq \{{\cf}\}^\triangleright$.  Suppose $x\in \{ {\cf}_1,{\cf}_2\}^\triangleright$, then 
${\cf}_1 \Nuc_{\acm} x$ and ${\cf}_2 \Nuc_{\acm} x$, so ${\cf}_1 x, {\cf}_2 x\in {\Acc}({\acm}).$ 
Now ${\cf}\leq^p {\cf}_1\vee {\cf}_2$ implies ${\cf}x\leq^p ({\cf}_1\vee {\cf}_2)x$, thus by Lemma~\ref{comprel}(2)
$${\cf}x \leq^p ({\cf}_1\vee {\cf}_2) x = {\cf}_1x \vee {\cf}_2x \in {\Acc}({\acm}), $$
and it follows that ${\cf} \Nuc_{\acm} x$, or equivalently $x\in \{ {\cf}\}^\triangleright$. 
\end{proof}
\begin{lem}\label{crl} 
Let $\mathcal{V}$ be a subvariety of $\RL $ containing $\mathbf{W}_{\acm}^+$ for some ACM ${\acm}$. Then for all ${u}\in \indes{\acm}$, ${u}\in {\Acc}({\acm})$ if and only if $\mathcal{V}\models \mathrm{acc}_{\acm}({u})$. 
\end{lem}
\begin{proof}  Let $\acm=(\Reg_k,\State,\Inst,q_f)$ be a $k$-ACM. 
The forward direction follows from Lemma~\ref{crl1}. For the reverse direction note that from $\mathbf{W}_{\acm}^+\in\mathcal{V}$ we have $\mathbf{W}_{\acm}^+\models \mathrm{acc}_{\acm}({u})$. By Lemma~\ref{valuation}, $\mathbf{W}_{\acm}^+,\bar e\models {\Pcom}$ and so $\mathbf{W}_{\acm}^+,\bar e \models {u}\leq q_f$. 
Let $t_1,...,t_n\in (\State\cup\Reg_k)^*$ be given so that ${u}={t}_1\vee\cdots\vee {t}_n$, so $\bar{e} ({t}_1\vee\cdots\vee {t}_n) \subseteq \bar{e}(q_f)$, which yields $\{{t}_1, \ldots, {t}_n\}^{\triangleright\triangleleft} \subseteq {q_f}^{\triangleright\triangleleft}$ by Lemma~\ref{valuation}. This is equivalent to ${\{q_f\}^\triangleright} \subseteq \{{t}_1,...,{t}_n \}^\triangleright.$ 
 Since $\leq_{\acm}$ is reflexive, $q_f\in {\Acc}({\acm})$ and thus $q_f\Nuc_\acm~1,$
so $1\in\{q_f\}^\triangleright$. Therefore, $1\in \{{t}_1,...,{t}_n \}^\triangleright$, that is $\{{t}_1,...,{t}_n \} \Nuc_\acm 1$, so $t_1 \Nuc_\acm 1, \ldots ,t_n \Nuc_\acm 1$. Hence ${t}_1,...,{t}_n\in{\Acc}({\acm})$ and by Lemma~\ref{comprel}(2) we conclude $u\in{\Acc}({\acm})$.
\end{proof}

As a consequence of Lemma~\ref{crl}, if $\mathcal{V}\subseteq\RL $ is a variety containing $\mathbf{W}_{\acm}^+$, for some ACM $\acm$, then $\{\mathrm{acc}_{\acm}({u}): {u}\in {\Acc}({\acm}) \}=\{\mathrm{acc}_{\acm}({u}): \mathcal{V}\models \mathrm{acc}_{\acm}({u})\}$. Since $\langle {\State}\cup \Reg_k, {\Pcom}\rangle$ is a finite presentation and all equations in $\mathrm{acc}_{\acm}({u})$ have a common antecedent $\& {\Pcom}$, the following is immediate:
\begin{thm}\label{Vmhard} Let ${\acm}$ be an ACM and $\mathbf{W}^+_{\acm}\in\mathcal{V}\subseteq\RL $ for a variety $\mathcal{V}$. Then deciding the word problem of $\mathcal{V}$ is at least as hard as deciding membership in ${\Acc}({\acm})$.
\end{thm}
\begin{cor}\label{wp}
 If $\mathcal{V}$ is a subvariety of $\RL $ containing $\mathbf{W}^+_{\acm}$, where ${\acm}$ is an ACM such that membership in ${\Acc}({\acm})$ is undecidable, then $\mathcal{V}$ has an undecidable word problem. In particular, any variety in the interval $\CRL$ to $\RL$ has undecidable word problem since $\mathbf{W}_{\Uacm}^+\in \CRL $, where $\Uacm$ is the machine from Theorem~\ref{lincoln}.
\end{cor} 
The above results hold even for the $\{\jn, \cdot, 1\}$ reducts of these varieties. 
Since $\{\mathrm{acc}_{\acm}({u}): \mathcal{V}\models \mathrm{acc}_{\acm}({u})\}\subseteq \{\xi: \mbox{$\xi$ is a quasieq. such that }\mathcal{V}\models \xi \}$, we therefore also obtain the undecidability of the quasiequational theory. 
The quasiequational theories of $\RL$ and $\CRL$ alone were known to be undecidable; see \cite{JT02} and \cite{Lincoln}, respectively.
%
%SECTION: Equations in the signature $\{\vee,\cdot, 1\}$ and machine admissibility
%
\section{Equations in the signature $\{\vee,\cdot, 1\}$ and machine admissibility} 
Our goal is to find proper subvarieties of $(\Cc )\RL $ for which Theorem~\ref{Vmhard} will be applicable, as well as strengthening this result to the undecidability of the equational theory for some proper subvarieties of $\CRL .$ Since structural rules correspond to equations in the signature $\{\vee,\cdot, 1\}$ (see \cite{GJ}), we will restrict our attention to varieties axiomatized by such equations.

Since in residuated lattices multiplication distributes over joins, every equation over $\{\vee,\cdot,1\}$ is equivalent to an equality between finite joins of monoid terms. This equality can in turn be written as two inequalities and in each one of them the joins on the left-hand side of the inequality yield a conjunction of inequalities of the form  $t_0\leq t_1\vee\cdots \vee t_l$, where $t_0,...,t_l\in X^*$ are monoid terms. We call equations of this form  \emph{basic equations} (or basic inequalities), if it is further true that the variable sets on the two sides of the inequality are the     same. It can be easily shown that joinands on the right-hand side containing variables that do not appear on the left can be safely omitted, resulting in an equivalent equation. (In the case where all the joinands are of this form, the equation implies $1 \leq x$, so it defines the trivial variety).  Furthermore, if there are variables that appear on the left and not on the right, then the inequality implies integrality ($x \leq 1)$. These claims are easy to prove and the needed instantiations are also mentioned on page 277 of \cite{CGT2012}. The trivial variety and variety axiomatized by integrality together with any set of $\{\vee,\cdot,1\}$-equations have the FEP, hence decidable universal (and quasiequational) theory. Therefore the restriction of the variables appearing on both sides does not leave out any unknown cases of (un)decidability. Via a process of \emph{linearization} \cite{GJ} any basic equation is further equivalent to one where the term $t_0$ is {\em linear}, namely to an equation of the form $x_1\cdots x_n \leq \bigvee_{i=1}^k m_i, $ where $k \geq 1$, $m_1,...,m_k\in X^*$ and  $x_1, \ldots, x_n$ are distinct elements of $X$. (For example the equation $x^2 \leq x$ is equivalent to the linearized equation $x_1 x_2 \leq x_1 \vee x_2$). Such equations are called simple in \cite{GJ}, but in this paper we will reserve the name \emph{simple equations of $\RL$} for the subclass where the variable sets on the left and the right-hand sides of the equation are equal.

When writing simple equations, we will be using the set of variables $\{x_i : i \in \Z^+\}$, and we will assume implicitly that this set is ordered by the natural order of the indices. We will informally use variables like $x, y, z$ in some of the examples, as well. We also define $\mathbf{x_n} := (x_1, \ldots, x_n)$, for all $n \in \Z^+$ and for a tuple $a\in \N^n$ of natural numbers, we define $\mathbf{x_n}^{a}=x_1^{a(1)}\cdots x_n^{a(n)}$; we also define $\mathbf{x_n}^{\mathbf{1}}=x_1\cdots x_n$. For reasons that will be clear later, we will actually think of $a$ as a column vector (as opposed to a row vector).

For a simple equation $\varepsilon$, we will be interested in its commutative version $\varepsilon_\mathcal{C}$, obtained from $\varepsilon$ by rearranging the variables within each monoid term according to the natural ordering of their indices and removing any resulting duplicate joinands on the right-hand side. In particular, $\CRL +\varepsilon\models \varepsilon_\comeq$, and we call equations of the form $\varepsilon_\comeq$ {\em simple equations of $\CRL $}. As the encodings for the undecidability are harder for commutative varieties than for arbitrary ones, by proving the results in the commutative case we obtain as corollaries results for general subvarieties of $\RL$. Therefore,  
we restrict ourselves to simple equations of $\CRL $ and we will refer to them simply as \emph{simple equations}. 
 Such equations are of the form 
$$[\Dset]: \mathbf{x_n}^{\mathbf{1}} \leq \bigvee_{d\in \Dset} \mathbf{x_n}^{\vd},$$
where $\Dset$ is a finite nonempty set of $n$-column vectors with entries in $\N$, such that the variable sets on the two sides are the same, namely there is no row such that all column vectors in $\Dset$ are zero on that row. Note that because of the equality of the variable sets on the left and on the right and due to the idempotency of join in residuated lattices, every simple equation is fully determined by the set of joinands on its right-hand side. Our notation is chosen so that if $\Dset$ is a set of $n$-columns, then $[\Dset]$ denotes the simple equation displayed above (the exponents of the joinands on the right-hand side come from $\Dset$).
\begin{figure}[!htb]
\[
\Scale[.85]{
\begin{array}{ c | c  | c | c } 
&[\Dset]  &  \Dset \\
\hline
\mathrm{i.}&x\leq x^2 &  \{ 2\} \\ \hline
\mathrm{ii.}&x\leq 1\vee x^2  & \{0, 2\}\\ \hline
\mathrm{iii.}&x\leq x^{2}\vee x^{4} & \{2,4 \} \\ \hline
\mathrm{iv.}&xy\leq1\vee x^2y\vee x^3y^2 &  
\left\{\hspace{-5pt}\begin{array}{rcl}
   0& 2 &  3  \\
    0&1 &  2 
\end{array}\hspace{-5pt}\right\} 
\end{array}
\begin{array}{ c | c  | c | c } 
&[\Dset]  &  \Dset \\
\hline
\mathrm{v.}&xyz \leq x^2y\vee y^2z\vee xz^2  & 
\left\{\hspace{-5pt}\begin{array}{rcl}
      2&0 & 1 \\
      1&2 & 0\\
	  0&1 & 2
\end{array}\hspace{-5pt}\right\} 
\\ \hline
\mathrm{vi.} & xyz \leq yz\vee xz^2 &  
\left\{\hspace{-5pt}\begin{array}{rcl}
      0 &1  \\
      1 &0\\
	  1&2
\end{array}\hspace{-5pt}\right\}
\end{array}
 }
 \]
 %\vspace{-10pt}
 \captionof{table}{Some simple equations viewed as sets of column vectors.}\label{Rtable}
 \end{figure}

For the bigger class of basic equations of $\CRL$ (which may not be linearized), if $\Dset$ is again a nonempty set of column vectors and $f$ a column vector over the positive integers, we denote by $[f,\Dset]$, the basic equation 
$$[f,\Dset]: \mathbf{x_n}^{f} \leq \bigvee_{\vd\in \Dset} \mathbf{x_n}^{\vd}.$$ 

The following theorem provides a link between simple equations that hold in $\mathbf{W}^+$ and conditions that hold in $\mathbf{W}$.

Let $(W,\cdot,1)$ be a commutative monoid and $[\Dset]$ a $n$-variable simple equation given by $\Dset=\{\vd_j : 1 \leq j \leq m\}$. If $\mathbf{W}=(W,W',\Nuc, \cdot, 1)$ is a residuated frame then we write $\mathbf{W}\models (\Dset)$ iff for all $\mathbf{u_n}\in W^n$ and $v\in W'$, the following implication is satisfied (the premises above the line are understood conjunctively and the vertical line denotes the implication to the conclusion below):
$$ \infer[{(\Dset)}] {\mathbf{u_n}^\mathbf{1}\Nuc v }{\mathbf{u_n}^{\vd_1} \Nuc v& \cdots &\mathbf{u_n}^{\vd_m} \Nuc v}$$
For example if $[\Dset]$ is $x_1 x_1 \leq x_1 \jn x_2$, then $(\Dset)$ is: $$\forall x_1, x_2 \in W, v \in W', x_1 \Nuc v ~\&~ x_2 \Nuc v \implies x_1x_2 \Nuc v.$$
\begin{thm}[\cite{GJ}]\label{nucond}
Let $[\Dset]$ be a simple equation and suppose $\mathbf{W}$ is a residuated frame. Then $\mathbf{W}^+\models [\Dset]$ iff $\mathbf{W}\models (\Dset)$.
\end{thm}
%
%Subsection: Motivation for subvarieties of $\CRL$
%
\subsection{Motivation for subvarieties of $\CRL$}\label{motivation}
Recall from Example~\ref{ex: Meven}, that the computations of the $1$-ACM ${\acm}_{\mathrm{even}}$ leading to a final state are faithfully represented by the inequality relation of $\CRL $, in the sense that $\CRL \models (\& {\Inst} \Rightarrow u \leq q_f)$ iff $u \in \Acc(\acm)$. If we consider the inequality relation in $\CRL _{{\Dnots}}$, where ${\Dnots}$ is the simple equation $ x\leq x^2\vee x^4$, we observe that for the computation relation of a machine to be faithfully represented by the associated inequality relation it must further admit the ``ambient instruction'' given by 
$$t \leq^{\Dnots} t^2 \vee t^4, $$
for all $t\in ({\State}_{\mathrm{even}}\cup \Reg_1)^*$ in addition to being closed under the inference rules $[\cdot]$ and $[\vee]$. Let $\leq_{{\Dnots} {\acm}_{\mathrm{even}}}$ be the smallest compatible preorder generated by ${{\Inst}_{\mathrm{even}}}\cup \leq^{\Dnots}$, and define ${\Acc}({\Dnots} {\acm}_\mathrm{even})$ to be the set of accepted ID's under the relation $\leq_{{\Dnots}{\acm_\mathrm{even}}}$. It is clear that ${\Acc}({\acm}_\mathrm{even})\subseteq {\Acc}({\Dnots} {\acm}_\mathrm{even})$ since ${\leq_{\acm_\mathrm{even}}}\subseteq{\leq_{\Dnots\acm_\mathrm{even}}}$, and since there are no instructions (nor instances of $\leq^{\Dnots}$) that remove state variables we obtain ${\Acc}({\Dnots} {\acm}_\mathrm{even})\subseteq \indes {\acm_\mathrm{even}}$. However, while $q_0\reg_1^3 \nin {\Acc}({\acm}_\mathrm{even})$, we have
$q_0\reg_1^3 \in{\Acc}({\Dnots} {\acm}_\mathrm{even})$ since
$$q_0\reg_1^3 \leq^{\Dnots} q_0\reg_1^6\vee q_0\reg_1^{12}\in{\Acc}({\acm}_{\mathrm{even}}).$$ It is clear that the expansion of the machine by the ambient instruction (needed for representing the inequality relation in $\CRL _{{\Dnots}}$) does not have the same computation relation, or put differently the machine ${\acm}_{\mathrm{even}}$  is not suitable for representing the inequality relation in $\CRL _{{\Dnots}}$ because these ambient instructions are not already \textit{admissible} in it. 

Likewise, there is no guarantee that there is a machine that has an undecidable acceptance problem (for example the machine $\Uacm$) and in which these ambient instructions are available/admissible. For that reason we cannot use the same argumentation to show that $\CRL _{{\Dnots}}$ has undecidable word problem. 

Exactly the same issue occurs if the simple equation is contraction ${\sf{c}}:x \leq x^2$. Actually, for the case of contraction not only does this particular encoding fail to be faithful, but there is no faithful encoding of an undecidable machine: the word problem for $\CRL _{\sf{c}}$ is actually decidable \cite{vanA}.  However, we will show that even though for the equation ${\Dnots}$ above the current encoding is problematic (as is with contraction), surprisingly, unlike with contraction, there is a different encoding that works for ${\Dnots}$; this will allow us to prove that the word problem for $\CRL _{{\Dnots}}$ is undecidable. We present the idea of this new encoding by showing that it at least faithfully encodes the machine ${\acm}_\mathrm{even}$. As we will see, what makes it work is that the new encoding is such that, even if they were available, the ambient instructions would not contribute to any more accepted configurations; this is a rephrasing of what we referred to as: the given equation is \emph{admissible} in the particular machine. 

The idea is to construct a new machine ${\acm}_K$, for an appropriate integer $K$, as a modification of ${\acm}_\mathrm{even}$ that works at an exponential scale (with base $K$) compared to that of  ${\acm}_\mathrm{even}$. In particular, ${\acm}_K$ manages to replace the decrement instructions $p_0: q_0 \reg_1 \leq  q_1$ and $p_1 : q_1 \reg_1 \leq  q_0$ by {\em programs} (sets of instructions) $\mathcal{P}_0$ and $\mathcal{P}_1$, respectively, that divide the contents of register $\reg_1$ by the fixed constant $K$. For example, the general effect for $p_0$ being $q_0\reg_1^m\leq^{p_0} q_1\reg_1^n$ iff $m=n+1$ is mirrored by $\mathcal{P}_0$ in the sense that $q_0\reg_1^{M} \leq_{\mathcal{P}_0}q_1\reg_1^N$ iff $M=K\cdot N$, and consequently $q_0\reg_1^{K^{n+1}} \leq_{\mathcal{P}_0}q_1\reg_1^{K^n}$; therefore computations in $\acm_\mathrm{even}$ are simulated in $\acm_K$ by storing the contents of $\reg_1$ by $K^n$ instead of $n$. In this case, we will say a term is accepted if it computes a join of configurations of the form $q_f\reg_1^{K^0}$ (i.e. $q_f\reg_1$), so $q_0\reg_1^n\in{\Acc}({\acm}_K)$ iff $n=K^{2m}$ for some $m\geq 0$. Thereupon an additional necessary condition for acceptance in $\acm_K$ is demanded for configurations labeled by a state $q\in\State_{\mathrm{even}}$ (independently from the conditions of acceptance in $\acm_\mathrm{even}$) namely that if $q\reg_1^N$ is accepted in $\acm_K$ then $N$ must be a power of $K$.\footnote{This definition of acceptance for the machine ${\acm}_K$ is for heuristic convenience. In Section~\ref{mksec}, to properly define programs to multiply/divide by $K$, we will need to add new states and instructions to carry out such computations, as well a fresh variable $q_F$, acting as a new final state, and a set of instructions that guarantee $q_f\reg_1\leq_{{\acm}_K}q_F$.} For the equation ${\Dnots}$, if we choose $K> 2$ it is easily verified that if  ${\Dnots}$ is applied
$$q\reg_1^n \reg_1^m \leq^{\Dnots} q\reg_1^n( \reg_1^{2m} \vee \reg_1^{4m})  = q\reg_1^{n+2m}\vee q\reg_1^{n+4m},$$
the only way $n+2m$ and $n+4m$ are both powers of $K$ is if $m=0$.\footnote{Indeed, if $n+2m=K^a$ and $n+4m=K^{a+b}$, for some $a\geq 0$ and $b\geq 1$, then 
$K^a\geq 2m = K^{a+b}-K^a\geq K^a(K-1),$ 
and hence $K\leq 2$.} In such an instance, the configuration on the left-hand side of the equation appears as a joinand on the right-hand side. Consequently we see that, with respect to being accepted, instances of $\leq^{\Dnots}$ in a computation are superfluous, and we obtain
$$q\reg_1^n \reg_1^{2m}\vee q\reg_1^n\reg_1^{4m} \in {\Acc}({\acm}_K )
\quad\implies \quad
q\reg_1^n \reg_1^m \in  {\Acc}( {\acm}_K ),$$
thus ${\Acc}({\acm}_K) = {\Acc}({\Dnots}{\acm}_K )$. So, the equation ${\Dnots}$ is \textit{admissible} in the machine ${\acm}_K$.

The reason why this works is that the effect of the inequality ${\Dnots}$, even when applied repeatedly, is to modify the register values in a linear or polynomial way, but when these values are encoded on an exponential scale the applications of the inequality do not produce modifications on the same scale and thus do not lead to final configurations.

More generally, consider an $n$-variable simple equation $[\Dset]: \mathbf{x_n}^{\mathbf{1}} \leq \bigvee_{\vd \in \Dset} \mathbf{x_n}^{\vd}$. For $[\Dset]$ to be admissible, and viewing $[\Dset]$ as an ambient instruction, we need to consider all the substitution instances $\mathbf{t_n}^\mathbf{1}\leq \bigvee_{\vd\in\Dset}\mathbf{t_n}^\vd$ of $[\Dset]$, where the tuple of terms $\mathbf{t_n}=(t_1,\ldots, t_n)$ is given by a substitution $\subt: x_i \mapsto t_i$, for all $i$. Then for any term $s$, $s\mathbf{t_n}^\mathbf{1}\leq^{\Dset} s\bigvee_{\vd\in\Dset}\mathbf{t_n}^\vd$ is part of the computation that includes the ambient instructions coming from $[\Dset]$. It is shown in Lemma~\ref{mingly} that if $[\Dset]$ does not have instances equivalent to a $k$-mingle equation\footnote{Note that if there is an instance of $[\Dset]$ that is equivalent to {\em $k$-mingle} ($x^k\leq x$) for some $k>1$,  $[\Dset]$ cannot be admissible for any ACM $\acm$: from $q_f^k\leq^{\Dset}q_f$ we would obtain that $q_f^k$ is accepted, a contradiction. Actually, then the variety of $\CRL+[\Dset]$ has a decidable word problem. More generally, we note that $k$-mingle, as well as contraction, are examples of {\em knotted equations}: equations of the form $x^{k}\leq x^{l}$, where $k\neq l$. It is known \cite{vanA} that all knotted subvarieties of $\CRL$ have decidable universal theories, and therefore so do the subvarieties of $\CRL$ axiomatized by any set of simple equations $\Gamma$ for which $\CRL+\Gamma\models x^k\leq x^l$ by \cite{GJ}.} and $s\bigvee_{\vd\in\Dset}\mathbf{t_n}^\vd$ is accepted in an ACM $\acm$, then the term $s$ contains precisely one state variable and no term $t_i$ contains any state variable. In the case where $\acm$ is a $1$-ACM, for example $\acm=\acm_\mathrm{even}$, this implies that $s=q\reg_1^{\const}$, for some state $q$ and number $\const$, and $\subt$ is a (1-variable) substitution with $\subt(x_i)= t_i= \reg_1^{\rvec(i)}$, where $\rvec$ is an $n$-tuple of natural numbers. Using the equality relation for $\mathbf{A}_\acm$, this is equivalently written as
\begin{equation}\label{eq:Ron1acm}
q\reg_1^{\const+\rvec\mathbf{1}}
\leq^\Dset 
\bigvee_{\vd\in\Dset} q\reg_1^{\const + \rvec \vd},
\end{equation} 
where $\rvec \vd= \rvec(1)\vd(1)+ \cdots \rvec(n)\vd(n)$ and $\rvec \mathbf{1}= \rvec(1)+ \cdots \rvec(n)$.

Admissibility of $[\Dset]$ in such a 1-ACM $\acm$ is the demand that if the right-hand side of the above inequality is accepted in $\acm$ then the left-hand side is also accepted (thus making every instance of $[\Dset]$ superfluous). The most naive and obvious way to ensure this is to ask that the left-hand side already appears as one of the joinands on the right-hand side; that is $\rvec\Lvec{}=\rvec\bar \vd$ for some $\bar \vd\in \Dset$, hence rendering the substitution instance of $[\Dset]$ by $\subt$ trivial. Recall that in the case of the machine $\acm_K$ constructed from $\acm_\mathrm{even}$, if a configuration $q\reg_1^N$ is accepted in $\acm_K$ then $N$ must be some power of $K$. So, for $[\Dset]$ to be admissible in $\acm_K$ the most obvious condition to require is:
\begin{center}
If the exponents in the right-hand side of $[\Dset]$ produced by a 1-variable substitution are translated powers of $K$ (by the same constant), then the substitution instance is trivial.
\end{center}
In symbolic terms this can be written as
\begin{equation}\label{singlestar}\tag{$\star K$}
\begin{array}{c}
\mbox{If for some $\rvec\in\N^n$ and  $\const\in \N,$}\\
\mbox{every $\const+\rvec{\vd}$ is a power of $K$, where $\vd\in \Dset$,}\\ 
\mbox{then there exists $\bar \vd\in \Dset$ such that $\rvec{\bar \vd}= \rvec{\mathbf{1}}$}\end{array}
\end{equation}
in which case we say that $[\Dset]$ satisfies $(\star K)$. We also consider the condition $(\star)$: there exists $K>1$ such that $(\star K)$ holds. 

In the following sections we will make rigorous the notion of admissibility and carefully construct the machines ${\acm}_K$.
%
%Subsection: Spineless equations
%
\subsection{Spineless equations}
We will now define a class of simple equations, for which their defining subvarieties will have an undecidable word problem. The class is so vast that it is easier to define its complement. We motivate the definition with the following observation.

Consider the machine ${\acm}_{\mathrm{even}}$ from Example~\ref{ex: Meven}, and the simple equation ${\Ds}: x\leq 1\vee x^2$. As before, it is easy to see that $q_0\reg_1^3\in {\Acc}({\Ds} {\acm}_{\mathrm{even}})\setminus {\Acc}({\acm}_{\mathrm{even}})$. However, this behavior cannot be remedied by ${\acm}_K$ for any $K>1$. E.g., let $n = (K^4-K^2)/2$, then $q_0\reg_1^{K^2+n}\nin {\Acc}({\acm}_K)$ since $K^2+n\neq K^{2m}$ for any $m\in\N$. However,
$$q_0\reg_1^{K^2+n}=q_0\reg_1^{K^2}\reg_1^n\leq^{\Ds} q_0\reg_1^{K^2}\reg_1^{0}\vee q_0\reg_1^{K^2}\reg_1^{2n}= q_0\reg_1^{K^2}\vee q_0\reg_1^{K^4}\in {\Acc}({\acm}_K).  $$ 
By setting $\const=K^2$ and $\rvec=(n)\in\N^1$, this also demonstrates that $(\star K)$ is not satisfied by ${\Ds}$ for any $K>1$, i.e., ${\Ds}$ fails $(\star)$. In fact, we will show in Lemma~\ref{injective} that this failure occurs not just for functions of the form $n\mapsto K^n$ but actually for any function on $\N$ with infinite range (in particular those that are actually computable).

The equation ${\Ds}$ is an example of a {\em spinal equation} and as explained above, unfortunately it cannot be handled by our work. In general, a basic equation $[f,\spine]$ is called {\em spinal} if it is of the form:
$$[f, \spine]: \underbrace{x_1^{f(1)}\cdots x_k^{f(k)}}_{f} \leq \underbrace{(1\vee{})}_{\vs_0}~\underbrace{x_1^{\vs_1(1)}}_{\vs_1}\vee \underbrace{x_1^{\vs_2(1)}x_2^{\vs_2(2)}}_{\vs_2}\vee\cdots \vee \underbrace{x_{1}^{\vs_k(1)}\cdots x_k^{\vs_k(k)} }_{\vs_k},$$
where $f\nin \spine$, $\vs_j(j)\neq 0$ and $\vs_i(j)=0$ for each $0\leq i<j\leq k$, and $(1\vee)$ is meant to signify that $1$ may or may not be included in the join. Note that if the column vectors of the set $\spine=\{(\vs_0,) \vs_1, \ldots, \vs_k\}$ are listed in the above order, $\spine$ becomes an upper-right triangular matrix whose diagonal contains only positive entries. In Corollary~\ref{3to1} we establish that spinal equations falsify the condition $(\star K)$ for every $K>1$. 
\begin{defn}\label{def: spinal} 
We say that a basic equation $[f,\spine]$ is {\em spinal} if $f\nin \spine$, $\vs_i(i)\not =0$ and $\vs_i(j)=0$, for all $j > i$, for all $\vs_i$ in $\spine$ that are not constantly zero. In this case, we will refer to the set $\spine$ as a {\em spine}.
We say that a simple equation is {\em prespinal} if it has a spinal equation as the image under some monoidal substitution.
\end{defn}

From now on we will only consider monoidal substitutions (the image of every variable is a monoidal term, i.e. a product of variables) and we will refer to them simply as substitutions. 

Note that the only one-variable spinal equations are the knotted inequalities $x^n \leq x^m$ (for which we know that they define subvarieties of $\CRL$ with decidable quasiequational theory) and their variants  $x^n \leq 1 \jn x^m$, where $n\neq m$, for which decidability results are still open.\footnote{The only exception being the case where $n>m=1$, where equations of this form have the finite model property by Theorem 3.15 in \cite{GJ}.} Also, their equivalent simple equations are prespinal, as verified below in Lemma~\ref{und1var}. As a consequence of the definition, a simple equation $[\Dset]$ is prespinal if and only if $[\Dset\cup\{{\Zero}\}]$ is prespinal.

From Table~\ref{Rtable}, we see that (i)-(ii) are spinal. The simple equation (vi) is prespinal via the $1$-variable substitution $\subt$ given by $\subt(x)=x$, $\subt(y)=x$ and $\subt(z)=1$.
i.e., $\CRL +\mbox{(vi)}\models x^2\leq x$. On the other hand, no trivial equations are prespinal. The general characterization of whether a simple equation is prespinal will be addressed in \S\ref{charD}, where it is verified that (iii)-(v) in Table~\ref{Rtable} are not prespinal right after Theorem~\ref{spinalform}.
\begin{defn}\label{def:spineless}
A simple equation is called \emph{spineless} if it is not prespinal. A simple equation $\varepsilon$ for $\RL $ is called spineless if $\varepsilon_\comeq$ is spineless.
\end{defn}

To demonstrate the vastness of the collection of spineless equations, we will focus our attention only on $1$-variable basic equations below. Note that every one-variable basic equation has the form
$$ x^n\leq \bigvee_{p\in P} x^p $$
for some $n>0$ and for some finite subset $P$ of $\N$ such that $P\not = \{0\}$; we denote such an equation by $[n,P]$. Also, note that $[n,P]$ is trivial iff $n \in P$.
\begin{lem}\label{und1var}
Let $[n,P]$ be a $1$-variable basic equation. Then the linearization of $[n,P]$ is a spineless simple equation iff $[n,P]$ is trivial or $P$ contains at least two distinct positive integers.
\end{lem}
\begin{proof} 
The simple equation resulting from the linearization over $\CRL$ of $[n,P]$ is
\begin{equation}\label{1varLin}
[\Dset]: \mathbf{x_n}^\Lvec{}\leq \bigvee\left\{\mathbf{x_n}^d:{d\in \N^n}, \sum_{i=1}^n d(i) \in P \right\}.
\end{equation}
and is equivalent over $\CRL$ to $ [n,P]$.\footnote{By setting $x:= x_1\vee\cdots x_n$ in $[n,P]$, we obtain
$\bigvee\{\mathbf{x_n}^a: \sum_{i=1}^na(i)=n \} \leq \bigvee\left\{\mathbf{x_n}^d:{d\in \N^n}, \sum_{i=1}^n d(i) \in P \right\}$.
 Since $\mathbf{x_n}^\Lvec{} \leq  \bigvee\{\mathbf{x_n}^a: \sum_{i=1}^na(i)=n \}$, we obtain $[\Dset]$. Conversely,  by setting $x_i:=x$, for all $i\leq n$, in $[\Dset]$, we obtain $[n,P]$.}

We prove the contraposition for each direction. For the forward direction, if $P=\{p\}$, where $0<p \not = n$, then $[n,P]$ is spinal by definition and hence $[\Dset]$ is prespinal by its obvious substitution to $[n,P]$: $x_i\mapsto x$, for all $i$. 

For the reverse direction, suppose that $[n,P]$ is nontrivial and $P$ contains distinct positive numbers $p>q$. Then for each $i\leq n$, the terms $x_i^p$ and $x_i^{q}$ appear as joinands on the right-hand side of $[\Dset]$. Let $\subt$ be a monoidal substitution that is non-trivializing for $[\Dset]$. Then for some $i\leq n$, $\subt: x_i\mapsto w$ for some monoid term $w\neq 1$. So both $w^p$ and $w^{q}$ appear as joinands on the right-hand side of $\subt[\Dset]$. Since $p>q>0$, $w^p\neq 1\neq w^{q}$ and $\CRL\nmodels w^p= w^{q}$, so $\subt[\Dset]$ is not spinal, as $w^p$ and $w^q$ contain the same variables. Since $\subt$ was arbitrary, it follows that $[n,P]$ is spineless.
\end{proof}

In the following sections we undertake a deep analysis of spineless equations, culminating in Corollary~\ref{UndLink}. To highlight this result, we (re)state it here and we use it right afterwards, in Section~\ref{s:eqth}, to obtain results (Theorem~\ref{main2}) about the equational theory. 
\begin{thm}\label{main}
Let $\Gamma$ be a finite set of spineless simple equations, then any variety between $\CRL +\Gamma$ and $\RL $ has an undecidable word problem.
\end{thm}
%
%Subsection: From quasiequations to equations in $\CRL$
%
\subsection{From quasiequations to equations in $\CRL$}\label{s:eqth}
In this section we exploit the fact that in certain varieties certain quasiequations are equivalent to equations to show that even their equational theory is undecidable, making use of Theorem~\ref{main}.

The {\em negative cone} of a residuated lattice $\mathbf{A}$ is the set $A^-=\{a\in A: a\leq 1\}$. We will say that a subvariety $\mathcal{V}$ of $\CRL $ is \textit{negatively $n$-potent} if the negative cone of each algebra in $\mathcal{V}$ is $n$-potent, i.e., $\mathcal{V}\models (x\wedge 1)^n= (x\wedge1)^{n+1}$ (or equivalently, $\mathcal{V}\models (x\wedge 1)^n\leq (x\wedge1)^{n+1}$).
 
Let $t$ be a term and $S$ be a finite set of terms in the language of $\CRL $. It can be easily verified that\footnote{The forward direction is trivial, taking $k=m$,  since $\bigwedge S \leq s$, for all $s\in S$. The reverse direction holds by setting $m=k\cdot|S|$, and observing that $\prod_{s\in S}(s\wedge 1)\leq 1\wedge \bigwedge S$.}
\begin{equation}\label{prodjoin}
\begin{array}{r c l} 
(\exists m\in\N)(\exists s_1,...,s_m\in S) ~\CRL &\models&\prod_{i=1}^m (1\wedge s_i)\leq t \\ 
\mbox{if and only if}\quad (\exists k\in\N)~\CRL &\models& (1\wedge\bigwedge S)^k\leq t.
\end{array}
\end{equation}
If $\mathcal{V}\subseteq\CRL $ is a negatively $n$-potent variety, then we obtain\footnote{The reverse direction follows from (\ref{prodjoin}), while the forward direction uses the fact that ${(1\wedge x)^n}\leq (1\wedge x)^k$, if $k\leq n$, and $(1\wedge x)^n=(1\wedge x)^k$, if $k>n$, by the negative $n$-potency of $\mathcal{V}$.} 
\begin{equation}\label{prodjoin2}
(\exists m\in\N)(\exists s_1,...,s_m\in S) ~\mathcal{V}\models\prod_{i=1}^m (1\wedge s_i)\leq t  \iff\mathcal{V}\models (1\wedge\bigwedge S)^n\leq t.
\end{equation}
We consider the following quasiequation and equation:
$$\xi_S(t):  \amper_{s\in S} 1\leq s \implies 1\leq t \qquad \qquad  \varepsilon_S^n(t): (1\wedge \bigwedge S)^n\leq t. $$

In this way we establish the fact that satisfaction of a quasiequation in a negatively $n$-potent subvariety of $\CRL $ is equivalent to the satisfaction of a corresponding equation.
\begin{lem}\label{lem:nnpotent}
If $\mathcal{V}$ is a negatively $n$-potent subvariety of $\CRL $ and $S\cup \{t\}$ a finite set of terms in the language of $\mathcal{V}$, then 
$$\mathcal{V}\models \xi_S(t) \iff \mathcal{V}\models \varepsilon_S^n(t).$$
\end{lem}
\begin{proof}
Let $\Fv$ be the free algebra for $\mathcal{V}$, and define the congruence $\const:=\mathrm{Cg}(\{(1\wedge s, s): s\in S \})$. We denote the quotient algebra by $\Fv/\const$. For a subset $X$ of $\mathrm{F}_\mathcal{V}^-$ , we denote by $M(X)$ the convex normal submonoid of $\mathrm{F}_\mathcal{V}^-$ generated by $X$.\footnote{See Theorem 3.47 in \cite{GJKO}.} Observe that $\mathcal{V}\models \amper\limits_{s\in S} 1\leq s \Rightarrow 1\leq t $
$$\begin{array}{ l c l l }
& \iff & \mbox{in $\Fv/\const$, } [1\wedge t]_\const = [1]_\const & \\
& \iff & \mbox{in $\Fv$, } (1\wedge t)\in M(\{ 1\wedge s: s\in S \}) & \mbox{\cite{GJKO}}\\ 
&\iff &\mbox{in $\Fv$, }  (\exists m\in\N)(\exists s_1,...,s_m\in S) ~\prod_{i=1}^m (1\wedge s_i)\leq t &\mbox{\cite{GJKO}}\\
&\iff &(\exists m\in\N)(\exists s_1,...,s_m\in S) ~ \mathcal{V}\models \prod_{i=1}^m (1\wedge s_i)\leq t &\\
&\iff & \mathcal{V}\models (1\wedge\bigwedge S)^n\leq t &\mbox{{Eq. (\ref{prodjoin2})}}.
\end{array} $$
\end{proof}

For an inequality $p: s\leq t$, define the term $p^\to:= s\to t$. Let ${\acm}=(\Reg_k,{\State},{\Inst},q_f)$ be an ACM. Define ${\Inst}^\to:=\{p^\to: p\in {\Inst}\}.$ Then for ${u}\in A_{{\acm}}$, the quasiequation $\mathrm{acc}_{\acm}({u})$ is equivalent to $\xi_{{\Inst}^\to}({u}\to q_f)$. By Lemma~\ref{lem:nnpotent} and Theorem~\ref{Vmhard}, we obtain the following:
\begin{thm} 
Let $\mathcal{V}$ be a subvariety of $\CRL $ containing $\mathbf{W}_{\acm}^+$, for some $k$-ACM ${\acm}$ and satisfying $(x\wedge 1)^n\leq (x\wedge 1)^{n+1}$ for some $n\geq 1$. Then deciding membership in the equational theory of $\mathcal{V}$ is at least as hard as deciding membership in ${\Acc}({\acm})$.
\end{thm}

We say a simple equation $\varepsilon$ is \textit{expansive} if it has, as a substitution instance, an equation of the form
\begin{equation}\label{exp}x^n \leq \bigvee_{j=1}^m x^{n+c_j}, \end{equation}
for some $n,m\geq1$ and $c_1,...,c_n\geq1$. It is easy to verify that if $\varepsilon$ is expansive then $\CRL+\varepsilon$ is negatively $n$-potent. We say a variety  is {\em expansive} if it satisfies an expansive equation. As a consequence of Lemma~\ref{und1var}, if a simple equation is the equivalent linearization of an expansive basic equation where $m\geq 2$ then it is spineless. By the theorem above and Theorem~\ref{Vmhard} we obtain the following:
\begin{cor}\label{undeq}
Let $\mathcal{V}$ be an expansive subvariety of $\CRL $ containing $\mathbf{W}_{\acm}^+$, for some ACM ${\acm}$. Then deciding membership in the equational theory of $\mathcal{V}$ is at least as hard as deciding membership in ${\Acc}({\acm})$.
\end{cor}
 
In particular, we prove the following theorem as a consequence of Theorem~\ref{main} and the corollary above.
\begin{thm} \label{main2}
If $\Gamma$ is a finite set spineless simple equations containing an expansive equation then variety ${\CRL +\Gamma}$ has an undecidable equational theory.
\end{thm}
%
%SECTION: Admissibility
%
\section{Admissibility}
We now begin investigating the required features that a machine should have, in  order to achieve the exponential encoding. We begin by formalizing the notion of admissibility and its two natural parts. 

Let ${\acm}=(\Reg_k,{\State},{\Inst}, q_f)$ be a $k$-ACM and $[\Dset]$ a $n$-variable simple equation. We define the relation $\leq^\Dset$ to be the smallest relation containing
$$\mathbf{t_n}^\mathbf{1} \leq \bigvee_{\vd\in\Dset}\mathbf{t_n}^\vd, $$
for all $\mathbf{t_n}\in (({\State}\cup \Reg_k)^*)^n$, and closed under the inference rules $[\cdot]$ and $[\vee]$. We define the computation relation $\leq_{\Dset {\acm}}$ as the smallest compatible preorder generated by ${\Inst}\cup{\leq}^\Dset$, and set ${\Acc}(\Dset {\acm}):=\{u\in A_{\acm}: \exists {\uid}_f\in {\Fid}({\acm}),~u\leq_{\Dset {\acm}} {\uid}_f\}.$ 

The construction of $\leq_{\Dset\acm}$ enjoys an analogue to Lemma~\ref{basicrel}(2), and therefore the following analogue to Lemma~\ref{comprel}(2):
\begin{lem}\label{Dcomprel}
Let ${\acm}$ be an ACM and $[\Dset]$ be a simple equation. For all ${u},{v}\in A_\acm$, ${u}\vee {v}\in {\Acc}({\Dset\acm})$ if and only if ${u}\in{\Acc}({\Dset\acm})$ and ${v}\in{\Acc}({\Dset\acm})$.
\end{lem}

The frame  $\mathbf{W}_{\Dset {\acm}}$ is defined as $\mathbf{W}_{ {\acm}}$,  but the nuclear relation is defined with respect to ${\Acc}({\Dset\acm})$ instead of ${\Acc}({\acm})$.
\begin{lem}\label{Wsatsr}
If ${\acm}$ is an ACM and $[\Dset]$ a simple equation, then $\mathbf{W}_{\Dset {\acm}}^+\in \CRL +[\Dset]$.
\end{lem}
\begin{proof} Let $[\Dset]$ be an $n$-variable simple equation where $\Dset=\{\vd_1, \ldots, \vd_m\}$. 
It is enough to show that $\mathbf{W}_{\Dset {\acm}}^+\models [\Dset]$.
By Theorem~\ref{nucond}, this is equivalent to showing $\mathbf{W}_{\Dset \acm}\models (\Dset)$, i.e., for all $s \in W$, $\mathbf{t_n}\in W^n$,
$$\infer[{(\Dset).}] {\mathbf{t_n}^{\mathbf{1}} \NucS{\Dset {\acm}} s }{\mathbf{t_n}^{\vd_1} \NucS{\Dset {\acm}} s& \cdots &\mathbf{t_n}^{\vd_m} \NucS{\Dset {\acm}} s} $$

If the antecedent of the implication holds, by the definition of $\mathit{N}_{\Dset {\acm}}$ and $\leq_{\Dset {\acm}}$ and by Lemma~\ref{Dcomprel} we obtain
$$(\forall \vd\in \Dset)~\mathbf{t_n}^{\vd} \NucS{\Dset {\acm}} s 
	\iff
	(\forall \vd\in \Dset)~ s\mathbf{t_n}^\vd\in{\Acc}{(\Dset {\acm})} 
	 \iff \bigvee_{\vd\in \Dset}s\mathbf{t_n}^\vd\in{\Acc}{(\Dset {\acm})}
.$$

Now, by the definition of $\leq_{\Dset {\acm}}$, 
$$s\mathbf{t_n} \leq_{\Dset {\acm}}s\bigvee_{\vd\in\Dset}\mathbf{t_n}^\vd=\bigvee_{\vd\in\Dset} s\mathbf{t_n}^\vd ,$$
hence $s\mathbf{t_n} \in {\Acc}{(\Dset {\acm})} $ and $\mathbf{t_n} \NucS{\Dset {\acm}} s$. Therefore $\mathbf{W}_{\Dset {\acm}}^+ \models [\Dset].$
\end{proof}

Since ${\leq}_{\acm} \subseteq {\leq}_{\Dset {\acm}}$, it follows that ${\Acc}({\acm})\subseteq {\Acc}(\Dset{\acm})$. We say a simple equation $[\Dset]$ is \textit{admissible} in a machine ${\acm}$ if ${\Acc}({\acm}) = {\Acc}(\Dset {\acm})$. As the only difference between $\mathbf{W}_{ {\acm}}$ and $\mathbf{W}_{\Dset {\acm}}$ is ${\Acc}({\acm})$ and ${\Acc}(\Dset {\acm})$, if $[\Dset]$ is admissible in $\acm$ then $\mathbf{W}^+_{\acm} = \mathbf{W}^+_{\Dset {\acm}}$. Therefore by Lemma~\ref{Wsatsr} we obtain the following lemma:
\begin{lem}\label{Wadmiss}
If a simple equation $[\Dset]$ is admissible in ${\acm}$, then $\mathbf{W}^+_{\acm} \in \CRL+[\Dset].$
\end{lem}

As we will see, admissibility in ${\acm}$ depends on the machine ${\acm}$ as well as the equation $[\Dset]$. We define the intermediate notions \textit{register-admissibility} and \textit{state-admissibility} to make this distinction clear. 

For a given ACM ${\acm}=(\Reg_k,{\State},{\Inst},q_f)$ and $n$-variable simple equation $[\Dset]$, we define $\leq^{\regeq{\Dset}}$ to be the smallest relation containing, 
$$\mathbf{x_n}^\mathbf{\Lvec{}} \leq \bigvee_{\vd\in\Dset}\mathbf{x_n}^\vd, $$
for all $\mathbf{x_n}\in (\Reg_k^*)^n$, and closed under the inference rules $[\cdot]$ and $[\vee]$. Define the new computation relation $\leq_{{\regeq{\Dset}} {\acm}}$ to be the smallest compatible preorder generated by ${\Inst}\cup{\leq}^{{\regeq{\Dset}}}$, and set ${\Acc}({\regeq{\Dset}} {\acm}):=\{u\in A_{\acm}: \exists {\uid}_f\in {\Fid}({\acm}),~u\leq_{{\regeq{\Dset}} {\acm}} {\uid}_f\}$. Since the effect of $[\Dset]$ is restricted to only register terms in $\leq^{\Dset\Reg}$, by the same argument as Lemma~\ref{comprel}(1), it follows that 
${\Acc}(\regeq{\Dset} {\acm})\subseteq \indes {\acm}$.

It is clear then that ${\Acc}({\regeq{\Dset}} {\acm})\subseteq {\Acc}(\Dset {\acm})$, since ${\leq}^{\regeq{\Dset}}$ is merely a restriction of ${\leq}^{\Dset}$. In total, we obtain
$${\Acc}({\acm})\subseteq{\Acc}({\regeq{\Dset}} {\acm})\subseteq {\Acc}(\Dset {\acm}).$$
If ${\Acc}({\acm})={\Acc}({\regeq{\Dset}} {\acm})$, then we say $[\Dset]$ is \textit{register-admissible in ${\acm}$}. If ${\Acc}({\regeq{\Dset}} {\acm})= {\Acc}(\Dset {\acm})$, we say $[\Dset]$ is \textit{state-admissible in ${\acm}$}. Hence $[\Dset]$ is admissible in  ${\acm}$ iff $[\Dset]$ is both register and state-admissible in ${\acm}$.~
Due to the property that instructions in an ACM replace a single state-variable of a configuration by precisely one state-variable, we show  in Lemma~\ref{mingly} that state-admissibility is a property of the equation $[\Dset]$ independent of the machine ${\acm}$.
%
%Subsection: State-admissibility for spineless equations 
%
\subsection{State-admissibility for spineless equations}
We say that a simple equation $[\Dset]$ is \textit{mingly} if there exists a one-variable substitution $\subt$ such that $\subt [\Dset] : x^\lambda \leq \bigvee_{\vd \in \Dset} x$ for some $\lambda>1$. That is, if $[\Dset]$ is an equation in $n$-variables, $\subt (\m{x_n}^{\m 1})=x^\lambda$ and $\subt (\m{x_n}^{\vd})=x$, for all $\vd \in \Dset$.

Of course, this is equivalent over $\RL$ to $x^\lambda \leq x$, but since we do not assume idempotency of $\jn$ in $\m A_\acm$, we write $ x^\lambda \leq \bigvee_{\vd \in \Dset} x$ so as to be explicit about the implementation of the equation in computations. 

By definition, mingly equations are prespinal. 
The equation (vi) from Table~\ref{Rtable} is mingly, using the substitution witnessing it is prespinal, while the remaining equations can easily be verified to be not mingly.\footnote{Note that the remaining equations are such that each variable appears with degree at least $2$ on the right-hand side, so any $1$-variable non-trivializing substitution instance will result in a joinand of degree at least $2$. }  
Since mingly equations are prespinal by definition, we obtain the following result.
\begin{lem}\label{nming}
A spineless equation is non-mingly.
\end{lem}
As we will be dealing only with spineless equations, the equations we will consider are not mingly. The following lemma shows that only mingly equations invalidate state-admissibility and also that state-admissibility is independent of the choice of the machine.%does not %
\begin{lem}\label{mingly}
The following are equivalent for any ${\acm}$ and simple equation $[\Dset]$.
\begin{enumerate}
	\item $[\Dset]$ is state-admissible in ${\acm}$.
	\item ${\Acc}{(\Dset {\acm})}\subseteq \indes {\acm}$.
	\item $[\Dset]$ is not mingly.
\end{enumerate}
\end{lem} 
\begin{proof} Assume $\acm=(\Reg_k,\State,\Inst,q_f)$ and let $[\Dset]$ be an $n$-variable simple equation.\\
\noindent
($1\Rightarrow2$) We have that ${\Acc}(\Dset {\acm})={\Acc}(\regeq{\Dset} {\acm})\subseteq \indes {\acm}$. \\
\noindent
($2\Rightarrow3$) Proceeding by contraposition, suppose $[\Dset]$ is mingly. 
Then  $x^\lambda \leq \bigvee_{\vd \in \Dset} x$ is a direct substitution image of $[\Dset]$, for $\lambda >1$. For $x=q_f$, we have 
 $$q_f^{\lambda}\leq^\Dset \bigvee_{\vd\in\Dset} q_f \in{\Fid}({\acm})\subseteq{\Acc}(\Dset {\acm}).$$ Since $\lambda>1$, it follows that $q_f^{\lambda}\nin \indes {\acm}.$ 

\noindent
($3\Rightarrow1$)  Proceeding by contraposition, suppose ${\Acc}({\regeq{\Dset}} {\acm})$ is a proper subset of ${\Acc}(\Dset {\acm})$ and let $t\in {\Acc}(\Dset {\acm})\setminus {\Acc}({\regeq{\Dset}} {\acm})$ be a witness with minimal computation 
$$t=u_0\leq^{p_1}u_1\leq^{p_2} \cdots\leq^{p_N} u_N ={\uid}_f\in {\Fid}({\acm}), $$
for some $u_0, u_1,...,u_N\in A_{\acm}$ and $p_1,...,p_N\in {\Inst}\cup \{\Dset\}$.
 Since ${\Fid}({\acm})\subseteq  {\Acc}({\regeq{\Dset}} {\acm})$, we have that $t\nin {\Fid}({\acm})$ and so $N>1$.
By Lemma~\ref{Dcomprel}, we may assume $t\in ({\State}\cup \Reg_k)^*$. Since $N$ is minimal, it follows that $p_1 =\Dset $ and $u_1\in {\Acc}({\regeq{\Dset}} {\acm})\subseteq \indes {\acm}$. So,
$$t= s\mathbf{t_n}^\mathbf{1} \leq^\Dset \bigvee_{\vd\in\Dset}  s\mathbf{t_n}^\vd=u_1. $$
where $s\in({\State}\cup \Reg_k)^*$ and $\mathbf{t_n}\in (({\State}\cup \Reg_k)^*)^n$; here we used the fact that the rule $[\jn]$ is not applicable, as $t\in ({\State}\cup \Reg_k)^*$.

Since $u_1\in \indes {\acm}$, it follows that $s\mathbf{t_n}^\vd\in \conf {\acm}$ for all $\vd\in\Dset$. 
Also, because $p_1\neq \Dset\Reg$,  there is some $t_i$ that contains at least one state variable. As that $t_i$ has to also appear on the right-hand side and $s\mathbf{t_n}^\mathbf{1} \in \conf {\acm}$, it follows that $s$ cannot contain any state variable, so $s \in  \Reg_k^*$.
Consequently, every joinand in the right-hand side has a unique $t_i$ containing a state variable.
 
Therefore, applying the substitution $\subt$ defined by: $\subt(x_i)=x$ if $t_i$ contains a state variable and $\subt(x_i)=1$ otherwise, yields $x^\lambda \leq x$, where $\lambda$ is the number of $t_i$'s containing a state variable.

We will show that $\lambda> 1$. If, by way of contradiction, there was a unique $t_j$ containing a state variable, then it would have the form  $t_j=qx$ for some $q\in {\State}$ and $x\in \Reg_k^*$ (with $t_i\in \Reg_k^*$ for all other $t_i$'s) and $t_j$ would appear on all the joinands on the right-hand side. So, 
we have 
$$t=sqx\prod_{i\not = j} t_i \leq^{\mathrm{\regeq{\Dset}}}\bigvee_{\vd\in\Dset}  sqx\prod_{i\not = j}^n{ t_i}^{\vd(i)}=u_1\in {\Acc}(\regeq{\Dset} {\acm}),$$
and thus $t\in  {\Acc}(\regeq{\Dset} {\acm}),$ a contradiction.
\end{proof}

Therefore, in our search for an appropriate machine for spineless equations, state-admissibility will be automatic and will not restrict the type of possible machines.
%
%SECTION: The exponential encoding
%
\section{The exponential encoding}\label{mksec}
Given a $2$-ACM ${\acm}=(\Reg_2,{\State},{\Inst}, q_f)$ (we will later choose as ${\acm}$ a machine with undecidable halting problem)  and simple equation $[\Dset]$, our ultimate goal is to construct a new machine ${\acm}'$ ``simulating'' the machine ${\acm}$ such that ${\Acc}(\Dset {\acm}' ) = {\Acc}({\acm}')$. More specifically, for any spineless equation $[\Dset]$ we can construct a 3-ACM ${\acm}_K=(\Reg_3,{\State}_K,{\Inst}_K, q_F)$ for which it will be register-admissible, for some $K>1$ provided by Theorem~\ref{final}, that will simulate the behavior of the 2-ACM ${\acm}$ in the following way:
$$q\reg_1^{n_1}\reg_2^{n_2}\in {\Acc}({\acm}){\mbox{ if and only if }} q\reg_1^{K^{n_1}}\reg_2^{K^{n_2}} \in{\Acc}({\acm}_K).$$

So, the content $n$ of a register $\reg$ is not represented by $\reg^{n}$ but by $\reg^{K^n}$. Computationally, this will be achieved by replacing each increment-$\reg$ (decrement-$\reg$) instruction by a distinct set of instructions (i.e., a {\em program}) that will carry out the process of multiplying (dividing) the contents in register $\reg$ by $K$. The auxiliary register $\reg_3$ will be necessary to carry out such computations, and the faithfulness of this encoding will be guaranteed by the insistence that accepted configurations are those which have a computation resulting in a finite join of configurations labeled by state $q_F$ where all the registers are empty. 

We will first show how we can achieve multiplying the contents of a register $\reg \in \Reg_2$ by a fixed constant $K>1$. Let $\State_{+K}=\{a_0,...,a_K\}$ be a set of $(K+1)$-many fresh state-variables. We can add $K$ tokens to the contents of the auxiliary register $\reg_3$ by starting from the state $a_0$ and applying the instructions $\Inst_{+K}=\{+_i : i=1,...,K \}$, where $+_i: a_{i-1}\leq a_i\reg_3$, and reaching state $a_K$. Then for each $0\leq i<K$, we have 
$$a_0\leq_{+K} a_i\reg_3^N \iff N=i ,$$
where $\leq_{+K}$ is the computation relation defined in the usual way from $\Inst_{+K}$. Hence $a_0x\leq_{+K} a_Ky$ iff $y=x\reg_3^K$, for each $x,y\in \Reg_3^*$. 

Now, if we have $N$ tokens in register $\reg$ and we are at state $a_K$, then by removing one $\reg$-token, moving to state $a_0$ by the  instruction $\times_{\mathrm{loop}}: a_K\reg\leq a_0$, and using $\Inst_{+K}$ to add $K$  $\reg_3$-tokens and repeating this process, we can essentially exchange $N$ $\reg$-tokens for  $NK$  $\reg_3$-tokens. E.g.,
$$a_K\reg^N \leq^{\times_{\mathrm{loop}}}a_0\reg^{N-1}\leq_{+K}a_K\reg^{N-1}\reg_3^{K}\leq^{\times_{\mathrm{loop}}}a_0\reg^{N-2}\reg_3^{K}\leq_{+K}\cdots ~a_K\reg_3^{NK}.$$
If we set $\leq_{\times\reg}$ to be the computation relation defined from $\Inst_{\reg\times K}=\Inst_{+K}\cup\{\times_{\mathrm{loop}} \}$, then it is easily verified by induction that for each $N,M,n,m\in \N$ and $0\leq i \leq K $:

\begin{equation}\label{eq:TimesK}
\begin{array}{r c l}
a_K \reg^N \leq_{\times\reg} a_i \reg^n \reg_3^m &\iff & KN=Kn + m + (K-i)\\
a_i\reg^n \reg_3^m\leq_{\times\reg} a_K\reg_3^M &\iff& M=Kn+m + (K-i)
\end{array}
\end{equation}
Observe that multiplying the contents of register $\reg$ is achieved by an iterative process of adding $K$-many tokens to $\reg_3$ and then looping the process by removing one from $\reg$. We can define a division program analogously. However, in both cases, we start with tokens in $\reg$ and compute the product (or quotient) by $K$ in the register $\reg_3$ by emptying $\reg$. We would then like our machine to {\em transfer} those contents back to the original register $\reg$ to complete this program.

Let $\Trsf_\reg$ be the program with fresh states $\State_{\Trsf_\reg}=\{t_0,t_1 \}$ and instructions $\Inst_{\Trsf_\reg}=\{T_-,T_+ \}$ given by $T_-: t_0\reg_3\leq t_1$ and $T_+: t_1\leq t_0\reg$. Defining its computation relation to be $\leq_{\Trsf_\reg}$, we easily obtain for all $N,M,n,m\in \N$:
\begin{equation}\label{eq:TransK}
\begin{array}{r c l}
t_0 \reg_3^N \leq_{\Trsf_\reg} t_\delta \reg^n\reg_3^m &\iff &N=n+m+\delta \\
t_\delta \reg^n\reg_3^m\leq_{\Trsf_\reg} t_0\reg^M&\iff& n+m+\delta=M
\end{array}
\end{equation}
If we wish to implement the transfer program after, say, executing the program for multiplying by $K$, then we need an instruction to switch from the state $a_0$ to the state $t_0$. This can naively be achieved by a forking instruction, say $p:a_K\leq t_0$,\footnote{Technically, $p: a_0\leq t_0\vee t_0$ since $\vee$ is not idempotent, but this technicality is unnecessary for the example.} which would allow for the computation 
$$a_K\reg^N \leq_{\times \reg}a_K\reg_3^{NK} \leq^p t_0\reg_3^{NK} \leq_{\Trsf_\reg}t_0\reg^{NK}. $$
However, since the instruction $p$ can be applied to any configuration labeled by $a_K$, even those for which the register $\reg$ is nonempty, we also get unwanted instances of the form $a_K\reg^N\leq t_0\reg^{M}$ where $M=Km+N-m$ for each $0\leq m\leq N$. Since we want our simulation to be faithful, we need a way of switching to the transfer program {\em only when} the register $\reg$ is empty.
%
%Subsection: The Zero-Test Program
%
\subsection{The Zero-Test Program}
What we are asking for is similar to (what is commonly called) a zero-test instruction of a standard counter machine, i.e., an instruction which is applicable only when a specified register is empty. Since we require our computation relations to be compatible with multiplication, such an insistence is impossible. I.e., we insist $q\leq q'$ to entail $qx\leq q'x$ for any term $x$. 

However, following the ideas in \cite{Lincoln}, we construct a program that has a similar behavior utilizing the insistence that accepted configurations are those that compute final ID's, i.e., finite joins of the configuration labeled by a final state $q_F$ with all registers empty.

We define the (sub-)machine ${\o}=(\Reg_3,{\State}_{\o}, {\Inst}_{\o},q_F)$, with a fresh set of variables ${\State}_{\o} = \{z_1,z_2,z_3,q_F\}$ and instructions ${\Inst}_{\o}$ are given by: 
$$\begin{array}{ l c r c l }
{{\o}^i_j}&:&z_i \reg_j&\leq& z_i \\
{{\o}^i_F}&:& z_i&\leq& q_F\vee q_F
\end{array},$$
for each $i, j\in\{1,2,3\}$ with $i \not = j$, resulting in a total of $6+3=9$ instructions.

The addition of the auxiliary states $z_1, z_2, z_3$ is explained by the fact that the role of $z_i$ is to empty the contents of all registers other than $\reg_i$ and transition to the final state $q_F$. Thus it detects situations where $\reg_i$ is not already empty; as then it cannot reach a final ID. So, assuming we want to move from state $\qin$ to $\qout$ only when the register $\reg_i$ is empty, we can start a parallel computation, the main branch of which moves to state $\qout$ (even if $\reg_i$ is non-empty) but the auxiliary branch involving the $z_i$ terminates successfully only if $\reg_i$ is empty. This $z_i$-branch ensures/safeguards that the combined computation acts as intended. 

We call the above machine the \textit{zero-test program}, and we denote its computation relation by $\leq_{\o}$. The zero-test program for a register $\reg_i$ is implemented by a \textit{zero-test $\reg_i$ instruction $p$}, where $p$ is of the form $\qin \leq \qout\vee z_i$. Since the desired final ID's of ${\acm}_K$ will consist of only joins of the configuration $q_F$, i.e. all registers are empty, the above instruction copies the contents of the registers and creates two paths; one path with the state $\qout$ where $\reg_i$ is intended to be empty, and the second with a state $z_i$ where the program ${\o}$ is intended to empty registers $\reg_j$ and $\reg_k$ and then output to the final state. 
Below is an example of implementing the zero-test on register $\reg_1$ via the instruction $p:\qin\leq \qout\vee z_1$  on the configuration $\qin \reg_1\reg_2\reg_3$:
$$\begin{array} {r l l}
\qin \reg_1\reg_2\reg_3 &\leq^{p}&\qout \reg_1\reg_2\reg_3\vee z_1\reg_1\reg_2\reg_3 \\
&\leq^{{\o}^1_2}&\qout \reg_1\reg_2\reg_3\vee z_1\reg_1\reg_3 \\
&\leq^{{\o}^1_3}&\qout \reg_1\reg_2\reg_3\vee z_1\reg_1 \\
&\leq^{{\o}^1_F}&\qout \reg_1\reg_2\reg_3\vee q_F\reg_1 \vee q_F\reg_1 .
\end{array}$$

As we see, the above (maximal in $\o)$ computation detected that register $\reg_1$ is not empty in the configuration $q\reg_1\reg_2\reg_3$ since the final ID contains the configuration $q_F\reg_1$, and there are no $q_F$-instructions. In fact, $z_1\reg_1\reg_2\reg_3\nin{\Acc}(\o)$ since there is no instruction applicable to the state $z_1$ which alters the contents of register $\reg_1$. By a similar analysis, we obtain the following,
\begin{lem}\label{zero0}
$z_i\reg_1^{n_1}\reg_2^{n_2}\reg_3^{n_3}\in{\Acc}(\o)$ if and only if $n_i=0$.
\end{lem}
Let $\mathcal{P}$ be a program (i.e. a sub-machine) and $\leq_\mathcal{P}$ be its corresponding computation relation. We define the relation $\sqsubseteq_\mathcal{P}$ on $\conf{{\mathcal{P}}}$ via ${\cf}\sqsubseteq_\mathcal{P} {\cf'}$ iff ${\cf}\leq_\mathcal{P} {\cf'}\vee u$, where either $u=\bot$ or $u\in\mathrm{ID}({\o}) $ with $u\in{\Acc}(\o)$.\footnote{If $p$ is an instruction, by ${\cf}\sqsubseteq_{\{p\}} {\cf'}$ we mean ${\cf}\leq^p {\cf'}\vee u.$} If $\mathcal{P}$ contains no ${\State}_{\o}$-instructions (i.e., no instruction $z_ix\leq\cdots$), then ${\cf}\sqsubseteq_{\mathcal{P}}{\cf'}$ iff there is a computation from ${\cf}$ to ${\cf'}\vee u$ with instructions from $\mathcal{P}$ such that every zero-test was properly applied. Note that $\sqsubseteq_\mathcal{P}$ is transitive on configurations and ${\cf}\leq_\mathcal{P} {\cf'}$ implies ${\cf}\sqsubseteq_\mathcal{P} {\cf'}$. We obtain the following lemma.
\begin{lem}\label{zero}
Let $p$ be the instruction $\qin\leq\qout\vee z_i$ with distinct $\qin, \qout\nin {\State}_{\o}$. For $x,x'\in \Reg_3^*$, $\qin x \sqsubseteq_{\{p\}} \qout x'$ if and only if $x=x'=\reg_1^{n_1}\reg_2^{n_2}\reg_3^{n_3}$ and $n_i=0$.
\end{lem}
\begin{proof}
Let $x=\reg_1^{n_1}\reg_2^{n_2}\reg_3^{n_3}$. The only instruction applicable to $\qin x$ is $p$, so from $\qin\leq \qout\vee z_i$ we obtain $\qin x\leq^p \qout x\vee z_i x$. Since the only instructions applicable are those from $\{p\}$ and $\qin\neq\qout$, the computation cannot proceed from this configuration. Hence by Lemma~\ref{zero0}, 
$$\qin x \sqsubseteq_{\{p\}} \qout x' \iff x=x' \mbox{ and } z_ix\leq_{{\o}} q_F\iff x=x' \mbox{ and } n_i=0.$$
\end{proof}
%
%Subsection: Multiplying and Dividing
%
\subsection{Multiplying and Dividing}
We are now ready to faithfully simulate an increment-$\reg$ instruction by a program that multiplies the contents of register $\reg \in \{\reg_1, \reg_2\}$ by the fixed constant $K$. For $p: \qin\leq \qout \reg$, we define the program ${\times}(p)$ to have states $\State_{\times(p)}=\State_{+K}\cup \State_{\Trsf_\reg}$ and instructions $\Inst_{\times(p)}=\Inst_{+K}\cup\Inst_{\Trsf_\reg}\cup\{\times_{\mathrm{in}}, \times_{\Trsf},\times_{\mathrm{out}} \}$, where
$$\begin{array}{l c r c l      }
{\times_{\mathrm{in}}}&:& \qin  &\leq& a_K\vee z_3\\
{\times_{\Trsf}} &:& a_K &\leq& t_0\vee z_i\\
{\times_{\mathrm{out}}} &:& t_0&\leq& \qout \vee z_3
\end{array}.$$
The instruction $\times_{\mathrm{in}}$ is intended to verify that the auxiliary register $\reg_3$ is in fact empty, and initiate the process of storing in $\reg_3$ $K$-times the contents in the active register $\reg$. The instruction $\times_\Trsf$ is meant to check that all the contents of the active register $\reg$ have been emptied (and thus $K$-times that amount is in $\reg_3$), and initiate the transfer program. The instruction $\times_{\mathrm{out}}$ is intended to end the program by transitioning to the state $\qout$ only when the transfer is complete, i.e., when $\reg_3$ has been emptied. 
Below is an example of $\times(\qin\leq\qout \reg_1)$ running on the configuration $\qin \reg_1^2\reg_2$:
$$\begin{array}{r l l}
\qin \reg_1^2\reg_2 & \sqsubseteq_{\{\times_{\mathrm{in}}\}}&  a_K\reg_1^2 \reg_2\\
&\leq_{\times_{\reg_1}}& a_K\reg_2\reg_3^{2K}\\
&\sqsubseteq_{\{\times_{\Trsf}\}} &t_0\reg_2\reg_3^{2K}\\
&\leq_{\Trsf_{\reg_1}}&t_0 \reg_1^{2K}\reg_2\\
&\sqsubseteq_{\{ \times_{\mathrm{out}}\}}&\qout \reg_1^{2K}\reg_2
\end{array}$$

In this way, we obtain the following technical lemma by induction as a consequence of Lemma~\ref{zero} and Equations~\ref{eq:TimesK} and~\ref{eq:TransK}. We state the lemma for $\reg=\reg_1$, but the same holds when swapping the roles of $\reg_1$ and $\reg_2$.
\begin{lem}\label{times} 
Let $p: \qin\leq \qout\reg_1$ be an increment-$\reg_1$ instruction, where $\qin, \qout \not \in \State_{\times(p)}$. Then 
\[\qin \reg_1^{N_1}\reg_2^{N_2}\reg_3^{N_3} \sqsubseteq_{\times(p)} \qout \reg_1^{M_1}\reg_2^{M_2}\reg_3^{M_3}\] if and only if $M_1=KN_1$, $M_2=N_2$, and $M_3=N_3=0$. In fact, for each $n_1,n_2,n_3\in \N$ and state $q\in \State_{\times(p)} $,
\begin{enumerate}
\item $\qin \reg_1^{N_1}\reg_2^{N_2}\reg_3^{N_3} \sqsubseteq_{\times(p)} q\reg_1^{n_1}\reg_2^{n_2}\reg_3^{n_3}$ iff $N_3=0$, $N_2=n_2$, and 
$$KN_1 = \left\{\begin{array}{r l} n_1+n_3+\delta & \mbox{if } q=t_\delta \mbox{ where } \delta\in \{0,1\},\\ Kn_1+n_3+(K-\delta) & \mbox{if }q=a_\delta \mbox{ where } 0\leq \delta\leq K \end{array}\right. .$$
\item $q\reg_1^{n_1}\reg_2^{n_2}\reg_3^{n_3}  \sqsubseteq_{\times(p)} \qout \reg_1^{M_1}\reg_2^{M_2}\reg_3^{M_3} $ iff $M_3=0$, $M_2=n_2$, and 
$$M_1 = \left\{\begin{array}{r l} n_1+n_3+\delta & \mbox{if } q=t_\delta \mbox{ where } \delta\in \{0,1\},\\ Kn_1+n_3+(K-\delta) & \mbox{if }q=a_\delta \mbox{ where } 0\leq \delta\leq K \end{array}\right. .$$
\end{enumerate}
\end{lem}

For a configuration $\cf=q\reg_1^n\reg_2^m$ in $\acm$, by $\cf_K$ we denote the configuration $q\reg_1^{K^n}\reg_2^{K^m}$ in $\acm_K$.

\begin{cor}\label{timesK} Let $p$ be an increment instruction from some $2$-ACM $\acm$. Then $\cf\leq^p\cf'$ if and only if $\cf_K\sqsubseteq_{\times(p)}\cf'_K$ for any configurations $\cf,\cf'$ in $\conf{\acm}$.
\end{cor}

In a completely analogous way, given a decrement-$\reg$ instruction $p:\qin\reg\leq \qout $, we define the {\em division by $K$} program $\div(p)$ as follows. For its set of states $\State_{\div(p)}$, we define a fresh set of states $\State_{\reg-K}=\{s_0,...,s_K \}$ and set $\State_{\div(p)}=\State_{\reg-K}\cup\State_{\Trsf_\reg}$. We take as its instructions $\Inst_{\div(p)}=\Inst_{\reg-K}\cup\Inst_{\Trsf_\reg}\cup\{\div_{\mathrm{in}},\div_{\Trsf},\div_\mathrm{out} \}$, where $\Inst_{\reg-K}$ contains the instruction $\div_{\mathrm{loop}}: s_K\leq s_0\reg_3$ and $K$-many instructions of the form $-_i:s_{i-1}\reg\leq s_i $, for $1\leq i\leq K$, and finally
$$\begin{array}{l c r c l}
{\div_{\mathrm{in}}}&:& \qin  &\leq& s_0\vee z_3\\
{\div_{\Trsf}} &:& s_0 &\leq& t_0\vee z_i\\
{\div_{\mathrm{out}}} &:& t_0&\leq& \qout \vee z_3
\end{array}.$$

The instruction $\div_{\mathrm{in}}$ is intended to verify that the auxiliary register $\reg_3$ is in fact empty, and initiate the process of storing in $\reg_3$ the quotient by $K$ of the contents in the active register $\reg$. The instruction $\times_\Trsf$ is meant to check that all the contents of the active register $\reg$ have been emptied (and thus $K$-divided by that amount is in $\reg_3$), and initiate the transfer program. The instruction $\div_{\mathrm{out}}$ is intended to end the program transitioning to the state $\qout$ only when the transfer is complete, i.e., when $\reg_3$ has been emptied. 
Below is an example of $\div(\qin\reg_1\leq\qout)$ running on the configuration $\qin \reg_1^{2K}\reg_2$:
$$\begin{array}{r l l}
\qin \reg_1^{2K}\reg_2 & \sqsubseteq_{\{\div_{\mathrm{in}}\}}&  s_0\reg_1^{2K} \reg_2\\
&\leq_{\reg_1-K}& s_K\reg_1^K\reg_2\\
&\leq^{\div_{\mathrm{loop}}}& s_0\reg_1^K\reg_2\reg_3\\
&\leq_{\reg_1-K}& s_K\reg_2\reg_3\\
&\leq^{\div_{\mathrm{loop}}}& s_0\reg_2\reg_3^2\\
&\sqsubseteq_{\{\div_{\Trsf}\}} &t_0\reg_2\reg_3^{2}\\
&\leq_{\Trsf_{\reg_1}}&t_0 \reg_1^{2}\reg_2.\\
&\sqsubseteq_{\{ \div_{\mathrm{out}}\}}&\qout \reg_1^{2}\reg_2
\end{array}$$
The following technical lemma is easily verified by induction. We state the lemma for $\reg=\reg_1$, but the same holds by swapping the roles of $\reg_1$ and $\reg_2$.
\begin{lem}\label{div} 
Let $p: \qin\reg_1\leq \qout$ be a decrement-$\reg_1$ instruction. Then \[\qin \reg_1^{N_1}\reg_2^{N_2}\reg_3^{N_3} \sqsubseteq_{\div(p)} \qout \reg_1^{M_1}\reg_2^{M_2}\reg_3^{M_3}\] if and only if $KM_1=N_1>0$, $M_2=N_2$, and $M_3=N_3=0$. In fact, for each $n_1,n_2,n_3\in \N$ and state $q\in \State_{\div(p)} $,
\begin{enumerate}
\item $\qin \reg_1^{N_1}\reg_2^{N_2}\reg_3^{N_3} \sqsubseteq_{\div(p)} q\reg_1^{n_1}\reg_2^{n_2}\reg_3^{n_3}$ iff $N_3=0$, $N_2=n_2$, and
$$N_1 = \left\{\begin{array}{r l} n_1+n_3+\delta & \mbox{if } q=t_\delta \mbox{ where } \delta\in \{0,1\}, \\ Kn_1+n_3+(K-\delta) & \mbox{if }q=s_\delta \mbox{ where } 0\leq \delta\leq K \end{array}\right. .$$
\item $q\reg_1^{n_1}\reg_2^{n_2}\reg_3^{n_3}  \sqsubseteq_{\times(p)} \qout \reg_1^{M_1}\reg_2^{M_2}\reg_3^{M_3} $ iff $M_3=0$, $M_2=n_2$, and
$$KM_1 = \left\{\begin{array}{r l} n_1+n_3+\delta & \mbox{if } q=t_\delta \mbox{ where } \delta\in \{0,1\},\\ Kn_1+n_3+(K-\delta) & \mbox{if }q=s_\delta \mbox{ where } 0\leq \delta\leq K \end{array}\right. .$$
\end{enumerate}
\end{lem}
\begin{cor}\label{divK} Let $p$ be a decrement instruction from some $2$-ACM $\acm$ and $\cf,\cf'$ be configurations in $\acm$. Then $\cf\leq^p\cf'$ if and only if $\cf_K\sqsubseteq_{\div(p)}\cf'_K$.
\end{cor}
\begin{cor}\label{DivBack}
Assume that $p:\qin\reg\leq \qout$ is a decrement instruction, 
$\cf_\mathrm{in}$ is a $\qin$-configuration, $\cf$ is a $\State_{\div(p)}$-configuration and $\cf_\mathrm{out}$ is $\qout$-configuration. If $\cf_\mathrm{in}\sqsubseteq_{\div(p)}\cf$ and $\cf_\mathrm{in}\sqsubseteq_{\div(p)} \cf_\mathrm{out}$, then $\cf\sqsubseteq_{\div(p)}\cf_\mathrm{out}$.
\end{cor}
\begin{proof} 
Since $\cf_\mathrm{in}\sqsubseteq_{\div(p)}\cf$ and $\cf_\mathrm{in}\sqsubseteq_{\div(p)} \cf_\mathrm{out}$, by Lemma~\ref{div} the values of $\cf_\mathrm{in}$ and $\cf$, as well as the values of $\cf_\mathrm{in}$ and $\cf_\mathrm{out}$ are linked. Therefore, the values of $\cf_\mathrm{out}$ and $\cf$ are also linked and hence by Lemma~\ref{div} we obtain $\cf\sqsubseteq_{\div(p)}\cf_\mathrm{out}$.
\end{proof}
%
%Subsection: Construction of ${\acm}_K$
%
\subsection{Construction of ${\acm}_K$}
Let $\acm=(\Reg_2,\State,\Inst,q_f)$ be a $2$-ACM  and let  $K>1$ be an integer. 
Since the configuration $q_f$ is accepted in $\acm$ by definition, we will need $(q_f)_K=q_f\reg_1\reg_2$ to be accepted in $\acm_K$. To accommodate this, we define the \textit{end program} as follows.
For a fresh variable $c_F$, we define the set of states ${\State}_F=\{c_F\}$ and the set of instructions ${\Inst}_F=\{F_1,F_2\}$ by:
$$\begin{array}{l c r c l }
F_1&:& q_f\reg_1&\leq& c_F\\
F_2&:& c_F\reg_2&\leq& q_F \end{array} .$$
By $\leq_F$ we denote the computation relation for the end program. 

\begin{lem}\label{end}
For $q\in \State_{F}\cup\{q_f\}$, $q\reg_1^{n_1}\reg_2^{n_2}\reg_3^{n_3}\leq_F q_F$ if and only if $n_3=0$ and 
$$(n_1,n_2)=\left\{ \begin{array}{rl} (1,1)&\mbox{if $q=q_f$}\\ (0,1)&\mbox{if $q=c_F$}\\(0,0)&\mbox{if $q=q_F$} \end{array}\right. .$$
\end{lem} 

We write the instructions $\Inst$ of $\acm$ as the disjoint union $\Inst_+\cup\Inst_-\cup\Inst_\vee$ of its increment, decrement, and forking instructions, respectively. We can now formally define the $3$-ACM simulation of $\acm$ to be the machine ${\acm}_K=(\Reg_3,\State_K,\Inst_K,q_F)$, where 
\begin{itemize}
\item $\State_K$ is the (disjoint) union of $\State$, $\State_{\o}$, $\State_{F}$, $\State_{\times(p)}$ for each ${p\in \Inst_+}$, and $\State_{\div(p)}$ for each $p\in\Inst_-$.
\item $\Inst_K$ is the (disjoint) union of $\Inst_\vee$, $\Inst_{\o}$, $\Inst_{F}$, $\Inst_{\times(p)}$ for each ${p\in \Inst_+}$, and $\Inst_{\div(p)}$ for each ${p\in\Inst_-}$.
\item $q_F$ is the final state of $\acm_K$.
\end{itemize}
Formally, we view all states and instructions in some multiply/divide program $\mathcal{P}(p)$ (where $p$ is an increment/decrement instruction from $\acm$) as being labeled by the instruction $p$, e.g., a state from $\State_{\mathcal{P}(p)}$ is of the form $q^p$, and an instruction $\Inst_{\mathcal{P}(p)}$ is of the form $\rho^p$. In other words, we make states and instructions in each subprogram disjoint. In fact, since there are no instructions in $\mathcal{P}(p)$ of the form $\cdots\leq q_F\cdots$, we obtain the following useful observation.
\begin{lem}\label{internal states}
Let $p$ be an increment or decrement instruction from $\acm$ and $\mathcal{P}(p)$ its corresponding program in $\acm_K$. If $\cf$ is a configuration in $\acm_K$ labeled by a state from $\State_{\mathcal{P}(p)}$ then the only instructions applicable to $\cf$ are those from $\Inst_{\mathcal{P}(p)}$. Furthermore, if $\qout$ is the output state of $p$, then $\cf$ being accepted in $\acm_K$ implies $\cf\sqsubseteq_{\mathcal{P}(p)} \cf'\in \Acc(\acm_K)$, where $\cf'$ is labeled by $\qout$.
\end{lem}
Recall, for a configuration $\cf=q\reg_1^n\reg_2^m$ in $\acm$, by $\cf_K$ we denote the configuration $q\reg_1^{K^n}\reg_2^{K^m}$ in $\acm_K$.
\begin{lem}\label{Mhalts}
The following hold for any 2-ACM $\acm=(\Reg_2,\State,\Inst,q_f)$ and $K>1$.
\begin{enumerate} 
\item A configuration $\cf$ is accepted in $\acm$ iff $\cf_K$ is accepted in $\acm_K$. Furthermore, any accepted configuration in $\acm_K$ labeled by a state from $\State$ must be of the form $\cf_K$ where $\cf$ is accepted in $\acm$.
\item Let $p$ be an increment or decrement instruction of $\acm$ and $\cf$ a configuration of the corresponding program $\mathcal{P}(p)$ $(\mathcal{P}\in\{\times,\div\})$. 
Then $\cf$ is accepted in $\acm_K$ iff there are accepted configurations $\cf',\cf''$ in $\acm$ such that $\cf'\leq^p\cf''$ and $\cf'_K\sqsubseteq_{\mathcal{P}(p)} \cf\sqsubseteq_{\mathcal{P}(p)} \cf''_K$.
\end{enumerate}
\end{lem}
\begin{proof} For (1), let $\cf$ be a configuration in $\acm$.  
Since there are no $q_f$-instructions in $\acm$ by definition, if $\cf$ is labeled by state $q_f$ then it is accepted in $\acm$ iff $\cf=q_f$, i.e., both registers $\reg_1$ and $\reg_2$ are empty. By definition, the only $q_f$-instructions in $\acm_K$ are those found in the end program. By Lemma~\ref{end}, the only accepted configuration in $\acm_K$ labeled by $q_f$ is $\cf_K$. Now, suppose $p$ is a $q$-instruction from $\acm$. Clearly, if $p$ is a forking instruction, then $\cf\leq^p \cf'\vee \cf''$ in $\acm$ iff $\cf_K\leq^p \cf'_K\vee\cf_K''$. Otherwise, $p$ is an increment or decrement instruction, and by Corollaries~\ref{timesK} and~\ref{divK}, $\cf\leq^p \cf'$ in $\acm$ iff $\cf_K\sqsubseteq_{\mathcal{P}(p)} \cf_K'$ in $\acm_K$. The claim therefore follows by induction on the computation lengths.

For (2), consider a configuration $\cf$ in $\acm_K$ labeled by a state from some program $\mathcal{P}(p)$, where $p$ is an increment or decrement instruction from $\acm$. Let $\qin$ and $\qout$ be the input and output states of $p$, respectively. By Lemma~\ref{internal states}, we conclude that if a computation witnesses $\cf$ being accepted in $\acm_K$ it must implement the output instruction of $\mathcal{P}(p)$. That is $\cf\sqsubseteq_{\mathcal{P}(p)}\qout \reg_1^{n_1}\reg_2^{n_2}\reg_3^{n_3}$. By (1), $n_1$ and $n_2$ are powers of $K$ while $n_3=0$. By Lemmas~\ref{times} and~\ref{div}, the result follows.
\end{proof}

Let $\Uacm$ be the 2-ACM given by Theorem~\ref{lincoln}. Since membership of ${\Acc}(\Uacm)$ is undecidable, we obtain the following consequence of Lemma~\ref{Mhalts}(1):
 
\begin{cor}\label{mkundec} 
Membership in the set ${\Acc}(\Uacm_K)$ is undecidable for $K>1$.
 \end{cor}
%
%Subsection: Register-Admissibility in $\acm_K$
%
\subsection{Register-Admissibility in $\acm_K$}
Consider an $n$-variable simple equation $[\Dset]$, a $2$-ACM $\acm$, and an integer $K>1$. To show the register-admissibility of $[\Dset]$ in $\acm_K$, we need only show that for each configuration $\cf$ in $\acm_K$, if the ID $\bigvee_{\vd\in\Dset} \cf_\vd$ is obtained by an instance of $\leq^{\regeq{\Dset}}$ from $\cf$ and $\bigvee_{\vd\in\Dset} \cf_\vd$ is accepted in $\acm_K$, then $\cf$ is accepted in $\acm_K$. By Lemma~\ref{comprel}(2), this implication is equivalently stated as
$$\cf\leq^{\regeq{\Dset}} \bigvee_{\vd\in \Dset}\cf_\vd\quad\&\quad (\forall \vd\in\Dset)(\cf_\vd\in \Acc(\acm_K))\quad \implies \quad\cf\in \Acc(\acm_K) .$$

Since we are only considering applications of $[\Dset]$ to the register contents, we can split our analysis into cases depending upon the state $q\in \State_K$ that labels the configurations. The following useful observation follows from the fact that every variable that appears on the right-hand side of a simple equation appears also on the left-hand side.

\begin{lem}\label{cond1} If a substitution sends all the joinands of a simple equation to $1$, then 
it sends all variables of the equation to $1$.
\end{lem}

In the following for two tuples $\rvec$ and $\vd$ of the same length, $\rvec \vd$ denotes their dot product. In the next section we will actually view $\rvec$ as a row-matrix and $\vd$ as a column-matrix, so $\rvec \vd$ will be their matrix product. 
In this way, focusing on the list/column vector $d$ of exponents of the variables in $[\Dset]$ and also on the list/row vector $\sigma$ of the exponents of the images of the variables via a one-variable substitution, the above lemma can be stated as: for an $n$-variable simple equation $[\Dset]$, if $\rvec \vd = 0$ for each $\vd\in \Dset$, then $\rvec$ must be the constantly zero vector $\Zero\in\N^n$.

As observed in Section~\ref{motivation}, if $\cf \leq^{\regeq{\Dset}} \bigvee_{\vd\in\Dset} \cf_\vd$ is an instance of $\leq^{\regeq{\Dset}}$, we may write $\cf=qx\mathbf{x_n}^{\Lvec{}}$ and $\cf_\vd=qx\mathbf{x_n}^\vd$ for each $\vd\in \Dset$, where $x\in \Reg_3^*$ and $\mathbf{x_n}=(x_1,\ldots,x_n)\in (\Reg_3^*)^n$.  Let $x=\reg_1^{\const_1}\reg_2^{\const_2}\reg_3^{\const_3}$, where $\const_1,\const_2,\const_3\geq0$, and for each $j\in\{1,2,3\}$, define ${\rvec}_j\in\N^n$ via $x_i=\reg_1^{{\rvec}_1(i)}\reg_2^{{\rvec}_2(i)}\reg_3^{{\rvec}_3(i)}$, for each $i\in\nset{n}$. Then,
$${\cf}= q\reg_1^{\const_1+ {\rvec}_1{\Lvec{}}}\reg_2^{\const_2+ {\rvec}_2{\Lvec{}}}\reg_3^{\const_3+ {\rvec}_3\Lvec{}}, $$
and for each $d\in\Dset$,
$${\cf}_d=q\reg_1^{\const_1+ {\rvec}_1{\vd}}\reg_2^{\const_2+ {\rvec}_2{\vd}}\reg_3^{\const_3+ {\rvec}_3{\vd}}.$$

\begin{lem}\label{lem: zero end} The zero-test program is register-admissible for any simple equation, and the end program is register-admissible for any non-mingly simple equation.
\end{lem}
\begin{proof}
Let $[\Dset]$ be a simple equation. If $q$ is the final state $q_F$, then $\cf_\vd$ is accepted iff all registers are empty, i.e. ${\cf}_\vd=q_F$ for each $\vd\in\Dset$. Hence $x=1$ and $ {\rvec}_j \vd=0$ for each $\vd\in \Dset$ and $j\in\{1,2,3\}$.  For each $j\in\{1,2,3\}$, this implies that ${\rvec}_j={\Zero}$, by Lemma~\ref{cond1}. Therefore ${\cf}=q_F\in {\Acc}({\acm}_K)$. 

Suppose $q=z_i$, and without loss of generality, let $i=3$. By Lemma~\ref{zero0}, $\cf_\vd$ is accepted iff register $\reg_3$ is empty, i.e., ${\cf}_\vd\in {\Acc}(\o)$ iff $\const_3+ {\rvec}_3 {\vd}= 0$, for each $\vd\in \Dset$. This implies $\const_3=0$ and $ {\rvec}_3 \vd=0$, for each $\vd\in \Dset$. So by Lemma~\ref{cond1}, ${\rvec}_3 = {\Zero}$. Hence $\const_3+ {\rvec}_3{\Lvec{}} = 0$ and ${\cf}\in {\Acc}(\o)\subseteq {\Acc}({\acm}_K)$.

Lastly, suppose $q=c_F$. By Lemma~\ref{end}, ${\cf}_\vd\in {\Acc}(F)$ iff ${\cf}_\vd=c_F\reg_2$. Hence $\const_1=\const_3=0$, $ {\rvec}_1{\vd} = {\rvec}_3{\vd} = 0$ for each $\vd\in\Dset$, and $\const_2+ {\rvec}_2{\vd} =1$. Again, by Lemma~\ref{cond1}, ${\rvec}_1={\rvec}_3 = {\Zero}$. Let $\lambda=\rvec_2 \Lvec{}$. Then $\lambda$ is positive since $[\Dset]$ is a simple equation. If $\lambda=1$, then $\cf=c_F\reg_2$ and we are done. If $\lambda\neq 1$ then $\rvec_2$ is a substitution witnessing that $[\Dset]$ is mingly. 
\end{proof}

Now, suppose $\cf$ is labeled by a state $q\in \State$ from $\acm$. By Lemma~\ref{Mhalts}, $\cf_\vd$ is accepted in $\acm_K$ only if the contents of the registers $\reg_1, \reg_2$ are each powers of $K$ and the register $\reg_3$ is empty. That is, $\const_1+\rvec_1\vd$ and $\const_2+\rvec_2\vd$ are powers of $K$ and $\const_3+\rvec_3\vd=0$. On the one hand, Lemma~\ref{cond1} ensures that $\rvec_3$ is the zero vector and $\const_3=0$, and so $\reg_3$ is empty in $\cf$. 

By the motivation in Section~\ref{motivation}, a natural condition to consider would be to stipulate that $[\Dset]$ satisfies $(\star K)$. In such a case, if $\const_1+\rvec_1\vd$ is a power of $K$ for each $\vd\in \Dset$ then there exists $\bar \vd\in \Dset$ such that $\rvec_1\Lvec{}=\rvec_1\bar \vd$. Similarly, if $\const_2+\rvec_2\vd$ is a power of $K$ for each $\vd\in \Dset$, then there exists ${\bar \vd}'\in \Dset$ such that $\rvec_1\Lvec{}=\rvec_1{\bar \vd}'$. However, there is no reason {\em a priori} that entails $\bar \vd={\bar \vd}'$ and thus $\cf \in\{ \cf_\vd:\vd\in \Dset\},$ which would be sufficient to ensure that $\cf$ would be accepted if $\bigvee_{\vd\in\Dset} \cf_\vd$ were accepted.

Since the most naive and obvious way to ensure acceptance is to ask that the left-hand side $\cf$ appears as one of the joinands $\cf_\vd$ on the right-hand side, it is sufficient to stipulate that $[\Dset]$ satisfies the following condition:
\begin{center}
If the exponents of each variable in the right-hand side of $[\Dset]$ produced by a 2-variable substitution are translated powers of $K$, then the substitution instance is trivial.
\end{center}
In symbolic terms this can be written as:
\begin{equation}\label{cond2}\tag{${\star\star} K$}
\begin{array}{c}
\mbox{For all } \rvec,\rvec'\in\N^n \mbox{ and for all }\const,\const'\in \N,\\
\mbox{if }  \const+\rvec{\vd} \mbox{ and } \const'+ \rvec'{\vd} \mbox{ are powers of } K\mbox{ for each }\vd\in \Dset,\\ 
\mbox{then there exists } \bar \vd\in \Dset \mbox{ such that } \rvec{\bar \vd}= \rvec{\Lvec{}}\mbox{ and } \rvec'{\bar \vd}= \rvec'\Lvec{}. 
\end{array}
\end{equation}

In this case, we say $[\Dset]$ satisfies $({{\star\star}}K)$. We also consider the condition $({\star\star})$: there exists $K>1$ such that $({\star\star} K)$ holds. Note that, by setting $\rvec=\rvec'$, we see that if $[\Dset]$ satisfies $({\star\star} K)$ then it satisfies $(\star K)$.\footnote{Surprisingly, we prove in Theorem~\ref{final} that the converse holds for all $K$ sufficiently large.} So, we obtain the following lemma.
\begin{lem}\label{2starTo1star}
If a simple equation satisfies $({\star\star})$ then it satisfies $(\star)$.
\end{lem}

It is clear then that when $q\in \State$, if $[\Dset]$ satisfies $({\star\star} K)$ then the acceptance of $\bigvee_{\vd\in \Dset}\cf_\vd$ in $\acm_K$ implies the acceptance of $\cf$ in $\acm_K$ by Lemma~\ref{Mhalts}(1) and the observations above. 

As it turns out, the remaining cases can be reduced to the above, and so satisfying the condition $({\star\star}K)$ alone is sufficient to ensure register-admissibility. The only remaining cases to verify are when the state $q$ is internal to a multiply or divide by $K$ program. Let $p\in \Inst$ be some increment or decrement instruction for $\acm$. The idea is that if an instance of $[\Dset]$, which leads to acceptance in $\acm_K$, occurs internal to a program $\mathcal{P}(p)$ then, by using Lemma~\ref{Mhalts}(2) such an instance could have {\em equivalently} occurred at the end (or beginning) of executing the program $\mathcal{P}(p)$. Without loss of generality, suppose the instruction $p$ acts on register $\reg_1$ with input and output states $\qin$ and $\qout$, respectively. 

For instance, if $q$ is a transfer state $q=t_\delta$, where $\delta\in \{0,1\}$, then by Lemmas~\ref{times}(2) and \ref{div}(2), 
$${\cf}\sqsubseteq_{\mathcal{P}(p)}\cf' :=\qout \reg_1^{(\const_1+  {\rvec}_1{\Lvec{}}) +(\const_3 + {\rvec}_3{\Lvec{}})+\delta}\reg_2^{\const_2+ {\rvec}_2{\Lvec{}} },$$
and by Lemmas~\ref{times}(2),~\ref{div}(2), and~\ref{Mhalts}(2), for each $\vd\in \Dset$,
$$\cf_\vd \sqsubseteq_{\mathcal{P}(p)}\cf'_d := \qout \reg_1^{(\const_1+  {\rvec}_1\vd) +(\const_3 + {\rvec}_3\vd)+\delta}\reg_2^{\const_2+ {\rvec}_2 \vd }\in \Acc(\acm_K).$$
We see that by setting $\const=\const_1+\const_3+\delta$, $\const'=\const_2,$  $\rvec=\rvec_1+\rvec_3$, and $\rvec'= \rvec_2$, we obtain the instance $\cf'\leq^\Dset \bigvee_{\vd\in \Dset}\cf'_\vd \in \Acc(\acm_K)$. Since $\cf\leq_{\acm_K} \cf'$ and $\cf'$ is accepted ($\cf'$ is labeled by state $\qout\in \State$, which was handled above), it follows that $\cf$ is accepted.

Similarly, if $q$ is a multiply state $q=a_\delta$, for some $\delta\leq K$, then by Lemma~\ref{times}(2),
$${\cf}\sqsubseteq_{\times(p)}\cf':=\qout \reg_1^{K(\const_1+ {\rvec}_1{\Lvec{}} + K-\delta) +(\const_3 + {\rvec}_3{\Lvec{}}) }\reg_2^{\const_2+ {\rvec}_2{\Lvec{}} }$$
and by Lemmas~\ref{times}(2) and \ref{Mhalts}(2), for each $\vd\in \Dset$,
$${\cf_\vd}\sqsubseteq_{\times(p)}\cf'_d:= \qout \reg_1^{K(\const_1+ {\rvec}_1{\vd} + K-\delta) +(\const_3 + {\rvec}_3{\vd}) }\reg_2^{\const_2+ {\rvec}_2{\vd} }\in \Acc(\acm_K).$$
So by setting $\const=K\const_1+\const_3+K-\delta$, $\const'=\const_2,$  $\rvec=K\rvec_1+\rvec_3$, and $\rvec'= \rvec_2$, we obtain the instance $\cf'\leq^\Dset \bigvee_{\vd\in \Dset}\cf'_\vd \in \Acc(\acm_K)$. Since $\cf\leq_{\acm_K} \cf'$ and $\cf'$ is accepted ($\cf'$ is labeled by state $\qout\in \State$, which was handled above), it follows that $\cf$ is accepted.

Lastly, we consider when $q$ is a division state $q=s_\delta$, for some $\delta\leq K$. By Lemma~\ref{div}(2), 
$$\cf':=\qin \reg_1^{(\const_1+ {\rvec}_1{\Lvec{}} + \delta)+K(\const_3 + {\rvec}_3{\Lvec{}}) }\reg_2^{\const_2+ {\rvec}_2{\Lvec{}} }\sqsubseteq_{\div(p)}\cf, $$ 
and by Lemmas~\ref{div}(2) and~\ref{Mhalts}(2), for each $\vd\in \Dset$, 
$$\cf'_\vd:=\qin \reg_1^{(\const_1+ {\rvec}_1{\vd} + \delta)+K(\const_3 + {\rvec}_3{\vd}) }\reg_2^{\const_2+ {\rvec}_2{\vd} }\sqsubseteq_{\div(p)}{\cf_\vd} \sqsubseteq_{\div(p)} \cf_\vd''\in \Acc(\acm_K),$$
where $\cf''_\vd$ is the unique output configuration of $\div(p)$ labeled by $\qout$. 

Now, it is clear that by setting $\const=\const_1+K\const_3+\delta$, $\const'=\const_2$, $\rvec=\rvec_1+K\rvec_3$, and $\rvec'=\rvec_2$, we have that $\cf'\leq^\Dset \bigvee_{\vd\in \Dset} \cf'_\vd$. Hence $\cf'=\cf'_{\bar \vd}$ for some $\bar \vd\in \Dset$ by $({\star\star}K)$, and so $\cf_{\bar \vd}'\sqsubseteq_{\div(p)}\cf$. Since $\cf_{\bar \vd}'\sqsubseteq_{\div(p)}\cf_{\bar \vd}''$, by Corollary~\ref{DivBack}, it follows that $\cf \sqsubseteq_{\div(p)} \cf_{\bar \vd}''$.  
Therefore $\cf$ is accepted in $\acm_K$ if $\bigvee_{\vd\in \Dset} \cf_\vd$ is accepted in $\acm_K$. 

By the arguments above the following lemma is established: 

\begin{lem}\label{dredundant}
Let ${\acm}$ be a $2$-ACM and $K>1$. If a non-mingly simple equation satisfies $({\star\star}K)$ then it is register-admissible in ${\acm}_K$.
\end{lem}
%
%Subsection: Condition $({\star\star})$ and undecidability
%
\subsection{Condition $({\star\star})$ and undecidability} Assume that $[\Dset]$ is a non-mingly simple equation that satisfies $({\star\star})$. Since it is non-mingly, by Lemma~\ref{mingly} we get that $[\Dset]$ is state-admissible in any machine. Since it also satisfies $({\star\star})$,  by Lemma~\ref{dredundant} we have that 
$[\Dset]$ is register-admissible in $\acm_K$ for some integer $K>1$, where $\acm$ is any machine. In particular, $ [\Dset]$ is admissible in $\tilde\acm_K$, where $\tilde\acm$ is the machine with undecidable halting problem. By Corollary~\ref{mkundec}, the machine $\tilde\acm_K$ has an undecidable set of accepted configurations for any $K>1$.
By  Lemma~\ref{Wadmiss} we obtain $\mathbf{W}_{\tilde\acm_K}^+\in \CRL+[\Dset]$. Therefore, 
 $\CRL+[\Dset]$ has an undecidable word problem by Theorem~\ref{Vmhard}. This proves the following result.

\begin{cor}\label{preUndLink}
For any finite set $\Gamma$ of non-mingly equations that satisfy  $({\star\star})$, every subvariety of $\RL$ containing $\CRL+\Gamma$ has an undecidable word problem. 
\end{cor}

As motivation for the general case, we show that the $1$-variable basic equations $[n,P]:x^n\leq \bigvee_{p\in P} x^p$, where $P$ contains at least two distinct positive integers, considered in Lemma~\ref{und1var}, define varieties with undecidable word problem. The results of the next section will show that this holds for many more equations, all spineless equations. 
\begin{thm}\label{th:1-var}
Let $[n,P]$ be a $1$-variable basic equation where $P$ contains at least two distinct positive integers. Then the variety $\CRL +[n,P]$ has an undecidable word problem. If additionally $P$ only contains integers strictly greater than $n$, then the variety $\CRL+[n,P]$ has an undecidable equational theory.
\end{thm}

\begin{proof}
Let $[\Dset]$ be the $n$-variable simple equation that is the linearization of $[n,P]$ over $\CRL$ (given by Equation~\ref{1varLin} in Lemma~\ref{und1var}) and let $p,q\in P$ be such that $p>q>0$. Note that by Lemma~\ref{und1var}, $[\Dset]$ is spineless and hence it is not mingly by Lemma~\ref{nming}. By Corollary~\ref{preUndLink}, to establish the first claim it is enough to show $[\Dset]$ satisfies $({\star\star})$. We will show that $[\Dset]$ satisfies $({\star\star} K)$, for every $K> 1+p-q$; since $p>q$, this implies that $K>1$. 

Assume that there exist $\const,\const' \in \N$ and  $1$-variable substitutions $\rvec,\rvec'$ such that $\const + \rvec d$ and $\const +\rvec d'$ are powers of $K$ for each $d\in \Dset$. We will show that $\rvec$ and $\rvec'$ are trivial substitutions (i.e., all entries are 0), and hence $\rvec\Lvec{}=0=\rvec \bar d$ and $\rvec'\Lvec{}=0=\rvec'\bar d$ for every $\bar d \in \Dset$.

Arguing towards contradiction, suppose that $\rvec$ is nontrivial with $\rvec(i)>0$ for some $i\leq n$. Now (by Equation~\ref{1varLin}) the terms $x_i^p$ and $x_i^q$ appear as joinands on the right-hand side of $[\Dset]$, i.e., $\Dset$ contains $d$ and $d'$ such that $d(i)=p$, $d'(i)=q$, and $d(j)=d'(j)=0$ for each $j\neq i$. By the assumption on $\rvec$ and $\const$, $\const+\rvec d = K^{a+b}$ and $\const+\rvec d'=K^a$, for some $a,b\in\N$, with $b>0$ since $p>q$. We have that,
$$K^a(K-1) \leq  K^{a}(K^b-1)=K^{a+b}-K^a = \rvec d- \rvec d'. $$

Also, $\rvec d =\rvec(i)p $ and $\rvec d' = \rvec(i)q$ (by definition of $d,d'$), so we obtain
$$\rvec d - \rvec d' = \rvec(i)p-\rvec(i)q=\rvec(i)(p-q)\leq K^a(p-q) ,$$
where the last inequality follows from $\rvec(i)\leq \rvec(i)q\leq K^a$; note that $q\geq1$.
Combining the inequalities we obtain $K-1 \leq p-q$ and $K\leq 1+p-q$, a contradiction. Hence $[\Dset]$ satisfies $({\star\star}K)$. Furthermore, if all elements from $P$ are larger then $n$, then $[n,P]$ is an expansive equation. Therefore the second claim follows by  Corollary~\ref{undeq}. 
\end{proof}
%
%SECTION: Characterization of spineless equations
%
\section{Characterization of spineless equations}\label{charD} 
In this section we prove that a simple equation is spineless if and only if it satisfies $({\star\star} K)$ for every $K$ sufficiently large. 
%
%Subsection: Basic and simple equations of $\CRL$ as sets of tuples
%
\subsection{Basic and simple equations of $\CRL$ as sets of tuples}
Given the natural ordering of the variable set $\{x_i : i \in \Z^+\}$, note that using the above-mentioned vector notation, every commutative monoid term can be written in the form $\mathbf{x_n}^f$, for some $n \in \Z^+$ and some $n$-tuple $f$ of natural numbers; recall that $\mathbf{x_n}=(x_1, \ldots, x_n)$. If we actually extend our notation to the case where $\mathbf{x_\infty}=(x_i)_{i \in \Z^+}=(x_1, x_2, \ldots)$ and $f$ is a sequence of natural numbers that is eventually constantly zero, then every commutative monoid term is of the form $\mathbf{x_\infty}^f$, and thus it is fully specified by such an $f$. In the following we will work interchangeably in the free monoid over the variable set $\{x_i : i \in \Z^+\}$ and also in the isomorphic monoid $\F$ of eventually-zero sequences of natural numbers. More formally, $\N^{\Z^+}$ denotes the set of all functions from $\Z^+$ to $\N$ and for $f\in \N^{{\Z^+}}$, we define $\supp f :=\{i\in\Z^+: f(i)\neq 0 \}$ to be the \textit{support} of $f$. Then the set $\F := \{f\in \N^{{\Z^+}}: | \supp f|<{\infty} \}$ of all functions of finite support forms a commutative monoid $(\F, +,{\Zero})$, under addition and with unit the constantly-zero function ${\Zero}$. Clearly, this monoid is simply an additive rendering of the free commutative monoid on countably many generators and is isomorphic to the above multiplicative rendering by exactly the map $f \mapsto \mathbf{x_\infty}^f$. Up to now we have favored the multiplicative representation due to its connection with machines, but from now on we will use the additive one as it connects better with the linear algebra arguments of this section. Under this isomorphism the variable $x_i$ maps to the generator  $\mathbf{e}_i$, which has $1$ in the $i$-th entry and $0$ everywhere else.

For reasons that will be clear soon, we view the elements of $\F$ as column vectors and we also consider the bijective set $\F^{{\top}}$ of the row vectors, which are the transposes of the elements of $\F$.
In particular, for $f \in \F$ and $\rvec \in \F^{{\top}}$, the matrix product $\rvec f$ yields a $1\times 1$ matrix, which we identify with the natural number equal to its unique entry. Even though $f$ and $\rvec$ are each of infinite dimension, they both have finite support, so their product is well defined. For a subset $S$ of $\Z^+$ we define $\F_S$ to be the set of eventually zero functions from $S$ to $\mathbb{N}$, so $\F=\F_{\Z^+}$; we identify $\F_S$ with the corresponding subset of $\F$ in the natural way, as every function in $\F_S$ is the restriction to $S$ of the function in $\F$ that is defined to be zero outside $S$. We write $\F_n$ for $\F_{\{1, \ldots, n\}}$. So, if $f \in \F$ with support included in $\{1,...,n\}$, we will identify $f$ with the corresponding element of $\F_n$. We define sets $\F_S^{{\top}}$ and $\F_n^{{\top}}$ in a similar way. Therefore, $\F_n$ is the set of all $n \times 1$ matrices and $\F_n^{{\top}}$ is the set of all $1 \times n$ matrices.

For a set $X\subseteq\F$, we write $ {\rvec}{X} := \{ {\rvec}{f}\in \N: f\in X \}$ and $\supp X:=\bigcup_{f\in X}\supp f$.
For each $n\in\Z^+$, we define the column vector $\Lvec{n}\in\F$ to contain $1$ in its first $n$ entries and $0$ everywhere else. 
A substitution $\subt$ on $\F$ is fully determined by its application on the generators $\mathbf{e}_i\mapsto f_i\in\F$ for each $i\in\Z^+$, and as it is a homomorphism, namely an additive/linear map, its application is given by multiplication of an associated matrix $\Mmat_\subt$; so $\subt(f)=\Mmat_\subt f$. 
Since we only consider finite subsets $A$ of $\F$ in basic equations $[f,A]$, we may view $A$ as a subset of $\F_n$, where $n$ is the largest index in $\supp{A\cup\{f\}}$ and, in this way, will only consider substitutions $\subt:\F_n\to\F_k$, in which case the associated $\Mmat_\subt$ is a $k\times n$ matrix; in this case, we say that $\subt$ is a {\em $k$-variable substitution}.
We will write $\sigma_i\in\F_n^\top$ for the $i$-th row of $\Mmat_\subt$ for each $i\leq k$ and also $\Mmat_\subt=[\sigma_i]_{i=1}^k$. Abusing notation, we will identify $\subt$ with $\Mmat_\subt$ and also we use  $\subt[f,A]$ for the resulting basic inequality. 
As we have seen in the statement of condition $({\star\star})$, $1$-variable substitutions play an important role. Actually, every substitution $\subt$ is rendered as the product of 1-variable substitutions $\rvec_i$ (the ones corresponding to the rows of $\Mmat_\subt=[\sigma_i]_{i=1}^k$) as for every variable $x_j$, $\subt(x_j)=\sigma_1(x_j)\cdot\sigma_2(x_j)\cdots \sigma_k(x_j)$, when using multiplicative notation, and as a sum of $1$-variable substitutions $\rvec_i$ as for every $\m{e}_j$, we have $\subt(\m{e}_j)=\sigma_1(\m{e}_j)+\sigma_2(\m{e}_j)+ \cdots + \sigma_k(\m{e}_j)$, when using additive notation.
% 
%Subsection: Spinal equations
%
\subsection{Spinal equations} 
Let $[f,\spine]$ be a $k$-variable spinal equation. We define $\vs_0:=\Zero$ and $\spine_\Zero=\spine \cup \{ v_0\}$.
Using additive notation, it follows  from Definition~\ref{def: spinal}: 
\begin{enumerate}
\item $\spine$ contains a subset $\spine_+$ consisting of $k\geq 1$ many vectors $\vs_1,...,\vs_k$, where $\vs_j(i)$ is positive if $i=j$ and zero if $i>j$. 
\item $\spine$ is exactly either $\spine_+$ or $\spine_\Zero$.
\item $f$ is a vector in $\F_{k}$ with all entries positive such that $f\nin \spine$. 
\end{enumerate}
We write  $[\vs_1~\cdots~\vs_k]$ for the matrix with columns $\vs_1, \ldots,\vs_k$, in that order. 
Observe that (1) is equivalent to $[\vs_1~\cdots~\vs_k]$ being a $k\times k$ upper-triangular matrix whose diagonal entries are positive. 

Using this additive perspective, we will demonstrate why spinal equations fail to satisfy the condition $({\star\star})$, and thus the argument for register-admissibility in the machines $\acm_K$ found in Lemma~\ref{dredundant} is not applicable to extensions by such equations. In fact, we prove a much stronger property for spinal equations which entails such an argument will fail, not just for our exponential encoding, but for any similar sort of encoding in general.\footnote{Specifically, we mean the following: Let $\acm$ be a $2$-ACM and $\phi:\N\to \N$ any (computable) function with infinite range. Let $\acm_\phi$ be an ACM constructed so that the register contents $\langle n,m\rangle$ of a configuration from $\acm$ are stored as $\langle\phi(n),\phi(m),0,...,0\rangle$ in $\acm_\phi$, and programs constructed so that increments $n\mapsto n+1 $ [decrements $n\mapsto n-1 $] of a register in $\acm$ are simulated by $\phi(n)\mapsto \phi(n+1)$ [$\phi(n)\mapsto \phi(n-1)$] in $\acm_\phi$. Lemma~\ref{injective} ensures that the corresponding argument for register-admissibility is not valid for spinal equations without having more information about $\Acc(\acm)$.} 

To that aim, for a set $S$ of natural numbers, we consider the following property:
\begin{center}
If the exponents in the right-hand side of $[\Dset]$ produced by a 1-variable substitution are in a translation of $S$ (by the same constant), then the substitution instance is trivial.
\end{center}
In symbolic terms this can be written as
\begin{equation}\label{singlestar}\tag{$\star S$}
\begin{array}{c}
\mbox{If for some $\rvec\in\F_n^{\top}$ and  $\const\in \N,$}\\
\mbox{every $\const+\rvec{\vd}$ is in $S$, for $\vd\in \Dset$,}\\ 
\mbox{then there exists $\bar \vd\in \Dset$ such that $\rvec{\bar \vd}= \rvec{\mathbf{1}}$}
\end{array}
\end{equation}
In more compact terms, this can be written as 
$$(\exists \rvec\in \F_n^{\top}, \exists \const\in \N, \const+\rvec\Dset \subseteq S) \Ra   \rvec{\mathbf{1}}  \in \rvec\Dset.$$
Clearly, what we called $(\star K)$ is simply $(\star S)$, where $S$ is the set of all powers of $K$. In Lemma~\ref{injective}, we essentially show that $(\star S)$ fails for any prespinal equation $\varepsilon$ and infinite set $S$. 

\begin{lem}\label{injective}
If a simple equation satisfies $(\star S)$ for an infinite subset $S$ of $\mathbb{N}$, then it is spineless.
\end{lem} 

\begin{proof} 
We argue by contraposition, assuming that a simple equation $\varepsilon$ is prespinal. So there is a substitution $\subt$ such that  $[f,\spine] := \subt \varepsilon$ is a spinal equation, where every column vector of $\spine$ has $k$ entries/rows. We will construct a 1-variable substitution $\tau = [t_1~t_2~\cdots ~t_k]\in \F_k^\top$ such that $\const+ {\tau}{\spine_\Zero}\subseteq S$, for some $C$, and ${\tau}f  \nin {\tau}\spine_\Zero$ (hence also ${\tau}f  \nin {\tau}\spine$, as $\spine\subseteq \spine_\Zero$); let $\spine_\Zero=\{v_0, v_1,\ldots, v_k\}$. This will imply that $\varepsilon$ falsifies $(\star S)$ by the 1-variable substitution $\tau\subt$ and constant $\const$. 

First note that for any $\tau\in \F_k^\top$ we have $\tau \vs_0= 0$, so $\const+ {\tau} \spine_\Zero \subseteq S$ iff $C\in S$ and $\const+ {\tau} \spine_+ \subseteq S$. Observe that $\spmat_+: =[v_1~\cdots~v_k]$ is an upper-triangular $k\times k$ matrix  whose entries are non-negative integers; note the different font from the set $\spine_+$. Furthermore, the determinant $\delta:=\det{\spmat_+} = \vs_1(1)\cdots \vs_k(k)$ is positive since $\vs_n(n)$ is positive for each $1\leq n\leq k$ by definition. Hence $\spmat_+$ is invertible and $\spmat_+^{-1}=\delta^{-1}\adj{\spmat_+}$, where the adjoint $\adj{\spmat_+}$ is an upper-triangular matrix with integer entries which furthermore has positive entries on its diagonal (each of the form $\delta/\vs_n(n)$). 

Now, if $C \in S$ and $\tau\in \F_k^\top$ then
$$\const+ {\tau} \spine_+ \subseteq S \iff \tau \spmat_+\in (S-C)^{k} \iff \tau \in (S-C)^{k}\spmat_+^{-1}.$$
Observe that 
$$(S-C)^{k}\spmat_+^{-1}=(S-C)^{k}\frac{1}{\delta}\adj{\spmat_+}=\left( \frac{S-C}{\delta}\right)^{k}\adj{\spmat_+}.$$
Therefore, 
$$\const+ {\tau} \spine_+ \subseteq S \iff \tau \in \left( \frac{S-C}{\delta}\right)^{k}\adj{\spmat_+}.$$

We claim that there is a $C\in S$ such that the set $(S-C)/\delta$ has an infinite subset in the positive integers. 
Indeed, since $S$ is infinite there exists a coset $C+\N\delta$ that has infinite intersection with $S$, where we can take $C\in S$ without loss of generality; let $\bar\N\subseteq\N$ be the infinite set such that $C+\bar \N\delta $ is the intersection of $S$ with $C+\N\delta$. Hence $\bar\N$  is such an infinite subset of $(S-C)/\delta$, and actually $0\in \bar\N$ since $C\in S$. Consequently, if $\tau \in \bar \N^{k}\adj{\spmat_+}$, then $\const+ {\tau} \spine_+ \subseteq S$. Note that $ \bar \N^{k}\adj{\spmat_+}$ is an infinite set and all if its entries are integers, while we need $\tau \in \F_k^\top$. Therefore, it is enough to be able to find $[x_1~\cdots ~x_k] \in \bar\N^{k}$ such that $ [t_1~t_2~\cdots ~t_k] = [x_1~\cdots ~x_k] \adj{\spmat_+}$, where the entries $t_i$ are nonnegative and further ${\tau}f  \nin {\tau}\spine_\Zero$.

Note that since $\adj{\spmat_+}$ is upper-triangular, the value of $t_n$, for each $n\leq k$,  is determined only by the values $x_1,\ldots,x_n$; $t_n$ is a linear combination of only $x_1,...,x_n$. This allows us to recursively choose the values of the $x_i$'s, in order to specify the values $t_i$ one-by-one. Furthermore, at the recursive step where we have already determined the values of $x_1,...,x_{n-1}$,  the value of $x_n$ can be chosen arbitrarily large from the infinite set $\bar \N$; moreover, in the linear combination specifying $t_n$ the coefficient of $x_n$ is the $(n,n)$-entry of $\adj{\spmat_+}$, which is positive; this allows for the value of $t_n$ to be as large as we want (in particular nonnegative). Therefore, the only thing that we have to ensure is that $x_1,...,x_n$ are chosen in $\bar\N$ so that furthermore  ${\tau}f  \nin {\tau}\spine_\Zero$, i.e.  ${\tau}f \not =\tau v_i$ for all $1 \leq i \leq k$. Below we first prove that ${\tau}f \not =\tau v_k$ and then that ${\tau}f  > \tau v_i$, for all $i<k$.

Since $[f,\spine]$ is a spinal equation, we have $f\nin \spine$ and in particular $f\neq v_k$. Let $m$ be the largest number in $\{1, \ldots, k\}$ such that $f(m)\neq v_k(m)$ and $f(i)=v_k(i)$ for all $i>m$. 
We now define $x_i=0$ for each $i<m$; note that since $0\in\bar\N$, all of these values are in $\bar\N$. Since $t_i$ is a linear combination of $x_1, \ldots, x_i$, we have that $t_i=0$ for $i <m$. 
We define $x_m$ to be any positive number in $\bar\N$, resulting in a positive value for $t_m$, since $x_i=0$ for  $i<m$. Since $f(m)\neq v_k(m)$, we get $t_m f(m)\neq t_mv_k(m)$, and since  $f(i)=v_k(i)$ for all $i>m$ we obtain 
$$\tau f =\sum_{i=m}^k t_if(i)=t_mf(m)+\sum_{i>m} t_iv_k(i) \neq  t_mv_k(m)+\sum_{i>m} t_iv_k(i) = \tau v_k,$$
for all possible values of $t_i$ for $i>m$. So any choice of a positive value for  $x_m$ in $\bar\N$ ensures that  $\tau f\neq \tau v_k$.

For $m<i<k$, we continue choosing positive values for $x_i$ in $\bar\N$ that are large enough to ensure that $t_i$ is nonnegative, as explained above. Finally, at the last step, we have chosen $x_1, \ldots, x_{k-1}$ and therefore  determined the values of $t_1, \ldots, t_{k-1}$. Note that furthermore the values of $\tau v_1, \ldots, \tau v_{k-1}$ have also been determined: since $\spmat_+$ is an upper triangular matrix, we have $v_i(j)=0$ for each $j>i$, so $\tau v_i=t_1v_i(1)+\cdots +t_iv_i(i)=0$ for all $i<m$; in particular even though $t_k$ appears in $\tau$, it does not appear in the values of  $\tau v_1, \ldots, \tau v_{k-1}$. We now choose $x_k\in \bar\N$ so that $t_k>\tau \vs_i$ for all $i<k$. Since $[f,\spine]$ is a basic equation by definition, $f$ is positive in each of its entries  and in particular $f(k)\geq 1$, so we obtain $\tau f \geq t_kf(k)\geq t_k>\tau v_i$ for each $i<k$. 
\end{proof}
\begin{cor}\label{3to1} 
If a simple equation satisfies $(\star)$, then it is spineless.
\end{cor}

Thus we obtain the following from Lemma~\ref{2starTo1star} and Corollary~\ref{3to1}:
\begin{lem}\label{2star2spineless}
If a simple equation satisfies $({\star\star})$ then it is spineless.
\end{lem}
%
%Subsection: Prespinality
%
\subsection{Prespinality}
We will begin with a concrete example of a prespinal equation before illustrating the general case. 
\begin{ex}\label{ex: spine} Consider the $8$-variable simple equation $\varepsilon:stuvwxyz\leq$
$$
 1\vee 
 s w^2 x z^4\vee
 s^2\vee
 s^3 t x^2 y  z\vee
 s^4 t z^2 \vee
  s^5 t w y^2\vee
 s^6 z^4\vee
  s^7 v w x^2 y\vee
 s^8 t^2\vee
 s^9  u v x^2 y,$$ 
where for better readability we use the letters $s,...,z$ for the formal variables $x_1,...,x_8$, in that order; we will use both names for each variable below. 

Since $\varepsilon$ is an $8$-variable equation with $10$ distinct joinands, and all spinal equations in $8$-variables have no more than $9$ distinct joinands, the equation $\varepsilon$ is not spinal. However, it is easily verified that $\varepsilon$ is prespinal, as witnessed by the $3$-variable substitution $\subt$ defined via: 
$s\mapsto 1$, 
$t\mapsto x_1^2$, 
$z\mapsto x_1$,  
$y \mapsto x_1x_2$, 
$v\mapsto x_3$, and $u,w,x\mapsto x_2$. 
Indeed, we have $$\subt \varepsilon: x_1^4 x_2^4 x_3 \leq 1\vee x_1^4 \vee x_1^4x_2^3 \vee x_1 x_2^4 x_3.$$

We name the joinands in the right-hand side of $\varepsilon$ in order of appearance from left to right: $\vd_1=1, \vd_2=sw^2x z^4, \ldots, \vd_{10}=s^9  u v x^2 y$; we do the same for $\subt \varepsilon$: $\vs_0=1$, $\vs_1=x_1^4$, $\vs_2=x_1^4x_2^3$, and $\vs_3=x_1x_2^4x_3$. Also, we define the sets $\Dset=\{\vd_1,...,\vd_{10}\}$ and $\spine=\{\vs_0,...,\vs_3\}$. 
Given this particular ordering of variables and joinands, the set-theoretic equation $\subt \Dset =\spine$ induces the matrix equation $\subt \Dmat = \spmat$ (note the different font for the sets $\Dset, \spine$ and the matrices $\Dmat, \spmat$), where
$$\Scale[.65]{
\subt = 
 \hspace{-5pt}
\begin{array}{c}
\Scale[.95]{
\begin{array}{c cc lcc cc}s& ~t&u& v&w&x& y&z \end{array} }
\\
\begin{bmatrix} 
0& 2&0&0& 0&0&1& 1\\  
0& 0&1&0& 1&1&1& 0\\ 
0& 0&0&1& 0&0&0& 0
\end{bmatrix}
\\
~
\end{array} \hspace{-5pt},
\Dmat=
 \hspace{-5pt}
 \begin{array}{c}
 \phantom{X}\\
 s\\t\\u\\v\\w\\x\\y\\z
 \end{array}
 \hspace{-12pt}
\begin{array}{c}
\Scale[.76]{\begin{array}{cc ccc ccc cc}\vd_1 &\vd_2& \vd_3&\vd_4& \vd_5&\vd_6&\vd_7& \vd_8&\vd_9&\vd_{10} \end{array} }
\\
\begin{bmatrix} 
0&1&2&3 &4&5&6&7&8&  9\\
0&0&0&1 &1&1&0&0&2&  0\\
0&0&0&0 &0&0&0&0&0&  1\\
0&0&0&0 &0&0&0&1&0&  1\\
0&2&0&0 &0&1&0&1&0&  0\\
0&1&0&2 &0&0&0&2&0&  2\\
0&0&0&1 &0&2&0&1&0&  1\\ 
0&4&0&1 &2&0&4&0&0&  0
\end{bmatrix}
\end{array}
 \hspace{-5pt},
 \spmat=
 \hspace{-5pt}
 \begin{array}{c}
\Scale[.76]{\begin{array}{cc ccc ccc cc}\vs_0 &\vs_2& \vs_0&\vs_2& \vs_1& \vs_2&\vs_1& \vs_3&\vs_1&\vs_3 \end{array} }
\\ 
\begin{bmatrix} 
0&4&0& 4&4&4& 4&1&4 &1 \\  
0&3&0& 3&0&3& 0&4&0 &4\\  
0&0&0& 0&0&0& 0&1&0 &1
\end{bmatrix}
\\
~
\end{array}
 }$$
Note that the $(i,j)$-entry of $\Dmat$ represents the degree of the $i$-th variable $x_i$ in the joinand $\vd_j$, and the $(j,i)$-entry of $\subt$ represents the degree of $x_j$ in $\subt(x_i)$.

Note that, by omitting $\vs_0=\Zero$, the $3\times 3$ matrix $[\vs_1~\vs_2~\vs_3]$ is upper triangular with a positive diagonal (as demanded in the definition of $\subt\varepsilon$ being spinal) and this is the reason for the particular naming of $\vs_0, \vs_1, \vs_3,$ in that order. In turn, given this order, the substitution partitions the set of joinands $\Dset$, i.e., the columns of $\Dmat$ into 
$\Dset_0=\{\vd_1,\vd_3\}$, $\Dset_1=\{\vd_5,\vd_7,\vd_9 \}$, $\Dset_2=\{\vd_2,\vd_4,\vd_6 \}$ and $\Dset_3=\{\vd_8,\vd_{10}\}$, so that $\subt\Dset_j=\{\vs_j\}$, for all $j$. Guided by this ordering, we further rearrange the columns of $\spmat$ into a new matrix $\spmat'$ and the columns of $\Dmat$ into a new matrix $\Dmat'$, where the ordering of the columns within each $\Dmat_i$ is done randomly. If we represent $\Dmat'$ symbolically as $[\Dset_0~\Dset_1~\Dset_2~\Dset_3]$, then we also have $\subt \Dmat' = [\subt\Dset_0~\subt\Dset_1~\subt\Dset_2~\subt\Dset_3]$ and the equation $\subt \Dmat' =\spmat'$. 

We can improve the presentation of this equation even more by putting $\subt$ and $\Dmat'$ in a triangular form, at least in blocks. More specifically, we now rearrange the rows of $\Dmat'$ and simultaneously the columns of $\subt$ (this corresponds to permuting the variables of $\varepsilon$) to obtain new matrices $\Dmat''$ and $\subt '$, yielding the equation $\subt' \Dmat''  = \spmat'$:

$$\Scale[0.72]{ 
\begin{array}{c}
\Scale[.95]{\begin{array}{c cc lcc cc}s& ~t&z& w&x&y& u&v \end{array} }
\\
\left[
\begin{array}{c|cc|ccc|cc}
0& 2&1& 0&0&1& 0&0\\ \hdashline 
0& 0&0& 1&1&1& 1&0\\ \hdashline
0& 0&0& 0&0&0& 0&1\end{array}
\right] 
\\
~
\end{array}
\hspace{-10pt}
\begin{array}{c}
\Scale[.77]{\begin{array}{cc ccc ccc cc}\vd_1 &\vd_3& \vd_5&\vd_7& \vd_9&\vd_2&\vd_4& \vd_6&\vd_8&\vd_{10} \end{array} }
\\
\left[
\begin{array}{c c | ccc | c c c | cc }
0&2& 4&6&8& 1&3&5& 7&9\\
\hline
0&0& 1&0&2& 0&1&1& 0&0\\
0&0& 2&4&0& 4&1&0& 0&0\\ 
\hline
0&0& 0&0&0& 2&0&1& 1&0\\
0&0& 0&0&0& 1&2&0& 2&2\\
0&0& 0&0&0& 0&1&2& 1&1\\ 
\hline
0&0& 0&0&0& 0&0&0& 0&1\\
0&0& 0&0&0& 0&0&0& 1&1
\end{array}
\right]
\end{array}
\hspace{-5pt}
 = 
\hspace{-5pt}
\begin{array}{c}
\Scale[.8]{\begin{array}{cc ccc ccc cc}\vs_0 &\vs_0& \vs_1&\vs_1& \vs_1& \vs_2&\vs_2& \vs_2&\vs_3&\vs_3 \end{array} }
\\ 
\left[\begin{array}{c c | c c c | c c c |cc} 
0&0& 4&4&4& 4&4&4 &1&1 \\ \hdashline 
0&0& 0&0&0& 3&3&3 &4&4\\ \hdashline 
0&0& 0&0&0& 0&0&0 &1&1
\end{array}
\right] 
\\
~
\end{array}
\hspace{-5pt}.}$$
Finally, we observe that the rearrangement of the rows results in a partition of the set of rows such that the two partitions (of the set of rows and the set of columns) induce a blocking (given by the solid lines above) that has an upper-triangular shape. We denote by $\Dmat^\bl$ and $\subt^\bl$ the resulting block matrices, and the equation $\subt' \Dmat''  = \spmat'$ of matrices yields the equation $\subt^\bl \Dmat^\bl = \spmat^\bl$ of block matrices. We call the elements of $\Dmat^\bl$ {\em blocks} and they are submatrices of $\Dmat''$; we denote by $(\Dmat^\bl)_{ij}$  the $(i+1,j+1)$-block of $\Dmat^\bl$, where $0\leq i,j\leq 3$. We observe that 
\begin{enumerate}
	\item each $(\Dmat^\bl)_{ij}$ is the zero matrix when $i>j$ (all blocks below the diagonal are zero matrices) and 
	\item each row and each column of $(\Dmat^\bl)_{jj}$ contains a nonzero entry for $j\geq1$ (in the diagonal blocks no row and no column is fully zero, with the possible exception of the top left block $(\Dmat^\bl)_{00}$).
\end{enumerate}  
\end{ex}

We call such a partition of the matrix into a block matrix with these two features a \emph{blocking} of the matrix. Each blocking specifies an (ordered) partition of the set of rows and an (ordered) partition of the set of columns of a matrix in the obvious way (with the provision that the first class in this ordered list may be the empty set), but each of these two partitions is special as we will explain. Given a column partition, we define below an associated list of sets of rows. Whenever the original partition comes from a blocking, the resulting list is actually an (ordered) partition (every set in the list, except possibly the first one, is non-empty). The same holds with the roles of rows and columns swapped. We will explain that blockings correspond bijectively to column-partitions that happen to induce ordered row-partitions and to row-partitions that happen to induce ordered column-partitions.

Given an $I\times J$ matrix $\Dmat$, formally viewed as a function from $I \times J$, as usual $\Dmat_{ij}$ denotes its entry in the $i$-th row and $j$-th column, for $i\in I$ and $j\in J$; usually $I$ and $J$ are taken to be initial segments of the positive integers as in the example above. We denote by $\Dmat_{i\_}$ the $i$-th row and by $\Dmat_{\_j}$ the $j$-th column of $\Dmat$. Given an \emph{ordered partition} $(C_0, C_1, \ldots, C_k)$ of $J$ (i.e., $C_0$ may be empty, but not all of $J$, and the remaining list forms a partition of $J$) we define the list $(R_0, R_1, \ldots, R_k)$ of subsets of $I$ as follows:
for $m \geq 1$, $R_m$ contains those $i \in I$ such that the entry $\Dmat_{in}$ is zero for $n$ belonging to parts $C_\ell$ with $\ell<m$ and there is a non-zero entry $\Dmat_{in}$ for some $n$ belonging to the part $C_m$. In other words, if we group the columns of $\Dmat$ according to the ordered partition, then $R_m$ corresponds to those rows that are fully zero on all columns before $C_m$ and are not fully zero on $C_m$. We define $R_0$ as containing the remaining $i$'s that are not in $R_1 \cup \cdots \cup R_k$. (Note that we allow our ordered partitions to have an optional initial empty part $R_0$.) In the example above, the sets $R_n$ are all non-empty, thus resulting into a partition of the set of columns; however, this may not be the case when $(C_0, C_1, \ldots, C_k)$ of $J$ is an arbitrary ordered partition of $J$. Note that whenever  $(R_0, R_1, \ldots, R_k)$ is a partition, we can define a partition of $I\times J$ into blocks of the form $R_m \times C_n$ with the feature that, for $n\geq 1$, each block $R_n \times C_n$ is such that no column and no row of $\Dmat$ in that block is fully zero. Therefore, blockings correspond to column-partitions that happen to induce ordered row-partitions. Often, instead of writing a partition $(C_0, C_1, \ldots, C_k)$ of the column index set $J$, we will be writing the partition $(\Dset_0, \Dset_1, \ldots, \Dset_k)$ of the set $\Dset$ of the corresponding columns.

Conversely, given an ordered partition $(R_0, R_1, \ldots, R_k)$ of $I$ we define a list  $(C_0, C_1, \ldots, C_k)$ of subsets of $J$ as follows: for $n \geq 1$, 
$C_n$ contains those $j \in J$ such that the entry $\Dmat_{mj}$ is zero for $m$ belonging to parts $R_\ell$ with $\ell >n$, and there is a non-zero entry $\Dmat_{mj}$ for some $m$ belonging to the part $C_n$. In other words, if we group the rows of $\Dmat$ according to the ordered partition, then $R_n$ corresponds to those rows that finish with zeros on all rows after $C_n$ and are not fully zero on $C_n$. We define $C_0$ as containing the remaining $i$'s that are not in $C_1 \cup \cdots \cup C_k$. Again we can see that this yields an ordered partition (i.e., the sets $C_n$, $n \geq 1$, are non-empty) iff this corresponds to a blocking of the matrix. 

Given a blocking $\bl$ of $I\times J$ on a matrix $\Dmat$, we obtain the block matrix $\Dmat^\bl$ and observe that it has an upper-triangular form. In the following we will consider a blocking given by either its ordered row-partition or its ordered column-partition. In analogy with our notation for entries, rows and columns of a matrix, 
we define $\Dmat^\bl_{mn}:=\{\Dmat_{ij}: i \in R_m, j \in C_n\}$, $\Dmat^\bl_{m \_}=\{\Dmat_{ij}: i \in R_m\}$ and $\Dmat^\bl_{\_ n}=\{\Dmat_{ij}: j \in C_n\}$.  If $\Dset$ is a set of column vectors, we say $\bl$ is a \emph{blocking of} $\Dset$ if it is a blocking of a matrix $\Dmat$ whose set of columns form $\Dset$, and we define $\Dset_{j}^\bl$ to be the $j$-th set of column vectors $\Dmat^\bl_{\_ j}$. Observe that if $\bl=(R_0,\ldots,R_k)$ is a row blocking for $\Dset$, then by how the associated list of columns are defined, we get $\Dset^\bl_j=\{d\in \Dset: \supp{d}\cap R_j \neq\emptyset\}\setminus (\Dset^\bl_{j+1}\cup\cdots \Dset^\bl_k)$, for each $j\geq 1$, and $\Dset^\bl_0 = \Dset\setminus (\Dset^\bl_1\cup\cdots \Dset^\bl_k)$.

The blocking in the Example~\ref{ex: spine} is induced by $\subt$ in the sense that the corresponding partition on the set $\{1,\ldots,10\}$ of columns of $\Dmat$ along the vertical lines in $\Dmat^\bl$ into the sets $C_0=\{1, 3\}, C_1=\{5,7,9\}$, $C_2=\{2,3,6\}$, and $C_3=\{8,10 \}$ is given by the stipulation that $C_n$ is exactly the set of the columns that are mapped by $\subt$ to the same column vector, $v_n$, of $V$, for all $n \in \{0,1,2,3\}$. The partition of the set of rows $\{1,\ldots,8\}$ of $\Dmat$ (along the horizontal lines of $\Dmat^\bl$) yields the sets $R_3=\{ 3,4\}$, $R_2=\{5,6,7 \}$, $R_1=\{ 2,8\}$, and $R_0=\{1\}$. 

As in the example, every substitution $\subt$ that maps an $I \times J$ matrix $\Dmat$ to a spine $\spmat=[\vs_0~\vs_1~\cdots ~\vs_k]$ induces a blocking $\bl$ via a partitioning of the columns: $C_n=\{j \in J : \subt \Dmat_{\_ j}=\vs_n\}$. The blockings induced by substitutions enjoy further properties (in addition to yielding upper-triangular block matrices with diagonal blocks that have no fully zero row or column). 
Returning to the example we see that $\subt \Dmat^\bl_{\_ j}=\vs_j$ for each $0\leq j\leq 3$. We say that a substitution is a \emph{solution} for a set/matrix of column vectors if it sends all of the vectors of the set/matrix to the same vector. 
In this terminology, $\subt$ is a solution for each of the sets $\Dmat^\bl_{\_ 0}$, $\Dmat^\bl_{\_ 1}$ and $\Dmat^\bl_{\_ 2}$. Note that a substitution  is a solution for a set/matrix iff each of its rows is a solution for it. 

Also, looking at the induced partition on the rows, the $R_m$ part of each row $\rvec_m$ of $\subt$ is not fully zero and all of its elements are non-negative.
For $f\in \F$ and $T\subseteq\Z^{+}$ we say $f$ is {\em $T$-positive} if $f_T>\Zero,$ i.e., $f_T\neq\Zero$ and $f(i)\geq 0$ for each $i\in T$; put differently $f \not \in \F_{T^c}$, where $T^c$ is the complement of $T$.
In this terminology, the row $\rvec_m$ of $\subt$ is $R_m$-positive, for each $m$. Finally, we note that $\rvec_3$ is an element of $\F^\top_{R_3}$, $\rvec_2$ is an element of $\F^\top_{R_3 \cup R_2}$, and $\rvec_1$ is an element of $\F^\top_{R_3 \cup R_2 \cup R_1}$. In general, for every blocking defined by a substitution $\subt$ with respect to a spine, $\rvec_m$ is an element of $\F^\top_{R^+_m}$, where $R^+_m:=R_m \cup \cdots \cup R_k$. For $T\subseteq S$, we say $f$ is {\em $(T,S)$-positive} if $f$ is $T$-positive and $f \in \F_S$; put differently $f \in \F_S$ and $f \not \in \F_{T^c}$. Therefore, the row $\rvec_m$ of $\subt$ is $(R_m, R_m^+)$-positive, for each $m$. 

Given a row blocking $\bl=(R_0,\ldots,R_k)$ of a set $\Dset$, a $1$-variable substitution $\sigma \in \F^\top$ and $1\leq i \leq k$, if $\sigma$ is  $(R_i,R_i^+)$-positive and $\sigma$ is a solution for each set of columns $\Dset^\bl_{j}$, then we say that $\sigma$ is a  \emph{$(\bl, i)$-solution} for $\Dset$. 
Finally, we say that a $k$-variable substitution $\subt$ is a {\em $\bl$-solution for $\Dset$}, where $\bl=(R_0,\ldots,R_k)$, if for all $i$, the $i$-th row $\sigma_i$ of $\subt$ is a \emph{$(\bl, i)$-solution} for $\Dset$.
$$\Scale[0.73]{ 
\begin{bmatrix} 
\Zero & \subt^{\bl}_{11}& \cdots & \subt^{\bl}_{1i}&\cdots& \subt^\bl_{1k}\\
\vdots & \vdots& \ddots & \vdots & \vdots & \vdots \\
\Zero & \Zero &\cdots & \subt^{\bl}_{ii} & \cdots & \subt^{\bl}_{ik}\\
\vdots & \vdots& \ddots & \vdots & \ddots & \vdots \\
\Zero&\Zero&\cdots &\Zero &\cdots &\subt^{{\bl}}_{kk}
\end{bmatrix}
\begin{bmatrix} 
\Dset^{\bl}_{00} & \Dset^{\bl}_{01}& \cdots & \Dset^{\bl}_{0i}&\cdots& \Dset^{\bl}_{0k}\\
\Zero & \Dset^{\bl}_{11}&\cdots &\Dset^{\bl}_{1i}&\cdots &\Dset^{\bl}_{1k}\\
\vdots & \vdots& \ddots & \vdots & \vdots & \vdots \\
\Zero & \Zero &\cdots & \Dset^{\bl}_{ii} & \cdots & \Dset^{\bl}_{ik}\\
\vdots & \vdots& \ddots & \vdots & \ddots & \vdots \\
\Zero&\Zero&\cdots &\Zero &\cdots &\Dset^{{\bl}}_{kk}
\end{bmatrix} =
\begin{bmatrix} 
\Zero & \vs_1(1)& \cdots & \vs_i(1)&\cdots& \vs_k(1)\\
\vdots & \vdots& \ddots & \vdots & \vdots & \vdots \\
\Zero & \Zero &\cdots & \vs_i(i) & \cdots & \vs_k(i)\\
\vdots & \vdots& \ddots & \vdots & \ddots & \vdots \\
\Zero&\Zero&\cdots &\Zero &\cdots &\vs_k(k)
\end{bmatrix}
},
$$

The following lemma and theorem then follow by the definition of $\bl$-solutions.

\begin{lem}\label{b-solutions}
If a $k$-variable substitution $\subt$ is a $\bl$-solution for $\Dmat$, then the matrix $[\vs_1~\vs_2~\cdots~\vs_k]$, where $\vs_j=\subt \Dmat^\bl_{*j}$, 
is a $k\times k$ upper-triangular matrix whose diagonal contains positive entries.
\end{lem}

\begin{thm}\label{spinalform} 
An $n$-variable simple equation $[\Dset]$ is prespinal if and only if $\Dset$ has a $\bl$-solution $\subt$, for some row blocking $\bl$, such that $\subt\Lvec{n}$ and $\subt\Dset^{\bl}_{k}$ differ.
\end{thm}
 
The importance of Theorem~\ref{spinalform} is that it characterizes the notion of prespinality without a reference to a spine.

 Using Theorem~\ref{spinalform} we now  characterize which equations are spineless. For instance, we can verify that equations (iii)-(v) from Table~\ref{Rtable} are spineless. In each equation, the set $\Dset$ has only one blocking $\bl$, given by $(\Dset^\bl_0,\Dset^\bl_1)$ where $\Dset^\bl_1$ contains all the nonzero columns from $\Dset$. We see that there is no nonzero $(\bl,1)$-solution for either equation (iii) or (iv). In the case of equation (v), the only $(\bl,1)$-solutions are scalar multiples of $\rvec=[1 \; 1 \; 1]$, but $\rvec\Lvec{3}\in \rvec\Dset$. Therefore equations (iii)-(v) are spineless.
%   
%Subsection: Spineless equations satisfy $(\star\star)$
%
\subsection{Spineless equations satisfy $(\star\star)$}
Theorem~\ref{spinalform} provides the foundational link between an equation being spineless and satisfying $(\star\star)$, namely by the (non-)existence of certain $1$-variable substitutions viewed as solutions to particular linear systems. 

We prove that every spineless equation satisfies $({\star \star} K)$, for sufficiently large $K$, contrapositively. 
So, we assume that a simple equation $[\Dset]$ satisfies the antecedent of $({\star \star} K)$ and  in particular there exists a (nontrivial) $1$-variable substitution $\rvec$ for which $\rvec\Dset$ is contained in some shift of $K^\N$. 
The following lemma ensures that, if $K$ is chosen large enough, $\rvec$ \emph{induces} a blocking $\bl$ on $\Dset$ in the sense that the naturally ordered column-partition $(\Dset_{0},\ldots, \Dset_{k})$ of $\Dset$ induces $\bl$, i.e., $\Dset_j=\Dset^\bl_{j}$ for each $j$. 
Here we conventionally order the sets so that $\rvec \Dset_{j}< \rvec \Dset_{j+1}$ for $j<k$, and we always take $\Dset_{0}$ to be the (possibly empty) set of all $\vd\in \Dset$ such that $\rvec \vd= 0$. Note that if a $\rvec$ induces a blocking then it must be that $k\geq 1$ and so $\rvec$ is not the zero vector.

For a finite set $\Dset\subseteq\F$, we define $\Delta \Dset=\sum_{i=1}^n\max\{|\vd(i)-\vd'(i)|:\vd,\vd'\in \Dset \}$.

\begin{lem}\label{Ksol2}
Assume that for a finite $\Dset\subseteq\F$ and for some 1-variable substitution $\rvec$ that is nontrivial on $\Dset$, $\rvec \Dset$ is contained in some shift of $K^\N$, where $K>\Delta \Dset+1$. Then $\rvec$ induces a blocking on $\Dset$. 
\end{lem}

\begin{proof}
Suppose $\Dset\subseteq \F_n$ and let $(\Dset_0,...,\Dset_k)$ be the column partition of $\Dset$ such that $0=\rvec\Dset_0<\cdots <\rvec \Dset_k$. 
We will show that the associated list $(R_0, R_1, \ldots, R_k)$ of sets of rows is an ordered row-partition, i.e., each of the $R_1, \ldots, R_k$ is not empty. Recall that $R_j$ is the set of all indices $i\in \nset{n}$ such that $\Dmat_{ij}\neq \Zero$ but $\Dmat_{il}=\Zero$ for each $l<j$, where $\Dmat: =[\Dset_0\cdots\Dset_k]$ (the order of the columns within each $\Dset_i$ plays no role). We first prove that $R_k$ induces $\Dset_k$, 
i.e., each $\vd\in\Dset_{k}$ must have some nonzero entry/row that is zero in every $\vd'\in \Dset\setminus \Dset_k$. 
If not, then there exists some $\vd\in\Dset_{k}$ such that for each $i\in\supp{\vd}$ there exists $j<k$ and vector $\vd'\in \Dset_j$ with $i\in\supp{\vd'}$. By assumption, $\rvec \vd$ and $\rvec \vd'$ are in the same shift of $K^\N$, say $K^\N-C$ for some $C\in \N$, and because of our ordering convention, $\rvec \vd' < \rvec \vd$. So, 
$\rvec \vd = K^{a+1} -C$ and $\rvec \vd'\leq K^a- C$ for some $a\geq0$. 
 Since $\vd'(i) \geq 1$, we have $\rvec(i)\leq\rvec(i)\vd'(i) \leq \rvec \vd'\leq K^a$, so
$$K^a(K-1)= K^{a+1}-K^a\leq\rvec{\vd}-\rvec{\vd'}\leq \rvec{|\vd-\vd'|}\leq  K^a\Delta\{\vd,\vd'\} \leq K^a\Delta \Dset,$$
where the entries of $|\vd-\vd'|$ are the absolute values of the corresponding entries of $\vd-\vd'$.
Therefore,  $K\leq \Delta \Dset+1$, which contradicts the assumption on the size of $K$. 
So, $R_k$ is nonempty and we obtain $\Dset_k=\{\vd\in\Dset:\supp{\vd}\cap R_k\neq \emptyset\}$.

Continuing in this way for $1\leq j<k$, set  $\Dset'= \Dset_0\cup\cdots \cup \Dset_{j} $. Since $\Dset'\subseteq \Dset$ implies $\Delta\Dset'\leq\Delta\Dset$, the same argument shows each $\vd\in \Dset_j$ contains a nonzero entry that is zero for each $\vd'\in \Dset'\setminus \Dset_j$. 
Hence $R_j$ induces $\Dset_j$, i.e., $$\Dset_j= \{\vd\in \Dset: \supp{\vd}\cap R_j\neq\emptyset \}\setminus(\Dset_{j+1}\cup\cdots\cup \Dset_k). $$
Given $R_0=\supp{\Dset}\setminus R^{+}_1$ by definition, we conclude that $\bl=(R_0,...,R_k)$ is a blocking on $\Dset$ such that $\Dset_{j}^\bl=\Dset_j$ for each $j\leq k$, i.e., $\rvec$ induces the blocking $\bl$ on $\Dset$. 
\end{proof}

Suppose $\rvec$ induces a row blocking $\bl=(R_0, R_1, \ldots, R_k)$ of $D$. Since $\rvec \Dset_1>0$ by definition, it must be that $\rvec$ is $R_1$-positive, and since $\rvec \Dset_0=0$ by definition, it must be that the support $\rvec$ is contained in $R_1^+$. That is, $\rvec$ must be $(R_1,R_1^+)$-positive. Since it is also a solution for each $\Dset^\bl_{j}$ by definition, this proves the following lemma.

\begin{lem}\label{NewSubs}
If a $1$-variable substitution $\rvec$ induces a blocking $\bl$ on $\Dset$, then it is a $(\bl,1)$-solution for $\Dset$. In particular, if $\subt$ is any $\bl$-solution for $\Dset$, the substitution obtained by replacing the first row of $\subt$ by $\rvec$ is also a $\bl$-solution for $\Dset$.
\end{lem}

We now demonstrate the converse of Corollary~\ref{2star2spineless} by proving the following stronger statement. The proof relies on Lemma~\ref{cor: ArbKspines}, which we prove in the next section.

\begin{lem}\label{spineless2star}
A spineless equation satisfies $({\star\star} K)$ for all sufficiently large $K$.
\end{lem}

\begin{proof}
Toward establishing the contrapositive, we assume that the $n$-variable simple equation $[\Dset]$ fails $({\star\star} K)$ for infinitely many $K\in\N$. Then there are infinitely many pairs of 1-variable substitutions $\rvec,\rvec'$ witnessing such failures that furthermore induce  pairs of blockings for $\Dset$ by Lemma~\ref{Ksol2}. Since there are only finitely many blockings, and thus finitely many pairs of them, there must exist blockings $\bl=(R_0,...,R_k)$ and $\cl=(R'_0,...,R'_l)$ that are witnessed infinitely often as a pair. By Lemma~\ref{cor: ArbKspines}, $[\Dset]$ has a $\bl$-solution $\subt$ and a $\cl$-solution $\subt'$. By Lemma~\ref{NewSubs}, we can replace the first rows of 
$\subt$ and $\subt'$ with 1-variable substitutions $\rvec_1$ and $\rvec_1'$ (from the above-mentioned infinitely many) witnessing the failure of $({\star\star} K)$ and inducing $\bl$ and $\cl$, respectively.

If either $\subt\Lvec{n}\neq\subt\Dset^\bl_{k}$ or $\subt'\Lvec{n}\neq\subt'\Dset^\cl_{l}$, then $[\Dset]$ is prespinal by Theorem~\ref{spinalform}, and we are done. If not, we have $\subt \Lvec{n}=\subt \Dset^\bl_{k}$ and likewise $\subt' \Lvec{n}=\subt' \Dset^\cl_{l}$. In particular, $\rvec_1\Lvec{n}=\rvec_1 \vd$ for every $\vd\in \Dset^\bl_{k}$ and $\rvec_1'\Lvec{n}=\rvec_1' \vd$ for every $\vd\in \Dset^{\cl}_{l}$, and hence
for every $\vd\in \Dset^{\bl}_{k} \cap \Dset^{\cl}_{l}$. If $\Dset^{\bl}_{k}$ and $\Dset^{\cl}_{l}$ were not disjoint, there would be a $\vd\in \Dset$ such that $\rvec_1\Lvec{n}=\rvec_1 \vd$ and $\rvec_1'\Lvec{n}=\rvec_1' \vd$, which would imply that $\rvec_1,\rvec_1'$ satisfy $({\star\star}K)$, contradicting their choice above.
 Hence $\Dset^{\bl}_{k}$ and $\Dset^{\cl}_{l}$ are disjoint and by the definition of blockings, $R_k$ and $R'_l$ are also disjoint. 

 As a result the row partition $\mathfrak{a}: =(B_0, B_1)$, where $B_1=R_k\cup R'_l$ and $B_0=\{1,...,n\}\setminus B_1$, is actually a row blocking on $\Dset$ which furthermore induces the partition of columns $\Dset^\mathfrak{a}_1=\Dset^\bl_{k}\cup \Dset^\cl_{l}$ and $\Dset^\mathfrak{a}_0=\Dset\setminus \Dset^\mathfrak{a}_1$. We will construct a $1$-variable substitution $\alpha$ that will serve as an $\mathfrak{a}$-solution for $\Dset$ witnessing the prespinality of $[\Dset]$. 

Let $\rvec_k$ and $\rvec'_l$ be the bottom rows of $\subt$ and $\subt'$, respectively. Since $\rvec_k$ is a $(R_k,R^+_k)$-positive solution for $\Dset^\bl_{k}$, the value $t= \rvec_k\Dset^\bl_{k}$ is positive, and since $R_k$ and $R_l'$ are disjoint, $\rvec_k\Dset^{\cl}_{l}$ is zero; in detail $R'_l$ contains columns that appear in earlier blocks than $R_k$. Similarly, $t'= \rvec'_l\Dset^{\cl}_{l}$ is positive and $\rvec'_l\Dset^\bl_{k}$ is zero. Note that $\rvec_k\Lvec{n}=t$ and $\rvec'_l\Lvec{n}=t'$ follow by the fact that $\subt \Lvec{n}=\subt \Dset^\bl_{k}$ and $\subt' \Lvec{n}=\subt' \Dset^\cl_{l}$.

We now define the $1$-variable substitution ${\alpha}= t'\rvec_k+t\rvec'_l$. Since $R_k$ and $R'_l$ are disjoint, it follows that ${\alpha}\Dset^\bl_k=t't+t0=tt'$ and ${\alpha}\Dset^{\cl}_l=t'0+tt'=tt'$, so ${\alpha}\Dset^\mathfrak{a}_1=tt'$; also ${\alpha}\Lvec{n}=t't+tt'=2tt'$. In the case when $\Dset^{\mathfrak{a}}_0$ is nonempty, we have ${\alpha} \Dset^{\mathfrak{a}}_0=0$, because  $\rvec_k\Dset^\bl_{j}=0$, for $j<k$, and $\rvec'_l\Dset^{\cl}_{i}=0$, for $i<l$.
Since $t,t'>0$, we have $2tt'>tt'>0$, and therefore it follows that $\alpha \Lvec{n}\nin \alpha \Dset$ and so $[\Dset]$ is prespinal by Theorem~\ref{spinalform}.
\end{proof}

Lemmas~\ref{2star2spineless} and~\ref{spineless2star} establish the equivalences between a simple equation being spineless and satisfying $({\star\star})$, and as well as satisfying $(\star)$.
\begin{thm}\label{final}
A simple equation is spineless iff it satisfies $({\star\star}K)$ for every sufficiently large $K$.
\end{thm} 

Therefore, to establish Theorem~\ref{final} we must prove Lemma~\ref{cor: ArbKspines}.
%
%Subsection: Solutions in $\R^n$
%
\subsection{Solutions in $\R^n$} 
The goal of this section is to prove 
Lemma~\ref{cor: ArbKspines}. To address this, we recall a {theorem of alternatives} for positive solutions to linear systems accompanied by conventional terminology (see \cite{Roman}). 
 
 Let $v\in\R^n$ (viewed as a row vector) and $M\subseteq \R^n$. We say that $v$ is {\em orthogonal} to  $M$ if $v M=0$. We say that $v$ is \textit{strictly positive} if $v\neq {\Zero}$ and $v(i)\geq0$ for each $i\in \nset{n}$. The set ${X}^n_{+}$ denotes the set of all strictly positive vectors in $X^n$, called the \textit{strictly positive orthant} in ${X^n}$, where $X\in\{\Z,\mathbb{Q},\R \}$. The following (folklore) theorem is equivalent to Farkas's Lemma (for instance, see Theorem~27 \cite{Perng}).

 \begin{thm}\label{Palter}
Let $M\subseteq\R^n$ be nonempty set of vectors and $i\in \nset{n}$ a fixed index. Then exactly one of the following holds:
\begin{enumerate}
\item there exists a strictly positive vector $v$ orthogonal to $M$ where $v(i)>0$, or
\item there exists a strictly positive vector $w\in \mathrm{span}(M)$ where $w(i)>0$.
\end{enumerate} 
 \end{thm}
 
Note that $\mathrm{span}(M)_S = \mathrm{span}(M_S)$ for any $M\subseteq \R^n$ and $S\subseteq \nset{n}$; here the subscript $S$ denotes restriction to $S$.

\begin{cor}\label{span1}
Let $M\subseteq \R^n$ and $T\subseteq  S \subseteq \nset{n}$ be non-empty. If there is no $T$-positive vector in $\R^{|S|}_+$ orthogonal to $M_S$ then there exists $L\in\N$ such that, for any $v\in\R^n_+$ orthogonal to $M$, $v(i)\leq L\cdot\max\{v(j): j\in \nset{n}\setminus  S\}$ for each $i\in T$. 
\end{cor}
\begin{proof} By Theorem~\ref{Palter}, for each $i\in T$ there exists a strictly positive (in $\R^{|S|}$) vector $w_i \in \mathrm{span}(M_S)$ where $w_i(i)>0$. Then $\bar w:=\sum_{i\in T}w_i\in \mathrm{span}(M_S)$ is strictly positive with $T\subseteq\supp{\bar w}$. Let $w\in \mathrm{span}(M)$ be such that $w_S = \bar w$. Note that since $T$ is nonempty and $T\subseteq \supp{w_S}$, it follows that $t:=\min\{w(i): i\in T\}$ is positive. Set $S^c:=\nset{n}\setminus S$ and $m:=\sum_{j\in S^c}|w(j)|$, where we take the empty sum to be zero. Define $L$ to be the smallest positive integer greater than $m/ t$.

Now, suppose $v\in\R^n_+$ is orthogonal to $M$. Then $v w^{{\top}}=0$ and hence
$$ \sum_{j\in  S}  w(j)v(j) = \sum_{j\in S^c} - w(j)v(j).$$
Let $N$ denote this common value. 
Considering the right-hand side of the equation, we obtain $N\leq m\cdot \max\{v(j): j\in S^c\}.$  
Considering the left-hand side of the equation, $tv(i)\leq w(i)v(i)\leq N$, for all $i\in T$, since $T\subseteq S$ and $w(i)\geq t >0$. 
Therefore, for all $i\in T$, we deduce $v(i)\leq  N/t\leq L\cdot \max\{v(j): j\in S^c\}$.
\end{proof}

\begin{lem}\label{cor: ArbKspines}
Let $\Dset$ be a finite subset of $\F_n$ and $\bl$ be a blocking on $\Dset$. If there are infinitely many $K\in \N$ for which there exists a $1$-variable substitution $\rvec$ that induces $\bl$ and $\rvec\Dset$ is contained in some shift of $K^\N$, then $\Dset$ has a $\bl$-solution.
\end{lem}
\begin{proof}
Working contrapositively, we assume that  $\Dset$ has no $\bl$-solution. We will show that if $K> 1$ is such that  there exists a $1$-variable substitution $\rvec\in \F^\top_n$ that induces a blocking $\bl$ on $\Dset$ and that $\rvec\Dset$ is contained in some shift of $K^\N$,  then $K$ can be no larger than a certain multiple of $\Delta \Dset$. 
Let $\bl=(R_0,...,R_k)$, where $k\geq1$.

For any nonempty $A\subseteq \F_n$, fix $\bar a\in A$ and define $\bar A:=\{a-\bar a: a\in A \}$; note that the entries of the column vectors are in $\Z$. For a set of rows $S\subseteq\nset{n}$, $\rvec$ is a solution for $A_ S$ iff $\rvec A_S$ is a singleton iff  $\rvec A_S=\rvec \bar a$ iff $\rvec \bar A_S=\{0\}$ iff $\rvec$ is orthogonal to $\bar A_S$ in $\R^{n}$ (regardless of the choice $\bar a\in A$). Hence, if $T\subseteq  S$, then there exists a $T$-positive solution for $A$ in $\F^\top_S$ iff there exists a $T$-positive solution for $A_ S$ iff there exists a $T$-positive vector of 
$\R^{|S|}_+$ orthogonal to $\bar A_S$.\footnote{For the reverse direction, if $v\in \mathbb{R}_+^{|S|}$ is orthogonal to $\bar A_S$, then since $\bar A_S$ has integer entries, by Gaussian Elimination we may assume that $v\in \mathbb{Q}_+^{|S|}$, and so $t\cdot v\in \Z_+^{|S|}$ for some $t\in\N$ and $t\cdot v$ is orthogonal to $\bar A_S$. } 

Now, by definition of being a blocking, for each $i\geq k$ the set $\Dset^\bl_i$ is nonempty, so we can define $\bar \Dset^\bl_i=\{d-\bar d_i: d\in \Dset^\bl_i \}$ for some fixed $\bar d_i\in \Dset^\bl_i$. Since, if $\Zero\in \Dset$ then $\Zero\in \Dset^\bl_0$ by definition, we may define the (possibly empty) set $\bar \Dset^\bl_0:=\Dset^\bl_0$. We note that if $\rvec$ is a $(\bl,i)$-solution for $\Dset$ then $\rvec \bar\Dset^\bl=0$, where $\bar \Dset^\bl:=\bar\Dset^\bl_0\cup\cdots\cup \bar\Dset^\bl_k$.

Since $\Dset$ has no $\bl$-solution, there must be some $1\leq i\leq n$ for which there is no $(\bl, i)$-solution. However, since $\rvec$ induces $\bl$, Lemma~\ref{NewSubs} implies that $\rvec$ is a $(\bl, 1)$-solution, and so $i>1$. Therefore, for some $i>1$, there is no $R_i$-positive $v\in\R^{|S|}_+$ orthogonal to $M_S$, where $M={\bar \Dset^{\bl}}$ 
and $S={R^{+}_i}$. Since $\rvec$ induces a blocking, $\rvec$ is strictly positive, and since $\rvec$ is a $(\bl,1)$-solution for $\Dset$, $\rvec$ is orthogonal to ${\bar \Dset^{\bl}}$; since further $R_i \sbs R^+_i$, by Corollary~\ref{span1} we have that there exists $L \in \mathbb{N}$ such that $\rvec(t)\leq L\cdot \max\{ \rvec(x): x\in X\}$ for all $t\in R_i$, where $X:= \{1,...,n\}\setminus R_i^+=R_0\cup\ldots \cup R_{i-1}$.

Since $\rvec$ induces $\bl$ and since  $i>1$, we have $0<\rvec\Dset^\bl_{i-1}<  \rvec\Dset^\bl_{i}$. 
As $\rvec \Dset$ is contained in a shift of $K^\N$, say $K^\N-\const$ for some $\const\in \N$, there must be $a\geq 0$ and $b>0$ such that $\rvec\Dset^\bl_{i-1}=K^a-\const$ and $\rvec\Dset^\bl_{i}=K^{a+b}-\const$.
Observe that $\rvec(x)\leq K^a $ for all $x\in X$ since $\rvec\Dset^\bl_j\leq \rvec\Dset^\bl_{i-1}\leq K^a$ for each $j\leq i-1$ by definition of $\rvec$ inducing $\bl$. Hence $\rvec(t)\leq LK^a$ for all $t\in X\cup R_i$. 

For $\vd\in \Dset^\bl_i$ and $\vd'\in \Dset^\bl_{i-1}$, we have $\supp{\{\vd,\vd'\}}\subseteq X\cup R_i$, so
$$K^a(K-1) \leq K^a(K^b-1)= \rvec{\vd}-\rvec{\vd'} \leq  \rvec{|\vd-\vd'|}\leq L K^a\Delta\{\vd,\vd'\}\leq LK^a\Delta\Dset.$$
It follows that $K\leq L\Delta\Dset+1$.
\end{proof}

The lemma above completes the proof of Lemma~\ref{spineless2star} and hence Theorem~\ref{final}.

Now, if $\Gamma$ is a finite set of spineless simple equations then each equation in $\Gamma$ must be non-mingly by Lemma~\ref{nming}, and furthermore there must exist a smallest $K$ for which each equation in $\Gamma$ satisfies $({\star\star}K)$ as a consequence of Theorem~\ref{final}. Therefore, by Corollary~\ref{preUndLink}, we obtain:

\begin{cor}\label{UndLink}
For any finite set of spineless simple equations $\Gamma$, every subvariety of $\RL$ containing $\CRL+\Gamma$ has undecidable word problem.
\end{cor}

This completes the proof of Theorem~\ref{main}, and therefore also of Theorem~\ref{main2}.
%
%SECTION: Concluding Remarks
%
\section{Concluding Remarks}
First, we note that the quasiequations used to establish Theorem~\ref{Vmhard} are in the signature $\{\vee,\cdot,1\}$, so all complexity lower-bound/undecidability results hold even when restricting the word problem to the $\{\vee,\cdot,1\}$-fragments of such varieties. On the other hand, the equations used to establish Theorem~\ref{main2} make use of the full signature.
\begin{cor}
Let $\varepsilon$ be a spineless equation that is simple over $\RL$ and $\mathcal{V}$ a variety of residuated lattices containing $\CRL_\varepsilon$ as a subvariety. Then the word problem (and hence quasiequational theory) for the $\{\vee,\cdot,1\}$-fragment of $\mathcal{V}$ is undecidable. Furthermore, if $\varepsilon$ is expansive then the equational theory for $\mathcal{V}$ is undecidable.
\end{cor}
Given a simple equation $\varepsilon$, by $(\varepsilon)$ we denote its corresponding sequent-style inference rule, e.g., if $\varepsilon:xy\leq x^2yx\vee x\vee 1$, then
$$\infer[(\varepsilon).]{\Gamma,\Delta_1,\Delta_2,\Sigma\vdash\Pi}{\Gamma,\Delta_1,\Delta_1,\Delta_2,\Delta_1\Sigma\vdash\Pi & \Gamma,\Delta_1,\Sigma\vdash\Pi & \Gamma,\Sigma\vdash\Pi } $$
\begin{cor}
Let $\varepsilon$ be a spineless simple equation and $\m {L}$ any logic contained in the interval from $\m {FL}_\mathsf{e}+(\varepsilon)$ to $\m {FL}$. Then deducibility in the $\{\vee,\cdot,1\}$-fragment of $\m {L}$ is undecidable. Furthermore, if $\varepsilon$ is expansive then provability in the ($0$-free fragment of) $\m {L}$ is undecidable.
\end{cor}
%
%Subsection: Commutative varieties and single variable extensions
%
\subsection{Commutative varieties and single variable extensions}
In the following table we display decidability results for subvarieties of $\CRL$ axiomatized by $1$-variable equations using Lemma~\ref{und1var}. The numbers $n,p,q,...$ are distinct and positive, $m\geq0$, and furthermore are all given so that $\varepsilon$ is not trivial. By $(1~\vee)$ we mean $1$ may or may not be included in the expression.
\begin{center}
\begin{tabular}{ l | c | c  }
\hspace{5em}$\varepsilon$ &Eq. Th. of $\CRL_\varepsilon$&Quasi-Eq Th. of $\CRL_\varepsilon$\\ \hline
$x^n\leq x^m$ & FMP \cite{vanA} & FEP \cite{vanA}\\ \hline
$ x^n\leq x^{m}\vee 1$ & ?&?\\ \hline
$ x^n \leq  (1~{\vee})~x^p\vee x^q\vee\cdots  $ & ?&  Und. (Thm~\ref{main})\\ \hline
$ x^n\leq x^{n+p}\vee x^{n+q}\vee\cdots $ &Und. (Thm~\ref{main2}) & Und. (Thm~\ref{main}) 
\end{tabular}
\end{center}

We note that subvarieties axiomatized by equations of the form $x^n\leq x\vee 1$ have the finite model property. In fact, any subvariety of $\RL$ axiomatized by a simple equation in which each term on the right-hand side is linear (i.e., any variable occurs at most once in any joinand) has the finite model property (see Theorem~3.15 in \cite{GJ}). Similarly, while the subvariety of $\CRL$ axiomatized by the simple equation $\varepsilon: xyz\leq xy\vee xz\vee yz\vee x\vee y\vee z$ has an undecidable word problem (it is easily verified that $\varepsilon$ is spineless), $\CRL+\varepsilon$ has the FMP.

Subvarieties of $\CRL$ axiomatized by equations of the form $x^n\leq x^{m+1}\vee 1$, the simplest of which is $\Ds: x\leq x^2 \vee 1$, have no known decidability results for their (quasi-)equational theories. Focusing on the equation $\Ds$, we make the following observations:
\begin{itemize}
\item $\CRL_\Ds$ does not have the finite embeddability property. In fact, extensions of $\CRL$ by equations of the form $x^n y^m\leq x^{2n}\vee y^{2m}$ do not have the FEP for any choice $n,m$. This follows from the fact that such equations hold in chains and the FEP fails in such extensions of $\CRL$ (see \cite{HRT}).

\item The quasiequational theory of $\CRL_\Ds$ does not have a primitive recursive decision procedure. This can be shown using Theorem~\ref{Vmhard} and the machine constructed in \cite{Urq} (which shows that provability in $\m {FL}_\mathsf{ec}$, while decidable, is not primitive recursive). In fact, the same construction can be used to show that there is no primitive recursive decision procedure for the quasiequational theory of the subvariety of $\CRL$ axiomatized by $x^m\leq x^{m+n}~(\vee ~1)$ with $m,n>1$. Furthermore, by Corollary~\ref{undeq} the same holds for the equational theory of the subvariety of $\CRL$ axiomatized by $x^m\leq x^{m+n}$, as this equation is expansive. A more general treatment will be given in a forthcoming paper.
\end{itemize}
%
%Subsection: Non-commutative varieties
%
\subsection{Non-commutative varieties} 
As mentioned above, although the main reason for our study has been to establish undecidability for commutative varieties, we also get results about non-commutative ones. Here we compare that portion of our results with existing ones. In \cite{Hor}, it is shown that any variety of residuated lattices containing (as a subvariety) $\mathcal{H}=\RL+(x^3\leq x^2)+(x\leq x^2)$ has an undecidable word problem witnessed in its $\{\leq, \cdot, 1\}$-fragment; of course no subvariety of $\CRL$ contains $\mathcal{H}$, so this result has no implication about commutative varieties. In fact, the algebra $\mathbf{W}^+\in \mathcal{H}$ constructed in \cite{Hor} satisfies every equation for which the deletion of any collection of its variables results in either a trivial equation or one with the right-hand side containing a {\em square} subterm (i.e., a joinand of the from $uv^2w$ for $u,v,w\in X^*$ with $v\neq 1$.) As a result, even though Theorem~\ref{main} covers a lot of commutative varieties, it does not offer any new results for non-commutative varieties axiomatized by $1$-variable equations. In fact, the results in \cite{Hor} entail undecidability of the word problem for many non-commutative extensions of $\RL$ by prespinal equations; e.g., $\RL_\Ds$ has an undecidable word problem since $\mathbf{W}^+\models x\leq x^2\vee 1$. 

However, the are infinitely many equations in  two or more variables for which Theorem~\ref{main} is applicable while $\cite{Hor}$ is not. For example, any equation that is rendered trivial via commutativity (e.g., $x^2y\leq xyx$) is spineless and hence subvarieties of $\RL$ by such equations have an undecidable word problem by Theorem~\ref{main}. 

More interesting $3$-variable basic equations can be obtained by using {\em square-free} joinands.\footnote{A word $w\in X^*$ is square-free if $w\neq ux^2v$ for any words $u,v,x\in X^*$ and $x\neq 1$.} Let $X=\{x,y,z\}$ and let $h:X^*\to X^*$ be the homomorphism extending the assignment $h(1)=1$, $h(x)=xyz$, $h(y)=xz$, and $h(z)=y$. One can produce square-free words of arbitrary length by considering $h^0(x)=x$ and $h^{k+1}(x)=h(h^k(x))$ (see \cite{Lot}). For any nontrivial $1$-variable equation $\varepsilon:x^n \leq (1~{\vee})~x^p \vee x^q\vee \cdots$,  we denote by $h(\varepsilon)$ the basic equation:
$$ h^n(x) \leq  (1~{\vee})~ h^p(x)\vee h^q(x)\vee\cdots  .$$ 

Since $\varepsilon$ is nontrivial, $h(\varepsilon)$ is nontrivial and furthermore has square-free joinands. Consequently, $\mathcal{H}\nmodels h(\varepsilon)$ and so \cite{Hor} is not applicable to equations of this form. However, if $\varepsilon$ is a spineless $1$-variable basic equation then it can easily be shown that $h(\varepsilon)_\comeq$ is also spineless and therefore $\RL+h(\varepsilon)$ has an undecidable word problem by Theorem~\ref{main}. E.g., consider the equation $\varepsilon: x\leq x^2\vee x^3$, then 
$$h(\varepsilon):xyz\leq xyzxzy\vee xyzxzyxyzyxz \quad\mbox{and}\quad h(\varepsilon)_\comeq: xyz\leq x^2y^2z^2\vee x^4y^4z^4. $$
It is easily checked using Theorem~\ref{spinalform} that $h(\varepsilon)_\comeq$ is spineless and hence $h(\varepsilon)$ is spineless in view of Definition~\ref{def:spineless}. Therefore $\RL+h(\varepsilon)$ has an undecidable word problem by Theorem~\ref{main}, but $\mathcal{H}\nmodels h(\varepsilon)$.

While our undecidability results for the word problem in these varieties takes place in the $\{\vee,\cdot,1\}$-fragment, we can strengthen such results to the $\{\leq,\cdot,1\}$-fragment for many non-commutative extensions of residuated lattices (even by some prespinal equations) using a different encoding not relying on the $\vee$ operation. Such ideas will be explored in a forthcoming paper.

%Bibliography
\bibliography{FLe_Und_Ref}
\bibliographystyle{asl} 

\end{document}